\UseRawInputEncoding
\documentclass[11pt]{article}
\usepackage{amsfonts}
\usepackage{amsmath,amssymb,amsthm,latexsym}
\usepackage{amscd}
\usepackage[all]{xy}
\usepackage{hyperref}
\usepackage{mathrsfs}
\usepackage{fancyhdr}
\usepackage{indentfirst}
\usepackage{titlesec}
\usepackage{color}
\usepackage{subfigure}
\usepackage{tabularx}
\usepackage{multirow}
\usepackage{listings}
\usepackage{cite}
\usepackage{float}
\usepackage[all]{xy}
\usepackage{graphicx}
\usepackage{lineno}
\usepackage{geometry}
\newtheorem{theorem}{Theorem}[section]
\newtheorem{proposition}[theorem]{Proposition}
\newtheorem{definition}[theorem]{Definition}
\newtheorem{corollary}[theorem]{Corollary}
\newtheorem{lemma}[theorem]{Lemma}
\newtheorem{example}[theorem]{Example}
\newtheorem{problem}{Problem}
\newtheorem{newthe}[problem]{Theorem}
\newtheorem{remark}[theorem]{Remark}

\def\beq{\begin{equation}}
\def\eeq{\end{equation}}
\def\bspl{\begin{split}}
\def\espl{\end{split}}
\def\bgm{\begin{pmatrix}}
\def\edm{\end{pmatrix}}
\def\q{\theta}
\def\Q{\Theta}
\def\et{\eta}
\begin{document}
\title{Classification of Minimal Immersions of Conformally Flat 3-tori and 4-tori into Spheres by the First Eigenfunctions}
%Classification of Minimal Immersions by First Eigenfunctions for Conformally Flat  $3$-tori and $4$-tori in Spheres
\author {Ying L\"u\quad~~~ Peng Wang\quad~~~ Zhenxiao Xie}
\date{}
\maketitle

\begin{abstract}
This paper is devoted to the study of minimal immersions of flat $n$-tori into spheres, especially those immersed by the first eigenfunctions (such immersion is  called $\lambda_1$-minimal immersion), which also play important roles in spectral geometry. It is known that there are only two non-congruent $\lambda_1$-minimal $2$-tori %immersed by the first eigenfunctions 
in spheres, which are both flat. %: the Clifford $2$-torus in $\mathbb{S}^3$ and the equlaterial $2$-torus in $\mathbb{S}^5$, 
For higher dimensional case, the Clifford $n$-torus in $\mathbb{S}^{2n-1}$ might be the only known example in the literature. In this paper, by discussing the general construction of homogeneous minimal flat $n$-tori in spheres, we construct many new examples of $\lambda_1$-minimal flat $3$-tori and $4$-tori. In contrast to the rigidity in the case of  $2$-tori,  %known $\lambda_1$-minimal submanifolds %$\lambda_1$-minimal immersions of constantly curved $n$-spheres,  
%(such as $\lambda_1$-minimal $2$-tori), %and conformally flat $3$-tori), %of constant curvature and $\lambda_1$-minimal $2$-tori and $3$-tori, 
we show that there exists a $2$-parameter family of non-congruent $\lambda_1$-minimal flat $4$-tori. It turns out that the examples we constructed exhaust all $\lambda_1$-minimal immersions of conformally flat $3$-tori and $4$-tori in spheres. %by the first eigenfunctions. 
 %, which means that in general one can not expect a rigidity theorem even for $\lambda_1$-minimal submanifolds.  %immersed %submanifolds immersed
%by the first eigenfunctions. %, although such immersions  must be isometric to each other due to works of Montiel-Ros \cite{Mon-Ros} and El Soufi-Ilias  \cite{Sou-Ili}.
The classification involves some detailed investigations of shortest vectors in lattices, which  
\iffalse
In contrast to the rigidity of %$\lambda_1$-minimal immersions of constantly curved $n$-spheres,  
$2$-tori and conformally flat $3$-tori, %of constant curvature and $\lambda_1$-minimal $2$-tori and $3$-tori, 
we show that there exists a $2$-parameter family of non-congruent $\lambda_1$-minimal flat $4$-tori. 
\fi
%We also discuss the general construction of homogeneous minimal flat $n$-tori in spheres briefly. 
%Two kinds
%In addition, the approach we used 
can also be used to solve the Berger's problem on flat $3$-tori and $4$-tori. The dilation-invariant functional $\lambda_1(g)V(g)^{\frac{2}{n}}$
%有没有更通用的名称，在摘要里引用公式不太合适吧？
about the first eignvalue is proved to have maximal value among all flat $3$-tori and $4$-tori. % A class of $\lambda_1$-minimal flat $n$-tori %immersed by the first eigenfunctions 
%are also presented.  
\end{abstract}

\indent{\bf Keywords:} Minimal immersions; First eigenvalue; Conformally flat tori; Lattices; Berger's problem

\indent{\bf MSC(2000):\hspace{2mm} 53A30, 53B25, 53B30}

\vspace{4mm}

\section{Introduction}
%In spectral geometry, Berger's isoperimetric problem
The study of minimal immersions of Riemannian manifolds into spheres is an interesting topic in differential geometry. It builds a deep link between the spectral theory and minimal submanifold theory. The famous theorem of Takahashi \cite{Taka} states that the isometric immersion $x:(M^n,g)\rightarrow \mathbb{S}^m$ is minimal, if and only if the coordinate functions of $x$ are eigenfunctions of the Laplacian %of the Riemannian manifold $(M^m,g)$, 
with respect to the eigenvalue $n$.

Given a closed Riemannian manifold $(M^n, g)$, we denote the volume and the first eigenvalues of $(M^n, g)$ respectively by $V(g)$ and $\lambda_1(g)$. The dilation-invariant functional
%\beq\label{eq-lambda1}
$$\mathcal{L}_1(g)\triangleq \lambda_1(g)V(g)^{\frac{2}{n}}$$
%\eeq
on the set of all smooth Riemannian metrics is a basic functional considered in spectral theory. %On a given closed Riemannian manifold $M^m$,
%Fix a closed manifold $M^m$,
It shows in \cite{Sou-Ili} that the critical metric of $\mathcal{L}_1(g)$ among % the set of 
all smooth Riemannian metrics on  $M^n$ (called \emph{$\lambda_1$-critical metric}) admits an isometric minimal immersion of $M^n$ in spheres by the first eigenfunctions (called \emph{$\lambda_1$-minimal immersion}). %We will recall some well-known related results.

Moreover, it was proved by Hersch \cite{Hersch} that on the topological $2$-sphere, among all smooth Riemannian metrics, $8\pi$ is the maximal value of $\mathcal{L}_1(g)$, which can only be attained by the standard metric. In 1973, Berger \cite{Berger} obtained the maximal value %an upper bound
of $\mathcal{L}_1(g)$ on the topological $2$-torus among all flat metrics. This value is attained by the $\lambda_1$-minimal immersion of the equilateral $2$-torus in $\mathbb{S}^5$. %, and equipped with the induce metric.
Since then, %on a given closed manifold $M^m$, 
finding the uniform upper bound of $\mathcal{L}_1(g)$ among all smooth metrics is referred as {\em Berger's %isoperimetric
problem}. By introducing the conformal volume, Li and Yau \cite{Li-Yau} solved the case of $\mathbb{R}P^2$ in 1982, whose upper bound is attained by the  $\lambda_1$-minimal immersion of $\mathbb{R}P^2$ in $\mathbb{S}^4$ (i.e., Veronese surface). Due to also the work of El Soufi and Ilias \cite{Sou-Ili}, on any closed manifold of dimension $n$, Li-Yau's conformal volume can be used to provide an upper bound for $\mathcal{L}_1(g)$ in the  %given  
conformal class $[g]$. In the mean time, they also show that for the conformal manifold $(M^n, [g])$ admitting a $\lambda_1$-minimal immersion in spheres, the volume of such $\lambda_1$-minimal immersion is exactly equal to the conformal volume. 
Therefore it realizes the supremum for $\mathcal{L}_1(g)$ in $[g]$.  %prove to be  by that of the $\lambda_1$-minimal metric %, which means when equipped this metric, $M^m$ can be isometrically and minimally immersed into a sphere by the first eigenfunctions. In this paper, such immersion will be called \emph {$\lambda_1$-minimal immersion} for short.
Furthermore, for a given conformal manifold, Montiel-Ros \cite{Mon-Ros} and El Soufi-Ilias \cite{Sou-Ili} proved there exists at most one $\lambda_1$-minimal metric in the conformal class.  %which admits a $\lambda_1$-minimal immersion into some sphere.
Combining this with the existence of $\lambda_1$-maximal metric, Nadirashvili finally solved the Berger's problem completely for the topological $2$-torus in \cite{Nadi}.  We point out that for the case of Klein bottle, the Berger's problem has also been solved due to the work of Nadirashvili \cite{Nadi}, Jakobson-Nadirashvili-Polterovich \cite{Nadi2} and El Soufi-Giacomini-Jazar \cite{Soufi2}, where one $\lambda_1$-minimal immersion of the  Klein bottle is presented (see \cite{Hirsch-Mader}, \cite{Nadi2}). For surfaces of higher genus, we refer to \cite{Petrides, Karpukhin, Matthiesen-Siffert} and references therein. %the work if Petrides in  %It was proved by Nadirashvili in 1996 that on the $2$-torus and Klein bottle, there exists a metric (called $\lambda_1$-maximal metric) by which the maximal value of $\lambda_1(g)V(M^m,g)^{\frac{2}{m}}$ (among all smooth Riemannian metrics) is attained. Moreover he proved that $\lambda_1$-maximal metric can provides a $\lambda_1$-minimal immersion, from which Nadirashvili \cite{Nadi} solved the Berger's problem completely for $2$-torus. For the case of Klein bottle, Jakobson-Nadirashvili-Polterovich \cite{Nadi2} and Soufi-Giacomini-Jazar \cite{Soufi2} finally solved this problem in 2006.
In contrast to %these positive solutions on
the dimension $2$, %The Berger's problem has also been discussed on manifolds of dimension no more than $3$.
%it was proved by 
Colbois and Dodziuk proved in \cite{Colbois-Dodziuk} that  there is no uniform bound for the functional $\mathcal{L}_1(g)$ on any closed manifold of dimension $n\geq 3$. %It was proved by Colbois and Dodziuk \cite{Colbois-Dodziuk} that on any closed manifold of dimension $m\geq 3$, there is no uniform bound for the functional $\mathcal{L}_1(g)$.
This implies that one might %need to 
consider the Berger's problem restricting to a given conformal class (see % two recent papers 
\cite{Sou-Ili3, Karpukhin-Stern, Petrides2} and reference therein), for which the investigation of $\lambda_1$-minimal immersions of higher dimensional manifolds into spheres plays an important role. % It follows from the theory of conformal volume that the investigation of $\lambda_1$-minimal immersions of higher dimensional manifolds into spheres plays an important role in the study of Berger's problem.  % could provide help for this problem. %%Instead, Soufi and Ilias \cite{Soufi-Ilias} studied the critical point of this functional, which is called $\lambda_1$-critical point. They proved that $\lambda_1$-critical is equivalent to $\lambda_1$-minimal.

In the literature, there have been several known classes of $\lambda_1$-minimal submanifolds in spheres. A famous conjecture of Yau states that any closed embedded minimal hypersurface in $\mathbb{S}^{n+1}$ is $\lambda_1$-minimal. Due to the work of Muto-Onita-Urakawa \cite{Muto-Onita-Urakawa} and Tang-Yan \cite{Tang-Yan} on this conjecture, we know all isoparametric hypersurfaces and some ones of their focal submanifolds form a class of $\lambda_1$-minimal submanifolds in spheres. Another class of examples is due to Takahashi
%The famous theorem \cite{Taka} of Takahashi says that a Riemannian manifold $(M^n,g)$ can be immersed minimally into a sphere  %$\mathbb{S}^m$ by $x$, if and only if, the coordinates of $x$ are eigenfunctions corresponding to the eigenvalue $n$.
\cite{Taka}, who proved that for any positive integer $k$, up to a dilation of the metric, any compact irreducible homogeneous Riemannian manifold can be immersed minimally into a certain sphere by the $k$-th eigenfunctions (we call it \emph{$\lambda_k$-minimal immersion} for short). Later, the case of sphere equipped with the constantly curved metric was investigated in detail by Do Carmo and Wallach \cite{doCarmo-Wallach}. %especially   space form Later, DoCarmo and Wallach in \cite{doCarmo-Wallach},
It was proved that when $n\geq 3$, the linearly full $\lambda_k$-minimal immersion of $n$-sphere has rigidity if and only if $k\leq 3$. Moreover, they also proved that the immersion will span the full $k$-eigenspace when $k\leq3$.
%Howere both these property   in  Moreover, when $k\geq 4$, it was proved by them that $m+1$ may be less than the dimension of $s$-th eigenspace.
In this paper, we will show that these two properties do not hold for minimal flat tori of dimension $4$, even for the case of $\lambda_1$-minimal immersion. To be precise, we construct a $\lambda_1$-minimal flat $4$-torus in $\mathbb{S}^{11}$, which has rigidity but does not span the whole  eigenspace (see Example~\ref{ex:7-1}).
Furthermore, a $2$-parameter family of non-congruent $\lambda_1$-minimal flat $4$-torus in $\mathbb{S}^{23}$ is also constructed, among which there is a $1$-parameter family living in $\mathbb{S}^{15}$, neither rigid nor fully-spanning the eigenspace (see Example~\ref{ex-4tori}, Proposition~\ref{prop:except} and Remark~\ref{rk-1para}). 
% and a double covering of the Clifford $4$-torus in $\mathbb{S}^7$. 

Due to the work of Kenmotsu \cite{Kenmotsu} and Bryant \cite{Bryant}, we know all minimal flat $2$-tori in spheres are homogeneous. Let $\Lambda_n$ denote a  lattice of rank $n$. In \cite{Bryant}, Bryant proved that a flat torus $T^2=\mathbb{R}/\Lambda_2$ admits minimal immersions in spheres, if and only if  the Gram matrix of $\Lambda_2$ is rational (i.e., all entries are rational numbers) up to some dilation. This implies there are infinite non-congruent minimal flat $2$-tori in spheres. But among them, there are only two $\lambda_1$-minimal ones: %They are
the Clifford $2$-torus in $\mathbb{S}^3$, and the equilateral $2$-torus in $\mathbb{S}^5$. This classification is due to the work of Montiel-Ros \cite{Mon-Ros} and El Soufi-Ilias \cite{Sou-Ili}.

In contrast to the plentiful results on dimension $2$ in the literature, minimal flat tori of higher dimension haven't been investigated so much, especially for those $\lambda_1$-minimal ones. As far as we know, the Clifford $n$-torus in $\mathbb{S}^{2n-1}$ is the only known $\lambda_1$-minimal example (see \cite{Park-Urakawa}). In this paper, we construct
five non-congruent $\lambda_1$-minimal flat $3$-tori, a $2$-parameter family and another sixteen non-congruent $\lambda_1$-minimal flat $4$-tori in spheres. %They are listed in Table~\ref{tab:my_tabel3} and Table~\ref{tab:my_tabel}.    %a $2$-parameter family
%many examples of $\lambda_1$-minimal flat $3$-tori and $4$-tori in spheres,
%as well as some examples for general dimension $n$.
It turns out that these examples exhaust all non-congruent $\lambda_1$-minimal immersions of conformally flat $3$-tori and $4$-tori into spheres, see Theorem~\ref{thm:3}  %Table~\ref{tab:my_tabel3}, 
and Theorem~\ref{thm-classify4}. The classifications are listed in the following two tables (for the definition of irreducible and reducible see Section~\ref{sec-examples}). %Table~\ref{tab:my_tabel3} and Table~\ref{tab:my_tabel}. %and Table~\ref{tab:my_tabel}.
%Two kinds of examples for general dimension $n$ are also presented. 
%For readers' convenience, we list all non-congruent $\lambda_1$-minimal and conformally flat $3$-tori and $4$-tori in the following tables.
\begin{table}[H]
\renewcommand{\arraystretch}{2}
   \centering
\begin{tabular}{|c | c | c | c | c | c|}
%\specialrule{0.04em}{3pt}{3pt}
\hline
$\mathbb{S}^n$&Total Numbers&Reducible&Irreducible &Examples & $\lambda_1(g)V(g)^{\frac{2}{n}}$\\%Volume\\
\hline
$\mathbb{S}^5$ &1 &1 &0 &E.g.~\ref{ex:3-prod} &$4 \pi ^2$\\
\hline
$\mathbb{S}^7$ &2  &1 &1 &E.g.~\ref{ex:3-prod},~\ref{ex:3-1} & $\frac{4 \sqrt[3]{4} }{\sqrt[3]{3}}\pi ^2$, $3 \sqrt[3]{4} \pi ^2$\\
\hline
%\specialrule{0.04em}{3pt}{3pt}
$\mathbb{S}^{9}$ &1 &0 &1 &E.g.~\ref{ex:3-2} & $\frac{8 \sqrt[3]{2} }{\sqrt[3]{9}}\pi ^2$\\
\hline
$\mathbb{S}^{11}$ &1 &0 &1  &E.g.~\ref{ex:3-3}  &$4 \sqrt[3]{2} \pi ^2$ \\
\hline
\end{tabular}
   \caption{The classification of $\lambda_1$-minimal immersions of confromally flat $3$-tori }
    \label{tab:my_tabel3}
\end{table}
\begin{table}[H]{\small
\renewcommand{\arraystretch}{2}
   \centering
\begin{tabular}{|c | c | c | c | c | c|}
%\specialrule{0.04em}{3pt}{3pt}
\hline
$\mathbb{S}^n$&Total Numbers&Reducible&Irreducible&Examples &$\lambda_1(g)V(g)^{\frac{2}{n}}$\\
\hline
$\mathbb{S}^7$ &2 &1 &1 &E.g.~\ref{ex:4-prod}, E.g.~\ref{ex-4tori}&$4 \pi ^2,4 \sqrt{2} \pi ^2$  \\
\hline
$\mathbb{S}^9$ &3  &2 &1  &E.g.~\ref{ex:4-prod},~\ref{ex:5-1}& $\frac{4 \sqrt{2} }{\sqrt[4]{3}}\pi ^2$, $2 \sqrt[4]{27} \pi ^2$,  $\frac{16}{\sqrt[4]{125}}\pi ^2$ \\
\hline
%\specialrule{0.04em}{3pt}{3pt}
$\mathbb{S}^{11}$ &4 &2 &2 &E.g.~\ref{ex:4-prod},~\ref{ex:6-1},~\ref{ex:7-1} &$\frac{8 }{\sqrt{3}}\pi ^2$,~$\frac{8}{\sqrt{3}}\pi ^2$, $2 \sqrt{6} \pi ^2$, $\frac{8 \sqrt{2} }{\sqrt[4]{27}}\pi ^2$\\
\hline
$\mathbb{S}^{13}$ &3 &1 &2 &E.g.~\ref{ex:4-prod}
%,~Example~\ref{ex:7-2},
%$\sim$~E.g.
,~\ref{ex:7-2},~\ref{ex:7-3}&$4 \sqrt[4]{2} \pi ^2$, \!$\frac{8 \sqrt{2} }{\sqrt[4]{27}}\pi ^2$,\! $\frac{4 \sqrt[4]{26 \sqrt{13}-70} }{\sqrt{3}}\pi^2$ \\
\hline
%\specialrule{0.04em}{3pt}{3pt}
$\mathbb{S}^{15}$ &$1$-family \& 2 &0 & $1$-family \& 2 &E.g.~\ref{ex-4tori},~\ref{ex:8-1},~\ref{ex:8-2}&$4 \sqrt{2} \pi ^2$, $\frac{4 \sqrt[4]{8}}{\sqrt[4]{3}}\pi^2$, $\frac{4 \sqrt[4]{8 \sqrt{3}+12}}{\sqrt{3}}\pi^2$\\
\hline
%\specialrule{0.04em}{3pt}{3pt}
$\mathbb{S}^{17}$ &2 &0 &2 &E.g.~\ref{ex:9-1},~\ref{ex:9-2}&$\frac{16 }{3}\pi ^2$, $4 \sqrt[4]{3} \pi ^2$\\
\hline
%\specialrule{0.04em}{3pt}{3pt}
$\mathbb{S}^{19}$ &1 &0 &1 &E.g.~\ref{ex:10}&$\frac{8}{\sqrt[4]{5}} \pi ^2$\\
\hline
%\specialrule{0.04em}{3pt}{3pt}
$\mathbb{S}^{23}$ &$2$-family &0 &$2$-family &E.g.~\ref{ex-4tori}&$4 \sqrt{2} \pi ^2$\\
\hline
\end{tabular}
    \caption{The classification of $\lambda_1$-minimal immersions of confromally flat $4$-tori %\textcolor{red} {the first one E.g. 4.6 (E.g. 1.1)} 
    }
    \label{tab:my_tabel}
    }
\end{table}
%The following examples are the $2$-family of $\lambda_1$-minimal $4$-tori. 
\iffalse
\noindent The $2$-parameter family of non-congruent $\lambda_1$-minimal $4$-torus involves the flat torus $$T^4=\mathbb{R}^4/\mathrm{Span}_{\mathbb{Z}}\{e_1-e_4, e_2-e_4, e_3-e_4, 2e_4\},$$
with the immersions given by 
\[
\begin{aligned}
\!\!\Big(a_1e^{i\pi(u_1+u_2+u_3+u_4)},a_1e^{i\pi(u_1+u_2-u_3-u_4)},a_1e^{i\pi(u_1-u_2+u_3-u_4)},a_1&e^{i\pi(-u_1+u_2+u_3-u_4)},\\
a_2e^{i\pi(u_1+u_2+u_3-u_4)},a_2e^{i\pi(u_1+u_2-u_3+u_4)},a_2e&^{i\pi(u_1-u_2+u_3+u_4)}, a_2e^{i\pi(u_1-u_2-u_3-u_4)},\\ &~a_3e^{2i\pi u_1},a_3e^{2i\pi u_2},a_3e^{2i\pi u_3},a_3e^{2i\pi u_4}\Big),
\end{aligned}
\]
where $0\leq a_1\leq a_2\leq a_3$ and $a_1^2+a_2^2+a_3^2=1$.
\fi
\noindent In the above table, 
%Among theses examples, 
only Example~\ref{ex-4tori} describes a continuous family of non-rigid $\lambda_1$-minimal $4$-tori as mentioned before.    

\begin{example}\label{ex-4tori} Denote by $\{e_i\}$ the standard basis of $\mathbb{R}^4$. The flat $4$-torus $$T^4=\mathbb{R}^4/\mathrm{Span}_{\mathbb{Z}}\{e_1-e_4, e_2-e_4, e_3-e_4, 2e_4\}$$
admits a 
\iffalse
    For the $4$-torus $\mathbb{R}^4/\Lambda_4$ with 
    \begin{equation} 
\Lambda_4^*=\left (
\begin{array}{rrrrrrrrrrrr}
     1&0&0&{1}/{2}   \\
     0&1&0&{1}/{2} \\
     0&0&1&{1}/{2} \\
     0&0&0&{1}/{2} 
\end{array}
\right ),
\end{equation}
there is a 
\fi
$2$-parameter family of non-congruent $\lambda_1$-minimal %isometric 
immersions in $\mathbb{S}^{23}$ % up to congruence, 
given as follows:
\begin{equation*}
\begin{aligned}
\!\!\Big(a_1e^{i\pi(u_1+u_2+u_3+u_4)},a_1e^{i\pi(u_1+u_2-u_3-u_4)},a_1e^{i\pi(u_1-u_2+u_3-u_4)},a_1&e^{i\pi(-u_1+u_2+u_3-u_4)},\\
a_2e^{i\pi(u_1+u_2+u_3-u_4)},a_2e^{i\pi(u_1+u_2-u_3+u_4)},a_2e&^{i\pi(u_1-u_2+u_3+u_4)}, a_2e^{i\pi(u_1-u_2-u_3-u_4)},\\ &~a_3e^{2i\pi u_1},a_3e^{2i\pi u_2},a_3e^{2i\pi u_3},a_3e^{2i\pi u_4}\Big),
\end{aligned}
\end{equation*}
where $0\leq a_1\leq a_2\leq a_3$ and $a_1^2+a_2^2+a_3^2=\frac{1}{4}$. See Section \ref{sec-4} for more details.
\end{example}

Our construction depends on the variational characterizations (see Theorem~\ref{thm-variation} and Theorem~\ref{thm-vari}) we obtained for homogeneous minimal flat tori in spheres. Roughly speaking,  the construction of a homogeneous minimal flat $n$-torus in some sphere is equivalent to finding a $2$-tuple $\{Y, Q\}$ (we call it {\em matrix data}) satisfying some constrains, where $Q^{-1}$ is a Gram matrix of the lattice corresponding to this torus, and $Y$ is a set of finite integer vectors in $\mathbb{Z}^n$ describing the linear relations between lattice vectors involved in the minimal immersion.   %$Y$ is a set of finite integer vectors in $\mathbb{Z}^n$ describing the linear relations between lattice vectors corresponding to the $k$-th eigenfunctions, and $Q$ is the Gram matrix of some generator of lattice for such torus. The main constrain is that $Q$ can determine a hyper-ellipsoid, which has the least volume among all hyper-ellipsoids passing through those finite integer vectors in $Y$.   
\iffalse
the construction of homogeneous minimal flat tori in spheres is equivalent to %can be transformed into
finding the sets of finite integer vectors satisfying some constrains. The main constrain is that there should be at least one hyper-ellipsoid passing through these finite integer vectors. 
\fi
%We point out that
Theoretically, all homogeneous minimal flat tori in spheres can be constructed by the approaches we provided (see subsection~\ref{subsec-vara}). To construct $\lambda_1$-minimal flat $n$-tori, we also need to deal with the problem of finding all shortest vectors in a lattice, which is important but difficult in the theory of lattice (or geometry of numbers). Fortunately, a result of Ryshkov \cite{Ryshkov73} proved in Minkowski's reduction theory can be used in our construction to overcome this obstacle.

To classify all $\lambda_1$-minimal immersions of conformally flat $3$-tori and $4$-tori in spheres, it follows from the work of El Soufi and Ilias \cite{Sou-Ili} that we only need to classify all $\lambda_1$-minimal immersions of flat $3$-tori and $4$-tori in spheres. It turns out that they are all homogeneous  (in fact, a sufficient condition is given in Proposition~\ref{prop-homo} for general minimal flat tori in spheres to be homogeneous). %(see Proposition~\ref{prop-homo}, where a sufficient condition is given for general minimal flat tori in spheres to be homogeneous).
 Note that the moduli space of flat tori (modulo isometry) is $SL(n,\mathbb{Z})\setminus GL(n,\mathbb{R})\,/\,O(n)$ (see \cite{Wolpert}, or Section~\ref{sec3}). %This is a space of very high dimension and
 The action of $SL(n,\mathbb{Z})$ makes the classification highly nontrivial. % for which, some invariants are introduced %to those lattices considered 
%by us at first (see Section~\ref{sec-lattice}). %To overcome this, 
However, our variational characterization suggests that all we need to do is to find out all the possible integer sets $Y$, where $Q$ is uniquely determined if it exists. 
To do this, a coarse classification to lattices of rank no more than $4$ is given at first (see Theorem~\ref{thm-gen}), from which some  necessary constrains on $Y$ can be obtained. Then after introducing some invariants to the set of shortest lattice vectors, we can determine all the possibilities of $Y$ up to the action of $SL(n,\mathbb{Z})$.  
%The necessary conditions on $Y$ are studied in Section~\ref{sec-lattice} and a coarse classification to lattices of rank no more than $4$ is given in Theorem~\ref{thm-gen}. 
%Therefore, the possibilities of those integer sets are reduced significantly  down to finite, where our classification can be achieved.%contain just integer vectors whose coordinates 
%taking values in $\{0,\pm1\}$. %Then to simplify our classification, some invariants are introduced to those lattices considered by us. %Even though,  the classification is still highly nontrivial when $n=4$, since the action of $SL(n,\mathbb{Z})$  we only need to classify those possible integer sets. % of finite integer vectors.
\iffalse
Note that the moduli space of flat tori (modulo isometry) is given by $SL(n,\mathbb{Z})\setminus GL(n,\mathbb{R})\,/\,O(n)$ (see \cite{Wolpert}, or Section~\ref{sec3}). The action of $SL(n,\mathbb{Z})$ makes the accurate classification highly nontrivial, for which, some invariants are introduced to those lattices considered by us.
\fi

%\vspace{2mm}
The volume of these $\lambda_1$-minimal flat tori we construct are also  calculated (see Section~\ref{sec-examples}). It follows from the theory of conformal volume that  these $\lambda_1$-minimal metrics  maximize the functional $\lambda_1(g)V(g)^{\frac{2}{n}}$  in their respective conformal classes. %that the functional $\lambda_1(g)V(g)^{\frac{2}{n}}$ can be bounded from above by these volume %of the $\lambda_1$-minimal metric
%in the conformal classes determined by these $\lambda_1$-minimal metrics (see Table~\ref{tab:my_tabel3} and Table~\ref{tab:my_tabel}). %determined.
Among all $\lambda_1$-minimal flat $3$-tori, the maximal value of $\lambda_1(g)V(g)^{\frac{2}{n}}$ is $4 \sqrt[3]{2} \pi ^2$,  %these volume is $8\sqrt{6}\pi^3/9$, 
and it is $4 \sqrt[2]{2} \pi ^2$ among all  $\lambda_1$-minimal flat $4$-tori (see Table~\ref{tab:my_tabel3} and Table~\ref{tab:my_tabel}). %Fortunately, 
In fact, using the investigation about lattices given in Section~\ref{sec-lattice} and Section~\ref{sec-4}, we can prove the following theorem, which can be seen as a generalization of Berger's result %in \cite{Berger} 
from flat $2$-tori to flat $3$ and $4$-tori. 

\begin{newthe}\label{thm1}
%Let $\lambda_1(g)V(g)^{\frac{2}{n}}$ be the dilation-invariant functional defined in \eqref{eq-lambda1}.
Consider the dilation-invariant functional $\lambda_1(g)V(g)^{\frac{2}{n}}$ on the topological $n$-torus. %we have 

%\noindent 
(1) When $n=3$, among all flat metrics, %$3$-tori,
$$\lambda_1(g)V(g)^{\frac{2}{n}}\leq 4 \sqrt[3]{2} \pi ^2,$$
and the equality is attained by the $\lambda_1$-minimal flat $3$-torus given in Example~\ref{ex:3-3}. 

%\noindent 
(2)  When $n=4$, among all flat metrics, %Among all flat $4$-tori,  
$$\lambda_1(g)V(g)^{\frac{2}{n}}\leq 4 \sqrt{2} \pi ^2,$$
and the equality is attained by those $\lambda_1$-minimal flat $4$-tori given in Example~\ref{ex-4tori}. 
\end{newthe}

Inspired by this theorem, it is natural to consider the Berger's problem on conformally flat $3$-tori and $4$-tori: whether $4 \sqrt[3]{2} \pi ^2$ and $4 \sqrt{2} \pi ^2$ are respectively the upper bounds of $\lambda_1(g)V(g)^{\frac{2}{n}}$ on the topological $3$-torus and $4$-torus among all smooth conformally flat metrics.  

%Combining this theorem with 
%Note that 
In \cite{Sou-Ili3}, El Soufi and Ilias exhibited a class of flat $n$-tori for which the endowed flat metric maximizes $\lambda_1(g)V(g)^{\frac{2}{n}}$ on its conformal class (see Corollary 3.1 in their paper). Combining their result with our work (Theorem~1 and Theorem~\ref{thm-gen}), we can partially  solve the above problem.  
\begin{newthe}\label{thm2} 
Suppose $g$ is a smooth Riemannian metric on the the topological $n$-torus. If $g$ is conformal equivalent to a flat metric whose first eigenspace is of dimension no less than $2n$, then    
$\lambda_1(g)V(g)^{\frac{2}{n}}\leq 4 \sqrt[3]{2} \pi ^2$ when  $n=3$,
and $\lambda_1(g)V(g)^{\frac{2}{n}}\leq 4 \sqrt{\pi ^2}$
when $n=4$. 
%Suppose $(T^n,g)$ is a flat $n$-torus with the dimension of its first eigenspace at least $2n$. Then for all smooth metrics $h\in[g]$, $\mathcal{L}_1(h)\leq 4 \sqrt[3]{2} \pi ^2 $ when $n=3$ and
%$$\mathcal{L}_1(h)\leq 4 \sqrt[3]{2} %\pi^2, $$
%$\mathcal{L}_1(h)\leq 4 \sqrt{2} \pi ^2$   
%$$\mathcal{L}_1(h)\leq 4 \sqrt{2} \pi ^2$$ 
%when $n=4$. 
\iffalse
(1) Among all flat $3$-tori,  $$\lambda_1(g)V(g)^{\frac{2}{n}}\leq 4 \sqrt[3]{2} \pi ^2,$$
and the equality is attained by the $\lambda_1$-minimal flat $3$-torus given in Example~\ref{ex:3-3}. 

(2) Among all flat $4$-tori,  
$$\lambda_1(g)V(g)^{\frac{2}{n}}\leq 4 \sqrt{2} \pi ^2,$$
and the equality is attained by those $\lambda_1$-minimal flat $4$-tori given in Example~\ref{ex-4tori}.
\fi
\end{newthe}
%It is natural to ask whether $8\sqrt{6}\pi^3/9$ (w.r.t. $2\pi^4$) is also the upper bound of $\lambda_1(g)V(g)^{\frac{2}{n}}$ among all conformally flat metrics on $3$-torus (w.r.t. $4$-torus).

The paper is organized as follows. In Section~\ref{sec3}, we firstly recall the  basic spectral theory of flat tori,  and then discuss the homogeneity of minimal flat tori in spheres. Section~\ref{sec-variation} is devoted to presenting our basic setup on homogeneous minimal flat tori, as well as the variational characterizations obtained for them. New examples of $\lambda_1$-minimal flat $3$-tori and $4$-tori are constructed in Section~\ref{sec-examples}. We devote Section~\ref{sec-lattice} to investigate the shortest vectors of lattices, where a coarse classification to lattices of rank no more than $4$ is given. The classification of $\lambda_1$-minimal immersions of conformally flat $3$-tori and $4$-tori are obtained in Section~\ref{sec-4}. %By presenting a 
A class of $\lambda_1$-minimal flat $n$-tori is presented in Section~\ref{sec-high} % we show the applicability 
as an application of our construction method in higher dimensions. Finally, Section~\ref{sec-Berge} is devoted to discuss Berger's problem on conformally flat $3$-tori and $4$-tori, where Theorem \ref{thm1} and Theorem \ref{thm2} are proved.

\section{On isometric minimal immersions of %conformally
flat tori}\label{sec3}
In this section, we will firstly recall the basic theory of flat tori. Then a sufficient condition for minimal flat tori in spheres to be homogeneous will be given.
\vspace{3mm}
\subsection{Flat tori and lattices}
It is well known that a flat torus $T^n$ of dimension $n$ can be described as
$$T^n=\mathbb{R}^n/{\Lambda_n},$$
where $\Lambda_n$ is a lattice of rank $n$ on $\mathbb{R}^n$. Set $L_n$ to be the generator matrix of $\Lambda_n$, which means $\Lambda_n$ can be generated by row vectors of $L_n$. Two tori $T^n=\mathbb{R}^n/{\Lambda_n}$ and $\widetilde{T}^n=\mathbb{R}^n/{\widetilde{\Lambda}_n}$ are isometric if and only if $\Lambda_n$ and $\widetilde{\Lambda}_n$ are isometric, i.e., there exists an orthogonal matrix $O$ and an unimodular matrix $U\in SL(n,\mathbb{Z})$, such that $L_n=U\,\widetilde{L}_n\, O$, where $L_n$  (w.r.t. $\widetilde{L}_n$) is a generator matrix %$n\times n$ matrix whose row vectors are given by a basis
of lattice $\Lambda_n$ (w.r.t. $\widetilde{\Lambda}_n$).
It follows that the moduli space of flat $n$-tori  is
$$SL(n,\mathbb{Z})\setminus GL(n,\mathbb{R})\,/\,O(n).$$

The dual lattice of $\Lambda_n$ is defined to be a lattice $\Lambda_n^{*}$, whose generator matrix ${L}_n^*$ satisfies
${L}_n({L}_n^*)^t=I_n.$
%It is highly related to 
The spectrum of $T^n=\mathbb{R}^n/{\Lambda_n}$ is %, i.e., we have
$$\mathrm{Spec}(T^n)=\Big{\{}4\pi^2|\xi|^2\,\Big{|}\,\xi\in \Lambda_n^*\Big{\}},$$
and $e^{2\pi\langle\xi,u \rangle i}$ is an eigenfunction corresponding to the eigenvalue $4\pi^2|\xi|^2$, where
$u=(u_1,u_2,\cdots,u_n)$
is the %Cartesian
coordinates of $\mathbb{R}^n$, such that the flat metric on $\mathbb{R}^n$ ($T^n$) can be written as $du_1^2+d u_2^2+\cdots+d u_n^2$. 
\subsection{Minimal homogeneous flat tori in spheres}
%\subsection{Minimal immersions of flat tori in spheres}
%The famous theorem \cite{Taka} of Takahashi says that a Riemannian manifold $(M^n,g)$ can be immersed minimally into a sphere  %$\mathbb{S}^m$ by $X$, if and only if, the coordinates of $X$ are eigenfunctions corresponding to the eigenvalue $n$.

Let $T^n=\mathbb{R}^n/{\Lambda_n}$ be a flat torus. Assume $n$ is an eigenvalue of this torus, whose eigenspace is of dimension $2N$. It follows that there are exactly $N$ distinct lattice vectors  (up to $\pm1$) having the length $\frac{n}{4\pi^2}$ in the dual lattice $\Lambda_n^*$, which  are denoted by
$$\xi_1,\xi_2,\cdots,\xi_N.$$
%$X:T^n=\mathbb{R}^n/{\Lambda_n} \longrightarrow  \mathbb{S}^p$ be a linearly full minimal immersion. Assume
%$$\{\pm\xi_1,\pm\xi_2,\cdots\pm\xi_N\}$$
%enumerates all points in $\Lambda_n^*$ with length $n$.
By the theorem of Takahashi, any minimal isometric immersion of $T^n$ in spheres can be expressed as follows:
\beq\label{eq-X}
x=\bgm\Q_1&\Q_2&\cdots&\Q_N\edm A:T^n=\mathbb{R}^n/{\Lambda_n} \longrightarrow  \mathbb{S}^{2N-1},%~~~X=\bgm\Q_1&\Q_2&\cdots&\Q_N\edm A,
\eeq
where $\Q_r=\bgm\cos\q_r&\sin\q_r\edm,$ $\q_r=、2\pi\langle\xi_r,u\rangle$ for $~1\leq r\leq N$, and $A$ is a $2N\times 2N$ matrix. %Without loss of generality, we can assume $p=2N-1$ and
Write
$$AA^t=\bgm A_{11}& A_{12}&\cdots& A_{1N}\\
 A_{21}& A_{22}&\cdots& A_{2N}\\
\vdots&\vdots&\ddots&\vdots\\
 A_{N1}& A_{N2}&\cdots& A_{NN}\edm,~~~ A_{rs}=\bgm A_{rs}^{11}& A_{rs}^{12}\\
 A_{rs}^{21}& A_{rs}^{22}\edm,$$
we have the following conclusions.

\iffalse
\begin{lemma}\label{lem-12}
	If we switch $\xi_a$ with $\xi_b$, i.e., $\tilde{\xi}_a=\xi_b$, $\tilde{\xi}_b=\xi_a$, $\tilde{\xi}_i=\xi_i$ for the others, then
	\[
	\tilde{\alpha}_{ab}=\alpha_{ba}=\alpha_{ab}^t,\quad \tilde{\alpha}_{ij}=\alpha_{ij},~  i,j\not\in\{a,b\}.
	\]
\end{lemma}
\begin{proof} We only prove for $(a,b)=(1,2)$, the others are similar.
	From $X=\bgm\Q_1&\Q_2&\cdots&\Q_N\edm A=\bgm\tilde{\Q}_1&\tilde{\Q}_2&\cdots&\tilde{\Q}_N\edm \tilde{A}$ we get
	\[
	\tilde{A}=\begin{pmatrix}
		&I_2&\\I_2&&\\&&I_{2N-4}
	\end{pmatrix}A.
	\]
	So
	\[
	\tilde{A}\tilde{A}^t=\begin{pmatrix}
		&I_2&\\I_2&&\\&&I_{2N-4}
	\end{pmatrix}AA^t\begin{pmatrix}
		&I_2&\\I_2&&\\&&I_{2N-4}
	\end{pmatrix}=\begin{pmatrix}
		\alpha_{22}&\alpha_{21}&\cdots&\alpha_{2N}\\
		\alpha_{12}&\alpha_{11}&\cdots&\alpha_{1N}\\
		\vdots&\vdots&\ddots&\vdots\\
		\alpha_{N2}&\alpha_{N1}&\cdots&\alpha_{NN}
	\end{pmatrix}.
	\]
\end{proof}
\fi

\begin{lemma}\label{lem-theta} 
	%The following two sets have empty intersection,
	 $\Big{\{}\q_r\pm\q_s\,\Big{|}\, 1\leq r\neq s\leq N\Big{\}}\cap\Big{\{}0, \pm2\q_j\,\Big{|}\, 1\leq j\leq N\Big{\}}=\varnothing.$ 
\end{lemma}
	\begin{proof}
		This follows from the fact that
		$0<|\xi_r\pm\xi_s|<|2\xi_j|$ for $r\neq s$.
		%following two lattice points sets $\Big{\{}\xi_r\pm\xi_s\,\Big{|}\, 1\leq r\neq s\leq N\Big{\}}$ and $\Big{\{}\pm2\xi_q\,\Big{|}\, 1\leq q\leq N\Big{\}}$ have empty intersection.
	\end{proof}

\begin{lemma}\label{lem-alprr} For every $r$, $1\leq r\leq N$, we have
	$$ A_{rr}=\bgm a_r&\\&a_r\edm.$$
\end{lemma}
\begin{proof}
	By definition, we have $ A_{rr}^{12}= A_{rr}^{21}$.
From $|x|=1$, we can obtain that
	\begin{equation*}
		\begin{split}
			\frac{1}{2}=&\sum_{q}[( A_{qq}^{11}- A_{qq}^{22})\cos(2\q_q)+( A_{qq}^{11}+ A_{qq}^{22})+( A_{qq}^{12}+ A_{qq}^{21})\sin(2\q_q)
			+\\
			&\sum_{r\neq s}[( A_{rs}^{11}- A_{rs}^{22})\cos(\q_r+\q_s)+( A_{rs}^{11}+ A_{rs}^{22})\cos(\q_r-\q_s)+\\&~~~~~~
			( A_{rs}^{12}+ A_{rs}^{21})\sin(\q_r+\q_s)
			+( A_{rs}^{21}- A_{rs}^{12})\sin(\q_r-\q_s)].
	\end{split}\end{equation*}
	Note that for all $1\leq r\leq N$, $\theta_r\not\equiv0$. So it follows from the Lemma~\ref{lem-theta}  that $ A_{rr}^{11}- A_{rr}^{22}=0,~ A_{rr}^{12}+ A_{rr}^{21}=0,$
	which complete the proof.
\end{proof}
	Since $x$ is an isometric immersion, we have
	\beq\label{eq-XX}
	\langle\frac{\partial x}{\partial u_k},\frac{\partial x}{\partial u_l}\rangle=\delta_{kl},~~~1\leq k,l\leq n,
	\eeq
	where $u=(u_1, u_2, \cdots, u_n)$ is the coordinates of $\mathbb{R}^n$. Using the expression \eqref{eq-X} of $x$, \eqref{eq-XX} can be rewritten as below:
	$$\sum_{r}\xi_{rk}\xi_{rl}\Q_r\bgm0&1\\-1&0\edm A_{rr}\bgm0&-1\\1&0\edm\Q_r^t+\sum_{r<s}(\xi_{rk}\xi_{sl}+\xi_{sk}\xi_{rl})\Q_r\bgm0&1\\-1&0\edm A_{rs}\bgm0&-1\\1&0\edm\Q_s^t=\frac{\delta_{kl}}{4\pi^2},$$
    where $\xi_{rk}$ is the $k$-th coordinates of $\xi_r$, and we have used the fact that
    $$\Q_r\bgm0&1\\-1&0\edm A_{rs}\bgm0&-1\\1&0\edm\Q_s^t=\Q_s\bgm0&1\\-1&0\edm A_{sr}\bgm0&-1\\1&0\edm\Q_r^t,$$
    which can be verified easily.
	It follows from Lemma~\ref{lem-alprr} that the first term in the left hand side of the above equation is constant. Hence for all $1\leq k,l\leq n$, we have
	\beq\label{eq-alprs}
	\begin{split}
		0=\sum_{r<s}(\xi_{rk}\xi_{sl}+\xi_{sk}\xi_{rl})[&( A_{rs}^{11}- A_{rs}^{22})\cos(\q_r+\q_s)+( A_{rs}^{11}+ A_{rs}^{22})\cos(\q_r-\q_s)\\+
		&( A_{rs}^{12}+ A_{rs}^{21})\sin(\q_r+\q_s)
		+( A_{rs}^{21}- A_{rs}^{12})\sin(\q_r-\q_s)].
	\end{split}
	\eeq
	%where Lemma~\ref{lem-12} has been used.
 Define $\mathcal{E}=\{\xi_r\pm\xi_s,1\leq r<s\leq n\}$, we call the set of pairs  
 \[\{(\xi_{r_1},\xi_{s_1}), \cdots,(\xi_{r_p},\xi_{s_p}), (\xi_{r_{p+1}}, -\xi_{s_{p+1}}), \cdots, (\xi_{r_q}, -\xi_{s_q})\}\]
 $\eta$-set if
    $$\xi_{r_1}+\xi_{s_1}=\cdots=\xi_{r_p}+\xi_{s_p}=\xi_{r_{p+1}}-\xi_{s_{p+1}}=\cdots=\xi_{r_q}-\xi_{s_q}=\eta\in \mathcal{E}.$$
    Using $|\xi_{r_j}|=|\xi_{s_j}|$ we have
    $$\langle \eta, \eta\rangle=|\xi_{r_j}|^2+|\xi_{s_j}|^2\pm 2\langle\xi_{r_j},\xi_{s_j}\rangle=2\langle \xi_{r_j},\eta\rangle>0,\quad 1\leq j\leq q.$$
    It is straightforward to verify that $\pm\xi_{r_1},\cdots,\pm\xi_{r_q}, \pm\xi_{s_1},\cdots,\pm\xi_{s_q}$ are distinct with each other.

    Denote by  $\xi_{r_{j}}\odot\xi_{s_{j}}$ the symmetric product of $\xi_{r_{j}}$ and $\xi_{s_j}$, we have the following lemma.
\begin{lemma}\label{lem-homo}%\label{prop-homo}
	Let $x:T^n=\mathbb{R}^n/{\Lambda_n} \longrightarrow  \mathbb{S}^{ m}$ be a linearly full minimal flat torus. If for any $\eta\in\mathcal{E}$, the $\eta$-set forms a linearly independent set of symmetric products
 $\{
\xi_{r_1}\odot\xi_{s_1}, \xi_{r_2}\odot\xi_{s_2}, \cdots,\xi_{r_q}\odot\xi_{s_q}\}$,
	then $ m$ is odd and $x$ is homogeneous.
\end{lemma}
\begin{proof}
    Set
    $$t_{r_{a}s_a}\triangleq A_{r_{a}s_{a}}^{11}- A_{r_{a}s_{a}}^{22},~1\leq a\leq p,~~~t_{r_{b}s_b}\triangleq A_{r_{b}s_{b}}^{11}+ A_{r_{b}s_{b}}^{22},~p+1\leq b\leq q.$$
    By the linear independence of trigonometric functions, it follows from \eqref{eq-alprs} that $\{t_{r_{1}s_{1}},\cdots,t_{r_{q}s_{q}}\}$ satisfy the following linear equations:
    \beq\label{eq-line}
    (\xi_{r_{1}j}\xi_{s_{1}k}+\xi_{r_{1}k}\xi_{s_{1}j})t_{r_{1}s_{1}}+(\xi_{r_{2}j}\xi_{s_{2}k}+\xi_{r_{2}k}\xi_{s_{2}j})t_{r_{2}s_{2}}+\cdots+(\xi_{r_{q}j}\xi_{s_{q}k}+\xi_{r_{q}k}\xi_{s_{q}j})t_{r_{q}s_{q}}=0,~~1\leq j,k\leq n.
    \eeq

    Note that \eqref{eq-line} is equivalent to
    $$t_{r_{1}s_{1}}\xi_{r_{1}}\odot\xi_{s_{1}}+t_{r_{2}s_{2}}\xi_{r_{2}}\odot\xi_{s_{2}}+\cdots+t_{r_{q}s_{q}}\xi_{r_{q}}\odot\xi_{s_{q}}=0.$$
    It implies $t_{r_{j}s_{j}}=0$ for all $1\leq j\leq m$, since these symmetric products are linearly independent. Similarly, the other coefficients also vanish in \eqref{eq-alprs} and we have $ A_{rs}=0~ (r\not= s)$.

    By embedding $\mathbb{S}^{ m}$ into $\mathbb{S}^{2N-1}$, we can assume the immersion $x$ has the form as given in \eqref{eq-X}.
	From Lemma~\ref{lem-alprr} and $ A_{rs}=0~ (r\not= s)$,  the matrix $A$ in the expression \eqref{eq-X} satisfies
	$$
	AA^t=\text{diag}\Big\{a_1, a_1, a_2, a_2, \cdots, a_N, a_N\Big\}
	\iffalse
	\bgm
	a_1&&&&&&&\\
	&a_1&&&&&&\\
	&&a_2&&&&&\\
	&&&a_2&&&&\\
	&&&&\ddots&&&\\
	&&&&&\ddots&&\\
	&&&&&&a_N&\\
	&&&&&&&a_N
	\edm
	\fi
	,
	$$
	with
	$$a_1+a_2+\cdots+a_N=1,~~a_1 \geq a_2\geq\cdots\geq a_N\geq 0,~1\leq j\leq N.$$
	
	Consider the QR decomposition of $A^t$, there exists an upper triangular matrix $L$ such that
	$$
	L^tL=\text{diag}\Big\{a_1, a_1, a_2, a_2, \cdots, a_N, a_N\Big\}
	\iffalse
	\bgm
	a_1&&&&&&&\\
	&a_1&&&&&&\\
	&&a_2&&&&&\\
	&&&a_2&&&&\\
	&&&&\ddots&&&\\
	&&&&&\ddots&&\\
	&&&&&&a_N&\\
	&&&&&&&a_N
	\edm
	\fi
	,
	$$
	which implies $L$ must be diagonal, i.e.,
	$$L=\text{diag}\Big{\{}\sqrt{a_1},\sqrt{a_1},\sqrt{a_2},\sqrt{a_2},\cdots\cdots,\sqrt{a_N},\sqrt{a_N}\,\Big{\}}.$$
	Hence, up to an orthogonal transformation, $x$ can be expressed as below:
	$$x=\bgm\Q_1&\Q_2&\cdots&\Q_N\edm \text{diag}\Big{\{}\sqrt{a_1},\sqrt{a_1},\sqrt{a_2},\sqrt{a_2},\cdots,\sqrt{a_N},\sqrt{a_N}\,\Big{\}}.$$
	Since $x$ is linearly full, we can obtain that $ m$ is odd, and $x$ is the orbit of a torus group acting at	$$(\sqrt{a_1},\sqrt{a_1},\sqrt{a_2},\sqrt{a_2},\cdots,\sqrt{a_\frac{ m+1}{2}},\sqrt{a_\frac{ m+1}{2}}).$$
\end{proof}

\begin{definition}\label{def-unimo}
\iffalse
A finite set $V$ in $\mathbb{R}^n$ is called to satisfy the unimodular condition,
%if for all $1\leq k\leq n$, any nontrivial parallelepiped spanned by $k$ vectors in $V$ shares the same volume if and only if they live in a same $k$-dimensional subspace.
if for $\{v_1, v_2, \cdots, v_{k}\}$ and  $\{w_1, w_2, \cdots, w_{k}\}$, which are two arbitrary collections of $k$ vectors in $V$, either
$$v_1\wedge v_2\wedge \cdots \wedge v_{k}=\pm w_1\wedge w_2\wedge \cdots \wedge w_{k},$$
or
$$v_1\wedge v_2\wedge \cdots \wedge v_{k}\neq\pm w_1\wedge w_2\wedge \cdots \wedge w_{k}.$$
\fi
%$\{v_1, v_2, \cdots, v_{k+1}\}\subset V$ has either rank $k+1$, or rank $k$ and satisfies
%$$c_1v_1+c_2v_2+\cdots+c_{k+1}v_{k+1}=0,$$
%with coefficients all take values in $\{1,-1\}$.

A finite set $V$ of rank $k$ in $\mathbb{R}^n$ is called to satisfy the unimodular condition, if all the $k$-dimensional parallelepipeds spanned in $V$ have a same volume.
\end{definition}
\begin{remark}
Suppose $V$ is a finite set of rank $k$ and satisfies the unimodular condition, then for $\{v_1, v_2, \cdots, v_{l}\}$ and  $\{w_1, w_2, \cdots, w_{l}\}$, which are two arbitrary collections of $l$ vectors in $V$ with $l\leq k$, there must have
$$v_1\wedge v_2\wedge \cdots \wedge v_{l}=\pm w_1\wedge w_2\wedge \cdots \wedge w_{l}$$
when $\mathrm{Span}_\mathbb{R}\{v_1, v_2, \cdots, v_{l}\}= \mathrm{Span}_\mathbb{R}\{w_1, w_2, \cdots, w_{l}\}$.
\end{remark}
\begin{lemma}\label{lem:unimo}
	If $\{\xi_1,\xi_2, \cdots, \xi_N\}$ satisfies the unimodular condition, then for any $\eta\in\mathcal{E}$, the $\eta$-set forms a linearly independent set $\{\xi_{r_1}\odot\xi_{s_1}, \xi_{r_2}\odot\xi_{s_2}, \cdots,\xi_{r_q}\odot\xi_{s_q}\}$.
\end{lemma}
\begin{proof}

    We claim that $\eta, \xi_{r_1}, \xi_{r_2}, \cdots, \xi_{r_q}$ are linearly independent.
    Then %we consider the symmetric product of $\xi_{r_j}$ and $\xi_{s_j}$.
    by extending $\{\eta, \xi_{r_1}, \xi_{r_2}, \cdots, \xi_{r_q}\}$ to a basis of $\mathbb{R}^n$, it is easy to see that
   $$
       \xi_{r_1}\odot\xi_{r_1},~~ \xi_{r_2}\odot\xi_{r_2},~ \cdots,~~\xi_{r_q}\odot\xi_{r_q},~~\xi_{r_1}\odot\eta,~~\xi_{r_2}\odot\eta,~ \cdots,~~\xi_{r_q}\odot\eta
   $$
   are linearly independent. %, where $v$ is an arbitrary vector in $\mathbb{R}^n.$
   Combining this with $\xi_{s_j}=\pm(\eta-\xi_{r_j})$, we can derive that
    \begin{equation*}
    \xi_{r_1}\odot\xi_{s_1},~~ \xi_{r_2}\odot\xi_{s_2},~ \cdots,~~\xi_{r_q}\odot\xi_{s_q}
    \end{equation*}
    are linearly independent.
    
    Now it suffices to prove the above claim. We prove it by contradictions. First we extend $\eta$ to a
    maximal linearly independent subset in $\{\eta, \xi_{r_1}, \xi_{r_2}, \cdots, \xi_{r_q}\}$, which can be assumed  $\{\eta, \xi_{r_1}, \xi_{r_2}, \cdots, \xi_{r_{t-1}}\}$ with $t<q+1$. Then $\eta, \xi_{r_1}, \xi_{r_2}, \cdots, \xi_{r_{t-1}}, \xi_{r_{t}}$ must be linearly dependent. Now it is left to discuss two cases.

    %Suppose that $\{\xi_{r_1}, \xi_{r_2}, \cdots, \xi_{r_t}\}$ are maximal linearly independent subset with $t\leq m$. It can be divided into two cases to obtain the contradiction.
    The first case is that $\xi_{r_1}, \xi_{r_2}, \cdots, \xi_{r_{t-1}}, \xi_{r_t}$ are linearly dependent. 
    %The first case is that $\eta,\xi_{r_1}, \xi_{r_2}, \cdots, \xi_{r_t}$ are linearly independent, in which case we also have $t<m$.
    Let
    \beq
    \label{eq-case1}
    \xi_{r_t}=c_1\xi_{r_1}+c_2\xi_{r_2}+\cdots\cdots+c_{t-1}\xi_{r_{t-1}}.
    \eeq
    It follows from the unimodular condition that these coefficients all take values in $\{0,\pm 1\}$. Taking inner product of \eqref{eq-case1} with $\eta$, we can obtain
     \beq
    \label{eq-sum1}
    \sum_{j=1}^{t-1} c_j=1.
     \eeq
    Note that \[
    \xi_{s_t}\wedge\xi_{r_1}\wedge\cdots\wedge\xi_{r_{t-1}}=\pm\,\eta\wedge\xi_{r_1}\wedge\cdots\wedge\xi_{r_{t-1}}\neq 0.
    \]
    \iffalse
    Assume
    \beq
    \label{eq-case1}
    \xi_{r_m}=c_1\xi_{r_1}+c_2\xi_{r_2}+\cdots\cdots+c_t\xi_{r_t}.
    \eeq
    \fi
    %On the other hand,
    For any $1\leq j\leq t-1$, using $\xi_{s_j}=\pm(\eta-\xi_{r_j})$, we have
    $$\xi_{s_t}\wedge\xi_{r_1}\wedge\cdots\wedge\xi_{r_{j-1}}\wedge\xi_{s_j}\wedge\xi_{r_{j+1}}\wedge\cdots\wedge\xi_{r_{t-1}}=\pm(c_j-1)\xi_{s_{t}}\wedge\xi_{r_1}\wedge\cdots\wedge\xi_{r_{t-1}}.$$
    From the unimodular condition again we get
    $$c_j-1\in\{0,\pm1\},~~1\leq j\leq t-1,$$
    which implies $c_j\in\{0,1\}$ for all $1\leq j\leq t-1$. Combining this with \eqref{eq-sum1}, we derive that only one $c_j$ is nonzero and equals $1$. It follows that $\xi_{r_t}=\xi_{r_j}$, contradicting with the fact that $\pm\xi_{r_1},\cdots,\pm\xi_{r_q}, \pm\xi_{s_1},\cdots,\pm\xi_{s_q}$ are distinct with each other.

    The second case is that 
    $\xi_{r_1}, \xi_{r_2}, \cdots, \xi_{r_t}$ are linearly independent. 
    Then we can assume
    \beq
    \label{eq-case2}
    \eta=c_1\xi_{r_1}+c_2\xi_{r_2}+\cdots\cdots+c_t\xi_{r_t}.
    \eeq
    %It follows from the unimodular condition that these coefficients all take values in $\{0,\pm 1\}$.
    Taking inner product of \eqref{eq-case2} with $\eta$, we can obtain
     \beq
    \label{eq-sum2}
    \sum_{j=1}^t c_j=2.
     \eeq
    On the other hand, using $\xi_{s_j}=\pm(\eta-\xi_{r_j})$, we have $$\xi_{s_j}=\pm\left(c_1\xi_{r_1}+\cdots+c_{j-1}\xi_{r_{j-1}}+(c_j-1)\xi_{r_j}+c_j\xi_{r_{j+1}}+\cdots+c_t\xi_{r_t}\right),~~1\leq j\leq t.$$
    %$$\xi_{r_1}\wedge\cdots\wedge\xi_{r_{j-1}}\wedge\xi_{s_j}\wedge\xi_{r_{j+1}}\wedge\xi_{r_t}=\pm(c_j-1)\xi_{r_1}\wedge\cdots\wedge\xi_{r_t}.$$
    It follows from the unimodular condition that
    $$c_j, c_j-1\in\{0,\pm1\},~~1\leq j\leq t,$$
    which implies $c_j\in\{0,1\}$ for all $1\leq j\leq t$. Combining this with \eqref{eq-sum2}, we derive that only two $c_j$ are nonzero and equal $1$, which can be assumed to be $c_1$ and $c_2$. It follows that $\eta=\xi_{r_1}+\xi_{r_2}$. So we have $\xi_{r_2}=\pm\xi_{s_2}$ also   contradicting with the fact that $\pm\xi_{r_1},\cdots,\pm\xi_{r_q}, \pm\xi_{s_1},\cdots,\pm\xi_{s_q}$ are distinct.
\end{proof}
%The next corollary is obvious. 
By Lemma \ref{lem-homo} and \ref{lem:unimo} we immediately obtain the following proposition. 
%\begin{corollary}%\label{prop-homo}
\begin{proposition}
\label{prop-homo}
	Let $x:T^n=\mathbb{R}^n/{\Lambda_n} \longrightarrow  \mathbb{S}^{ m}$ be a linearly full minimal flat torus. If %the set constituted by eigenfucntions corresponding to $n$
	$\{\xi_1,\xi_2, \cdots, \xi_N\}$ satisfies the unimodular condition, then $ m$ is odd and $x$ is homogeneous.
\end{proposition}%\label{prop-homo}
%\end{corollary}
%%
\iffalse
\begin{proposition}\label{prop-homo}
	Let $X:T^n=\mathbb{R}^n/{\Lambda_n} \longrightarrow  \mathbb{S}^p$ be a flat torus of dimension $n\leq 4$. If it is a linearly full minimal immersion by the first eigenfunctions, then $p$ is even and $X$ is homogeneous.
\end{proposition}
\fi
%%

\begin{remark}\label{rk-homog}
In Section~\ref{sec-4}, we will show %the unimodular condition is always satisfied by
that all $\lambda_1$-minimal flat tori of dimension no more than $4$ %. Therefore they are all
are homogeneous. In fact all of them satisfy the unimodular condition with only one exception.
\iffalse
Explanation for the exceptional example in dimension $4$:

Assume $m\leq 4$, which seems to be true.
%If $m\leq 3$,
Note that in a lattice of rank $2$, there are at most three distinct shortest lattice vectors modulo the direction.
Then it is easy to verify that under the assumption that $m\leq 4$, both $\{\xi_{r_1}, \xi_{r_2}, \cdots, \xi_{r_m}\}$ and $\{\xi_{s_1}, \xi_{s_2}, \cdots, \xi_{s_m}\}$ have rank at least $m-1$, which also implies
$$\xi_{r_1}\odot\xi_{s_1},~~ \xi_{r_2}\odot\xi_{s_2},~ \cdots\cdots,~~\xi_{r_m}\odot\xi_{s_m}$$
are linearly independent, from which the conclusion of Lemma~\ref{lem-alprs} follows.
%If $m=4$, then both $\{\xi_{r_1}, \xi_{r_2}, \cdots, \xi_{r_m}\}$ and $\{\xi_{s_1}, \xi_{s_2}, \cdots, \xi_{s_m}\}$ have rank at least $3$, which also implies
%$$\xi_{r_1}\odot\xi_{s_1},~~ \xi_{r_2}\odot\xi_{s_2},~ \cdots\cdots,~~\xi_{r_m}\odot\xi_{s_m}$$
%are linearly independent.
\fi
\end{remark}

\section{Variational characterizations of homogeneous minimal flat tori}\label{sec-variation}%isometric immersions of tori}
 %in $\mathbb{S}^m$}
In this section, we give two variational characterizations for homogeneous minimal flat $n$-tori in spheres, from which two construction approaches can be derived. The problems appeared in these two approaches, as well as the further construction to $\lambda_1$-minimal immersions are also discussed.

Let $x: T^n=\mathbb{R}^n/\Lambda_n\rightarrow \mathbb{S}^m$ be a %linearly full
homogeneous minimal isometric immersions, with the metric  given by
$$\frac{4\pi^2}{n}(du_1^2+du_2^2+\cdots+du_n^2).$$
Suppose $\{\xi_j\}_{j=1}^{N}$ are all lattice vectors (up to $\pm1$) in $\Lambda_n^{*}$ of length $1$.
From the homogeneity of $x$, we can assume $m=2N-1$ % is an odd number,
and write $x$ %can be written
as follows:
\beq\label{eq-x}
x=(c_1e^{i\q_1},c_2e^{i\q_2},\cdots\cdots,c_Ne^{i\q_N}),
\eeq
where $\q_j=2\pi\langle\xi_j, u\rangle,~1\leq j\leq N$. %and
%$\{\xi_j\}_{j=1}^{N}$ are $N$ vectors of $\Lambda_n^{*}$ with the same length.
%Since $x$ is minimal and flat, we have
Then we have
%\beq
%4\pi^2|\xi_j|^2=e^{2\rho}n,
%\eeq
\beq\label{eq-flat}
%4\pi^2
\bgm
\xi_1^t&\xi_2^t&\cdots&\xi_N^t
\edm
\bgm
c_1^2&&&\\
&c_2^2&&\\
&&\ddots&\\
&&&c_N^2
\edm
\bgm
\xi_1\\
\xi_2\\
\vdots\\
\xi_N
\edm=%e^{2\rho}I_n,
\frac{1}{n}I_n.
\eeq
%where $e^{2\rho}$ is an insignificant constant, up to a %scaling, we assume it to be $\frac{4\pi^2}{n}$.

Assume $\{\eta_1,\eta_2,\cdots,\eta_n\}$ is a generator of $\Lambda_n^{*}$. Then there exist integers $a_{j_k}$ such that
$$\xi_j=a_{j_1}\eta_1+a_{j_2}\eta_2+\cdots\cdots+a_{j_n}\eta_n,\quad 1\leq j\leq N.$$
%Consider a coordinates system with $\{\xi_1,\xi_2,\cdots,\xi_n\}$ as an orthonormal basis,
We denote
$$A_j=(a_{j_1},a_{j_2},\cdots\cdots,a_{j_n}),\quad Y^t=(A_1^t,\cdots,A_N^t), %\quad B_j= A_j^t A_j,
\quad 1\leq j\leq N,$$
and use $Y$ also representing the set of its row vectors if no confusion caused.
%In the following, we also abuse the notation $Y$ to denote the set constituted by $A_j$ for $1\leq j\leq N$.

Moreover, we set
$$Q=\bgm
\eta_1\\
\eta_2\\
\vdots\\
\eta_n
\edm
\bgm
\eta_1^t&\eta_2^t&\cdots&\eta_n^t
\edm,
$$
which is a positive-definite matrix, called Gram matrix of $\Lambda^*_n$. It is well-known that the volume of $x$ equals $\frac{2^n\pi^{n}}{\sqrt{n^n\det(Q)}}.$
\begin{remark}\label{rk-notation}
The minimal immersion $x$ given in \eqref{eq-x} is uniquely determined by the following data set
$$\{Y ,~ Q,~ (c_1^2, c_2^2, \cdots, c_N^2)\},$$
which is called the {\bf \em matrix data} of $x$, and will be used to present examples in sec~\ref{sec-examples}.
\iffalse
We point out that two homogeneous minimal immersions determined by $\{Y ,~ Q,~ (c_1^2, c_2^2, \cdots, c_N^2)\}$ and $\{Y ,~ Q,~ (\tilde{c}_1^2, \tilde{c}_2^2, \cdots, \tilde{c}_N^2)\}$ with $(c_1^2, c_2^2, \cdots, c_N^2)\neq (\tilde{c}_1^2, \tilde{c}_2^2, \cdots, \tilde{c}_N^2)$
%having the same $Y$ and $Q$, but distinct $(c_1^2, c_2^2, \cdots, c_N^2)$,
can not be congruent. In fact, they are congruent if and only if there is an orthogonal matrix $O$ such that $$\emph{diag}\{c_1^2,c_1^2, c_2^2,c_2^2, \cdots, c_N^2, c_N^2\}=\emph{diag}\{\tilde c_1^2, \tilde c_1^2, \tilde c_2^2,\tilde c_2^2, \cdots, \tilde c_N^2, \tilde c_N^2\}O,$$
which means $O$ can only be the identity matrix.
\fi
\end{remark}
It is obvious that the condition $|\xi_j|=1$ is equivalent to
\beq\label{eq-AQA}
A_jQA_j^t=1,%1\leq j\leq N,
\eeq
i.e.,
$A_j$ lies on the hyper-ellipsoid $\mathcal{Q}$ determined by
$$\bgm
x_1&x_2&\cdots&x_n
\edm
Q\bgm
x_1\\x_2\\\vdots\\x_n
\edm=1.
$$
Furthermore, it is straightforward to verify that the flat condition \eqref{eq-flat} is equivalent to
\begin{equation}\label{eq:flat}
c_1^2 A_1^tA_1+c_2^2A_2^tA_2+\cdots\cdots+c_N^2A_N^tA_N=\frac{1}{n}Q^{-1}.
\end{equation}

\subsection{Two variational characterizations}\label{subsec-vara}
Consider the space $S(n)$ of $n\times n$ symmetric matrices over $\mathbb{R}$, which is a $\frac{n(n+1)}{2}$-dimensional Euclidean space endowed with inner product:
$$\langle S_1, S_2\rangle=%\frac{1}{n}
\operatorname{tr}(S_1S_2),\quad S_1, S_2\in S(n).$$
\iffalse
\begin{figure}[htbp]
\begin{minipage}{.49\linewidth}%\label{fig:subfig:b}
\centering
\includegraphics[width=0.75\textwidth]{sigma.pdf}
\caption{Space of symmetric matrices}
%%\end{figure}
\end{minipage}
%%\hfill
\begin{minipage}{.49\linewidth}%\label{fig:subfig:b}
\centering
%%\begin{figure}[htbp]
%%\centering
\includegraphics[width=0.95\textwidth]{WX.pdf}
\caption{$W_X$ in $\Sigma_+$}
\end{minipage}
\end{figure}
\fi
\iffalse
\begin{figure}[htbp]\centering
	\subfigure{\includegraphics[width=0.4\textwidth]{sigma}}
 %\caption{Space of symmetric matrices}
 \hspace{0.5in}
	%\caption{Space of symmetric matrices}
 \subfigure{\includegraphics[width=0.5\textwidth]{WX}}
 %\caption{$W_X$ in $\Sigma_+$}\label{fig:WX}
\end{figure}
\fi
For any given vector $v\in \mathbb{R}^n$, $\Pi_a(v):=\{M|\langle v^tv , M\rangle=vMv^t=a,a\in \mathbb{R}\}$ is an affine hyperplane dividing $S(n)$ into two half spaces:
$$S^+_a(v)=\{M|vMv^t\geq a\}, ~~~S^-_a(v)=\{M|vMv^t\leq a\}.$$
Let $\Sigma$ be the set of semi-positive definite matrices. Then it is well known that
$\Sigma%\cup\{0\}
=\cap_{v\in \mathbb{R}^n} S^+_0(v)$
is a convex cone in $S(n)$ with the set of positive definite matrices as its interior, which is denoted  by $\Sigma_+$. By our definition, we have $Q^{-1}\in \Sigma_+$, and $ A_j^t A_j\in \Sigma, 1\leq j\leq N$.

Given a subset $X\subset\mathbb{Z}^n$, let $C_X$  be the convex hull spanned by $A^tA$ for all $A\in X$, and $V_X$  be the  affine subspace  $\cap_{A\in X}\Pi_{{1}}(A)\subset S(n)$.
Moreover, we also consider the smooth linear submanifold $W_X= V_X\cap \Sigma_+$ (see Figure~\ref{fig:WX}), whose geometric meaning is the set of all hyper-ellipsoids passing through $X$.
\begin{figure}[H]
	\centering
	\includegraphics[width=0.5\textwidth]{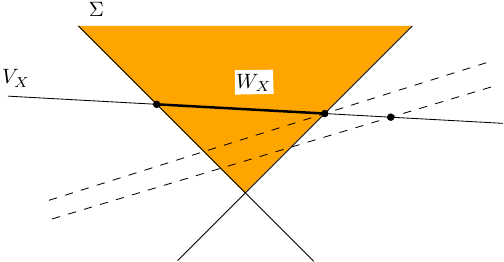}
	\caption{$W_X$ in $\Sigma_+$}\label{fig:WX}
\end{figure}

With respect to these notations, the condition \eqref{eq-AQA} is equivalent to $Q\in W_Y$,
and \eqref{eq:flat} is equivalent to
$\frac{Q^{-1}}{n}\in C_Y$.

\begin{theorem}\label{thm-variation}
Let $x: T^n=\mathbb{R}^n/\Lambda_n\rightarrow \mathbb{S}^{2N-1}$ be an isometric homogeneous immersion. If $x$ is minimal, then $Q$ is a critical point of the  determinant function restricted on $W_Y$, % and $Q^{-1}$ lies in the convex hull $C_Y$,
where $Y$ is the set of integer vectors determined by $x$, and $Q$ the Gram matrix of $\Lambda_n^*$ under some chosen generator.

Conversely, given a set $X$ of integer vectors, if $Q$ is a critical point of the determinant function restricted on $W_X$ and $\frac{Q^{-1}}{n}$ lies in the convex hull $C_X$, then the torus $\mathbb{R}^n/\Lambda_n$ determined by $Q^{-1}$ (as the Gram matrix of $\Lambda_n$) admits an isometric minimal immersion in $\mathbb{S}^{2N-1}$.
%$$Q\in V\bigcap \overset{\circ}{C}.$$
\end{theorem}
\begin{proof}
Let $\gamma(t)$ be a smooth curve in $W_Y$ passing through $Q$ at $t=0$. Set $f(t)\triangleq\det{\gamma(t)}$, %and $\eta(t)\triangleq\gamma(t)^{-1}$, which is a smooth curve on $W$.
it is easy to verify that
$$f^{'}(0)=f(0)\langle Q^{-1}, \gamma'(0)\rangle.$$
As an isometric homogeneous immersion, $x$ is minimal if and only if
$Q\in W_Y$, $Q^{-1}\in C_Y$.
The conclusion follows from that $ \mathrm{Span}\{ A^tA ~|~A\in Y\}$ is the normal space of $W_Y$ at $\gamma(0)=Q$.
\end{proof}
\begin{remark}\label{rk-diag}
We point out that if $Q$ is a critical point of the determinant function restricted on $W_X$, then $\frac{Q^{-1}}{n}$ lies in $\mathrm{Span}\{A^tA\, |\, A\in X\}$ automatically. In fact let $Q+tS$ be an arbitrary segment in $W_X$, %$S\in T_Q W_X$ be a tangent vector,
then it follows from $\langle Q^{-1},S\rangle$=0 that
$$\langle \frac{Q^{-1}}{n},Q+tS\rangle=1,$$
which implies $\frac{Q^{-1}}{n}\in \mathrm{Span}\{A^tA\, |\, A\in X\}$.

As a result, if all the vectors $A_i$ in $X$ can be arranged as row vectors to form a block diagonal matrix, then the critical point $Q$ is also block diagonal.
\end{remark}
Although the following conclusion is well known, we give a proof here for completeness.
\begin{lemma}\label{lem-det}
The function $\ln\circ\det$ restricted on $\Sigma_+$ is strictly concave.
\end{lemma}
\begin{proof}
%From the Minkowski's determinant inequality the function $\det$ is a concave function which can also be seen in the following way.
Given a positive definite matrix $P\in \Sigma_+$, Let $\gamma(t)$ be a smooth curve in $\Sigma_+$ with $\gamma'(0)=S\in S(n)\backslash \{0\}$.
%and a direction $S\in S(n)$, $P+tS$ is a path segment in $\Sigma_+$ for $t\in (-\epsilon,\epsilon)$.
Let $f(t)=\ln\circ\det(\gamma(t))$, then we have $f'(t)=\operatorname{tr}\left(\gamma(t)^{-1}S\right)$ and
\[
f''(0)%-(\operatorname{tr}((I+tP^{-1}S)^{-1}P^{-1}S))'
=-\operatorname{tr}\left(P^{-1}SP^{-1}S\right)=-\operatorname{tr}\left(HSHHSH\right)=-|HSH|^2< 0,
\]
where $H$ is a square root of the positive definite matrix $P^{-1}$ (i.e., $H^2=P^{-1}$), and we have used the fact that $S\in S(n)$.
\iffalse
Since $P^{-1}$ is also positive definite and trace function is invariant to similar transformations, we may assume $P^{-1}=\operatorname{diag}(a_1,\cdots,a_n)$ with $a_i>0$. Set $S=(s_{ij})$, we get
\[
f''(t)=\sum_{i,j=1}^n a_ia_js_{ij}s_{ji}=\sum_{i,j=1}^n a_ia_js_{ij}^2>0
\]
which means $f(t)$ is strictly convex.
\fi
\end{proof}
Therefore, for any given subset $X\subset \mathbb{Z}^n$, if $W_X\neq \varnothing$, the determinant function $\det$ has only one critical point in $W_X$, which is the maximal point. This implies the following result about uniqueness.

\begin{corollary}\label{cor-unique}
Let $x: T^n=\mathbb{R}^n/\Lambda_n\rightarrow \mathbb{S}^{2N-1}$ and $\tilde x: \widetilde{T}^n=\mathbb{R}^n/\widetilde{\Lambda}_n\rightarrow \mathbb{S}^{2N-1}$ be two isometric homogeneous minimal immersions, if after choosing suitable generator respectively, $x$ and $\tilde x$ %$\Lambda_n^*$ and $\widetilde{\Lambda}_n^*$ determine
share the same subset $Y\subset \mathbb{Z}^n$, %under suitable chosen generators,
then $T^n$ and $\widetilde{T}^n$ are isometric. Moreover, $x$ is congruent to $\tilde x$ if $\{A^tA\,|\, A\in Y\}$ are linearly independent.
\end{corollary}
\iffalse
\begin{remark}\label{rk-diag}
%Using the Theorem % ~\ref{thm-variation},
One can verify directly that the maximal point $Q$ is also block diagonal, if writing the integer vectors of $X$ as row vectors of a matrix, it is block diagonal.
\end{remark}
If all the vectors $A_i$ in $X$ can be arranged to form a block diagonal matrix, then there is $r<n$ such that every $A_i$ has an expression of $(p_i,0)$ or $(0,q_i)$. where $p_i\in \mathbb{R}^r$, $q_i\in \mathbb{R}^{n-r}$. So $A_i^tA_i$ must be of
\[
\begin{pmatrix}
    P_i&\\&O
\end{pmatrix}\quad \text{or}\quad
\begin{pmatrix}
    O&\\& Q_i
\end{pmatrix}, \quad P_i\in S(r),\~ Q_i\in S(n-r).
\]
Considering $Q$ is the critical point, it is block diagonal since $Q^{-1}\in C_X$ is obviously block diagonal.
\fi

It seems that we can use the following approach to construct a homogeneous minimal immersion of flat torus. Choose a subset $X\subset\mathbb{Z}^n$, determine whether $W_X$ is not empty, and then discriminate whether the maximal point $Q$ of $\det$ restricted on $W_X$ lies in the %interior of
the convex hull $C_X$.

Note that in general the dimension of $V_X$ is $n(n+1)/2-\sharp(X)$, which implies $W_X$ could be empty if $X$ involves too many vectors. Even if $V_X$ exists, it also could have no intersection with $\Sigma_+$. For example, we will get a degenerated matrix when
\[
Q=\begin{pmatrix}
	1&-1/2&-1/2\\-1/2&1&-1/2\\-1/2&-1/2&1
\end{pmatrix},\quad X=\begin{pmatrix}
1&&&1&&1\\&1&&1&1&\\&&1&&1&1
\end{pmatrix}.
\]
It seems to be a challenge to obtain a general method by which one can construct the desirable integer set $X$ efficiently. 

Next we give an alternative method to determine homogeneous minimal immersions of flat tori, by use of a variational characterization of $Q^{-1}$ (Compare Theorem \ref{thm-variation}). 
\begin{theorem}\label{thm-vari}
Let $x: T^n=\mathbb{R}^n/\Lambda_n\rightarrow \mathbb{S}^{2N-1}$ be an isometric homogeneous immersion, where $N$ is the half dimension of the eigenspace of $T^n$ corresponding to $n$. If $x$ is minimal and linearly full, then $Q^{-1}$ lies in the interior of $C_Y$, and is a critical point of the  determinant function restricted on $C_Y$.

Conversely, given a finite set $X$ of integer vectors such that $C_X\cap \Sigma_+\neq \varnothing$, if $P\in \overset{\circ}{C_X}$ is a critical point of the determinant function restricted on $C_X$, then the torus $\mathbb{R}^n/\Lambda_n$ determined by $nP$ (as the Gram matrix of $\Lambda_n$) admits an isometric minimal immersion in some $\mathbb{S}^{2N-1}$.
\end{theorem}
\begin{proof}
%\textcolor{red}{I think we should address the uniqueness of extreme value by Lemma 3.4, which ensures these two extreme values in Theorem 3.2 and 3.6 coincide.}

Suppose $X=\{A_1, A_2, \cdots, A_k\}$, and $P=\sum_{j=1}^k y_j A_j^tA_j$. It follows from  $P\in \overset{\circ}{C_X}$ that $y_j> 0$ for all $1\leq j\leq k$.
We only need to prove that $\frac{P^{-1}}{n}\in W_X$, i.e.,
\beq
\langle \frac{P^{-1}}{n}, A_j^tA_j\rangle=1,~~~1\leq j\leq k.
\eeq
For any given $2\leq j\leq k$, consider the line segment
$$\gamma_j(t)=P+t(A_j^tA_j-A_1^tA_1).$$
It is obvious that for sufficient small $t$, $\gamma_j(t)$ lies in $C_X$. So $P$ is a critical point of $\det(\gamma_j(t))$, which implies that $\langle P^{-1}, \gamma_j'(0)\rangle=0$. Therefore we have
$$\langle P^{-1}, A_j^tA_j\rangle=\langle P^{-1}, A_1^tA_1\rangle,~~~2\leq j\leq k.$$
Note that $\sum_{j=1}^k y_j=1$, it follows that
$$n=\langle P^{-1},P \rangle=\sum_{j=1}^k y_j\langle P^{-1}, A_j^tA_j\rangle=\big(\sum_{j=1}^k y_j\big)\langle P^{-1}, A_1^tA_1\rangle=\langle P^{-1}, A_1^tA_1\rangle,$$
which completes the proof.
\end{proof}
\begin{remark} The above theorem provides another approach to construct minimal flat tori. Choose a finite set $X$ of integer vectors such that $C_X\cap \Sigma_+\neq \varnothing$ (see Figure~\ref{fig:CX}). Calculate the the maximal point $P$ of $\det$ on $C_X$. If $P\in \overset{\circ}{C_X}$, then the matrix data $\{\frac{P^{-1}}{n}, X\}$ provides a minimal flat $n$-torus. Otherwise, we choose $C_{X'}$ to be the face of $C_X$ so that $P\in \overset{\circ}C_{X'}$ (such face could have high codimension and the existance is due to the compactness of $C_X$), then $\{\frac{P^{-1}}{n}, X'\}$ provides a minimal flat $n$-torus.

This remark can be seen as a generalization of Bryant's characterization to minimal flat $2$-tori in spheres, see Proposition 3.3 in \cite{Bryant}.
%This remark can be seen as a generalization of Proposition 3.3 of Bryant \cite{Bryant}, which provides a characterization to minimal flat $2$-tori  in spheres.

\iffalse
Given a finite set $X$ of integer vectors, the maximal point $P$ of $\det$ may lies on the boundary (some facet $C_{X'}$) of $C_X$, which in general can not guarantee $\frac{P^{-1}}{n}\in W_X$. In fact, it can only be proved directly (see also a answer to the related problem in \cite{John48}) that outside the hyper-ellipsoid determined by $\frac{P^{-1}}{n}$, there is no points of $X$.

However, if the facet $C_{X'}$ containing $P$ taking $\{e_i^te_i\}_{i=1}^n$ ($\{e_i\}_{i=1}^n$ is the standard basis of $\mathbb{R}^n$) as some of its vertices, then the torus $\mathbb{R}^n/\Lambda_n$ determined by $nP$ (as the Gram matrix of $\Lambda_n$) admits an isometric minimal immersion in $\mathbb{S}^{p}$.
\fi
\end{remark}
\begin{figure}[htb]
	\centering
	\includegraphics[width=0.33\textwidth]{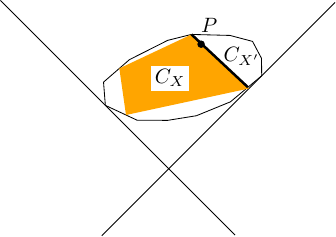}
	\caption{$C_X$ and its faces}\label{fig:CX}
\end{figure}

\subsection{Homogeneous $\lambda_1$-minimal flat tori}
\iffalse
\begin{definition}
	A flat torus in spheres %as \eqref{eq-x}
 is called  $\lambda_1$-minimal, if it is a minimal immersion by the first eigenfunctions. %, i.e., %unit-length
	%vectors $\xi_j$ ($1\leq j\leq N$) are the shortest nonzero vectors in $\Lambda_n^*$.
\end{definition}
\fi
To determine homogeneous $\lambda_1$-minimal immersions of flat tori, one has to solve the following problem. 
\vskip 0.2cm
%By the definition, it is easy to see we are facing a question ($Q_p'$):
\noindent
{\bf Problem.}
%For a given subset $X\subset\mathbb{Z}^n$, determine whether $W_X$ is not empty, and then discriminate whether the maximal point $Q$ of $\det$ restricted on $W_X$ lies in the %interior of
%the convex hull $C_X$
{\em Firstly choose a subset $X\subset\mathbb{Z}^n$, such that in $W_X$ there is  a maximal point $Q$ of $\det$, and $%\frac{1}{n}
Q^{-1}/n\in \overset{\circ}{C_X}$. Then  discriminate in the interior of the hyper-ellipsoid %$\mathcal{Q}$
determined by $Q$, whether there exist other nonzero integer vectors, i.e., discriminate whether $vQv^t\geq 1$ holds for all $v\in \mathbb{Z}^n\backslash \{0\}$.}
\vskip 0.2cm
%Firstly solve Problem 1, then discriminate in the interior of the hyper-ellipsoid $\mathcal{Q}$ determined by $Q$ whether there exist other nonzero integer vectors, i.e., discriminate whether $vQv^t\geq 1$ hold for all $v\in \mathbb{Z}^n\backslash 0$.
\iffalse
\begin{quote}
	Find the hyper-ellipsoid $\mathcal{Q}_p'$ with least volume, passing through every vector in given $Y$ and containing no $v\in \mathbb{Z}^n$ inside except the origin, i.e., to find $Q_p'\in\Sigma_+$ attaining maximum of $\det M$ with constraints $A_iMA_i^t=1$ ($1\leq i\leq N$) and $vMv^t\geq 1~(v\not= 0)$.
\end{quote}
\fi

In $\Sigma_{+}$, the infinite constraints
\beq\label{eq-constraints}
vMv^t\geq 1~ \text{for~ all~} v\in \mathbb{Z}^n\backslash \{0\}
\eeq
define a non-compact convex domain $\Omega$. 
\iffalse
as shown in the following figure. %(see figure~\ref{fig:Qp'}). 
\begin{figure}[H]
	\centering
	\includegraphics[width=0.4\textwidth]{QP}
	\caption{$\Omega$}\label{fig:Qp'}%{The infinite constraints $vMv^t\geq 1$ for all $v\in \mathbb{Z}^n\backslash 0$ will define a non-compact convex domain $\Omega$. When $V_Y$ is in  position as the dashed line, one can only get a minimal immersion but not $\lambda_1$-minimal. A $\lambda_1$-minimal torus can be obtained if $Q_p\in \partial \Omega$ which is a sufficient but not necessary condition.}
\end{figure}
\noindent 
When $V_X$ is in  position as the dashed line, one may only get a minimal immersion but not $\lambda_1$-minimal. 
\fi 
So $X$ can determine a homogeneous $\lambda_1$-minimal immersion of some flat torus if and only if $Q\in \partial \Omega$. 
%Obviously, the solvability of question ($Q_p'$) is based on the solvability of ($Q_p$), in addition with
Given a positive definite matrix $Q$, %to Problem 1,
it seems that to verify whether $Q\in \partial \Omega$, i.e., satisfies \eqref{eq-constraints}, %$vMv^t\geq 1$ for all $v\in \mathbb{Z}^n\backslash 0$ all of which
infers an infinite process. However, due to Minkowski's reduction theory of lattice (see \cite{Siegel} for reference), only finite inequalities need to be checked. This can also be seen from the following simple theorem.
\begin{theorem}%[General criterion for $\lambda_1$-minimality]
\label{thm:general}
For a given $n\times n$ positive definite matrix $Q$ with diagonal entries $1$, there exist $n$ integers $a_k(Q)>0$ $(1\leq k\leq n)$ such that $vQv^t\geq 1$ for all $v\in\mathbb{Z}^n\backslash \{0\}$ if and only if it holds for all integer vectors in
$$S\triangleq\big{\{}v=(v_1,\cdots,v_n)\in \mathbb{Z}^n \, \big{|}\, |v_k|< a_k(Q), 1\leq k\leq n\big{\}}.$$
\end{theorem}
\begin{proof}
    Let $\{\xi_i\}$ be a basis of $\mathbb{R}^n$ with Gram matrix $Q$, $\Pi_k$ the $(n-1)$-dimensional subspace spanned by $\xi_1,\cdots,\xi_{k-1},\xi_{k+1},\cdots,\xi_n$. So the distance of $\xi_k$ to $\Pi_k$ is a fixed value $d_k>0$, which can be determined from the entries $a_{ij}$ of $Q$ by using the least square method, such as
    \begin{equation}\label{eq:dn}
    d_n^2=|\xi_n|^2-|\xi_n^\top|^2=1-|\xi_n^\top|^2=1-(a_{1n},\cdots,a_{(n-1)n})(a_{ij})^{-1}_{1\leq i,j\leq n-1}(a_{1n},\cdots,a_{(n-1)n})^t,
\end{equation}
where $\xi_n^\top$ is the orthogonal projection of $\xi_n$ into $\Pi_n$.

    Suppose $a_k(Q)$ is the integer such that $a_k(Q)-1 \leq d_k^{-1} < a_k(Q)$, then we will get $|\alpha|\geq |c_k|d_k>1$ for any $\alpha=c_1\xi_1+\cdots+c_n\xi_n$
    when some $|c_k|\geq a_k(Q)$.
\iffalse
Explicitly, let $\xi_k^\top$ be the orthogonal projection of $\xi_k$ into $\Pi_k$. By least square method we get
\[\|\xi_n^\top\|^2=(a_{1n},\cdots,a_{(n-1)n})(a_{ij})^{-1}_{1\leq i,j\leq n-1}(a_{1n},\cdots,a_{(n-1)n})^t,\]
in which $a_{ij}$ are corresponding entries in $Q$. So
\begin{equation}\label{eq:dn}
    d_n^2=\|\xi_n\|^2-\|\xi_n^\top\|^2=1-\|\xi_n^\top\|^2
\end{equation}
where $a_n(Q)$ can be determined. The other $a_k(Q)$ can be set similarly, or we just do the same computation again right after exchanging the $k$-th and $n$-th rows and columns in $Q$.
\fi
\end{proof}
%we only need to check finite inequalities to discriminate whether the solution of Problem 1 can give  a solution  of Problem 2.
Although the criterion given in the above theorem only infers finite steps, it still requests a lot of computations. For $n\leq 6$, we give an alternative criterion %those finite inequalities
which comes from Minkowski's reduction theory (see \cite{Ryshkov73}). For the case $n>6$, we do not find such an explicit description %\textcolor{red}{no such explicit description can be obtained---can be found 
in the literature so far.  Applied to our case, the criterion is stated as below.
%When $n\leq 6$, due to Minkowski's reduction theory (see \cite{Ryshkov73}), we only have to check specific finite cases to make sure that the solution of Problem 1 gives  the solution  of Problem 2. Applied to our case, it states as below,

\begin{theorem}%[Special criterion for $\lambda_1$-minimality]
\label{thm:special}	
Suppose $Q$ is a $n\times n$ positive definite matrix with diagonal entries $1$, and $n\leq 6$. Then %the solution $Q$ of Problem 1 can determine $a$  homogeneous $\lambda_1$-minimal flat torus if and only if
$vQv^t\geq 1$ holds for all $v\in\mathbb{Z}^n\backslash \{0\}$, if and only if, it holds for those integer vectors $v$,
%it satisfies $vQ v^t\geq 1$, where
whose entries take values from the first $n$ integers in each rows below  with arbitrary sign (``$+$" or ``$-$"):
	\[
	\begin{array}{llllll}
		1&1&0&0&0&0\\
		1&1&1&0&0&0\\
		1&1&1&1&0&0\\
		1&1&1&1&1&0\\
		1&1&1&1&2&0\\
		1&1&1&1&1&1\\
		1&1&1&1&1&2\\
		1&1&1&1&2&2\\
		1&1&1&1&2&3\, .\\
	\end{array}
	\]
\end{theorem}
\section{Examples of $\lambda_1$-minimal flat tori of dimension $3$ and $4$}\label{sec-examples}
In this section, we present a series of $\lambda_1$-minimal immersions for flat tori $T^3=\mathbb{R}^3/\Lambda_3$ and $T^4=\mathbb{R}^4/\Lambda_4$, which are all homogeneous. Instead of giving the explicit petty immersions%expression
, we use the matrix data %set
$$\{Y ,~ Q,~ (c_1^2, c_2^2, \cdots, c_N^2)\}$$
%given 
introduced in the last section (see Remark~\ref{rk-notation}). Recall that $Q$ is a Gram matrix of the lattice $\Lambda_n^*$ under some chosen generator, which can determine a flat torus $T^n=\mathbb{R}^n/\Lambda_n$; $N$ denotes the number of distinct shortest lattice vectors of $\Lambda^*_n$ up to $\pm1$, %determined by $Q$
%up to $\pm1$,
i.e., the half dimension of the first eigenspace of flat torus $T^n=\mathbb{R}^n/\Lambda_n$; $Y$ is a $n\times N$ matrix taking these vectors as row vectors; %whose rows are integer vectors  corresponding to the %the information of
%shortest lattice vectors of $\Lambda_n^*$;
$c_i^2$ involves the information of immersion. %the information of shortest vectors in $\Lambda_n^*$, Gram matrix $Q$ contains the information of relative position of these vectors,
%Moreover, $N$ denotes the number of distinct shortest lattice vectors of the lattice $\Lambda^*_n$ determined by $Q$ up to $\pm1$, i.e., the half dimension of the first eigenspace of flat torus $T^n=\mathbb{R}^n/\Lambda_n$.

First of all, given two $\lambda_1$-minimal tori $f_i:T^{n_i}\to \mathbb{S}^{m_i} (i=1,2)$, one can construct a new $\lambda_1$-minimal $(n_1+n_2)$-torus by the following direct product (see  \cite{CH}, \cite{TZ} and \cite{2Wang})
\beq\label{eq-prod}
f=\left(\sqrt{\frac{n_1}{n_1+n_2}}f_1,\sqrt{\frac{n_2}{n_1+n_2}}f_2\right):T^{n_1}\times T^{n_2}\to \mathbb{S}^{m_1+m_2+1}.
\eeq
In the sequel, $\lambda_1$-minimal torus constructed in such way is called {\em reducible}, and those non-product ones are called {\em irreducible}.
\begin{remark}\label{rem:reducible}
A reducible $\lambda_1$-minimal flat torus constructed from two homogeneous and $\lambda_1$-minimal tori is homogeneous too. For such torus, by definition, suitable generator of the corresponding dual lattice can be chosen such that %It is easy to see
the matrix $Y$ is block diagonal. Then it follows from Remark~\ref{rk-diag} that %using Theorem~\ref{thm-variation}, one can verify directly that
the Gram matrix $Q$ is also block diagonal. % if $Y$ is block diagonal.
Conversely, a $\lambda_1$-minimal flat torus given by diagonal data set is a reducible $\lambda_1$-minimal flat torus.
\end{remark}
\begin{example}\label{ex:3-prod}
As stated in the introduction, the Clifford torus and equilateral torus are two (only two) %It follows from the classification of
$\lambda_1$-minimal flat $2$-torus, %the Clifford torus and equilateral torus.
whose matrix data are 
$$Q=\left(
\begin{array}{cc}
 1 & 0 \\
0 & 1  
\end{array}
\right),~~~
Y=\left(
\begin{array}{ccc}
 1 & 0  \\
 0 & 1 
\end{array}
\right),~~~ (c_1^2,\, c_2^2)=\left(\frac{1}{2},\, \frac{1}{2}\right),
$$
$$Q=\left(
\begin{array}{ccc}
 1 & -\frac{1}{2} \\
 -\frac{1}{2} & 1 
\end{array}
\right),~~~
Y^t=\left(
\begin{array}{cccc}
 1 & 0 & 1 \\
 0 & 1 & 1
\end{array}
\right),~~~ (c_1^2,\, c_2^2,\, c_3^2)=\left(\frac{1}{3},\, \frac{1}{3},\,\frac{1}{3}\right).
$$

Taking one from these two torus, considering the direct product (as in \eqref{eq-prod}) of it with the unite circle in $\mathbb{R}^2$, we can obtain two examples of $\lambda_1$-minimal flat $3$-torus in $\mathbb{S}^5$ and $\mathbb{S}^7$, whose volume is respectively $\frac{8\sqrt 3}{9}\pi^3$ and $\frac{16}{9}\pi^3$. Their matrix data can be stated as follows:
$$Q=\left(
\begin{array}{ccc}
 1 & 0 &0 \\
0 & 1 &0 \\
0 &0 & 1 \\
\end{array}
\right),~~~
Y=\left(
\begin{array}{ccc}
 1 & 0 & 0 \\
 0 & 1 & 0 \\
 0 & 0 & 1 \\
\end{array}
\right),~~~ (c_1^2,\, c_2^2,\, c_3^2)=\left(\frac{1}{3},\, \frac{1}{3},\, \frac{1}{3}\right),
$$
$$Q=\left(
\begin{array}{ccc}
 1 & -\frac{1}{2} & 0 \\
 -\frac{1}{2} & 1 & 0 \\
 0 & 0 & 1 \\
\end{array}
\right),~~~
Y^t=\left(
\begin{array}{cccc}
 1 & 0 & 1&0 \\
 0 & 1 & 1&0 \\
  0 & 0 & 0&1 \\
\end{array}
\right),~~~ (c_1^2,\, c_2^2,\, c_3^2,\, c_4^2)=\left(\frac{2}{9},\, \frac{2}{9},\,\frac{2}{9},\, \frac{1}{3}\right).
$$
\end{example}
%is a product of $T^1\times T^1$ while the other is not. As for $3$-torus, there are products of $T^1\times T^1\times T^1$ and $T^2\times T^1$.

Next we give some data, %irreducible examples  %of $\lambda_1$-minimal flat $3$-tori,
in which % of non-product structure, which will be called \emph{irreducible} and appear as
neither $Q$ nor $Y$ is block diagonal.  These data were obtained originally by applying our variational characterization, % introduced in the last section, 
where some tedious but %conventional 
routine computation is involved, and we omit it here. Alternatively,  
one can verify directly that all these data fit \eqref{eq-AQA} and \eqref{eq:flat}, and satisfy the criterion given in % which means $Q=Q_p$. The $\lambda_1$-minimality comes from
Theorem \ref{thm:special}, %or \ref{thm:general}.
hence provide irreducible examples of $\lambda_1$-minimal flat tori.
\begin{example}\label{ex:3-1}
	N=$4$. The following matrix data
	$$Q=\left(
	\begin{array}{ccc}
		1 & -\frac{1}{3} & -\frac{1}{3} \\
		-\frac{1}{3} & 1 & -\frac{1}{3} \\
		-\frac{1}{3} & -\frac{1}{3} & 1 \\
	\end{array}
	\right),~~~
	Y^t=\left(
	\begin{array}{cccc}
		1 & 0 & 0&1 \\
		0 & 1 & 0 &1\\
		0 & 0 & 1 &1\\
	\end{array}
	\right)
	,~~~(c_1^2,\, c_2^2,\, c_3^2,\, c_4^2)=\left(\frac{1}{4},\, \frac{1}{4},\,\frac{1}{4},\, \frac{1}{4}\right)
	$$
	gives a $\lambda_1$-minimal flat $3$-torus in $\mathbb{S}^7$, which  is irreducible, linearly full and has volume $2\pi^3$.   %$x: T^3=\mathbb{R}^3/\Lambda_3\rightarrow \mathbb{S}^7$ given by
	%is isometric and realized by the first eigenfunctions.
\end{example}
\begin{example}\label{ex:3-2}
	N=$5$. The following matrix data  %An immersion $x: T^3=\mathbb{R}^3/\Lambda_3\rightarrow \mathbb{S}^9$ is given by
	$$Q=\left(
	\begin{array}{ccc}
		1 & -\frac{1}{2} & -\frac{1}{4} \\
		-\frac{1}{2} & 1 & -\frac{1}{4} \\
		-\frac{1}{4} & -\frac{1}{4} & 1 \\
	\end{array}
	\right),~~~
	Y^t=\left(
	\begin{array}{ccccc}
		1 & 0 & 1&0&1 \\
		0 & 1 & 1&0&1\\
		0 & 0 & 0 &1&1\\
	\end{array}
	\right)
	,~~~(c_1^2,\, c_2^2,\, c_3^2,\, c_4^2,\, c_5^2)=\left(\frac{2}{9},\, \frac{2}{9},\,\frac{2}{9},\, \frac{1}{9},\,\frac{2}{9}\right)
	$$
	gives a $\lambda_1$-minimal flat $3$-torus in $\mathbb{S}^9$, which  is irreducible, linearly full and has volume $\frac{32\sqrt 3}{27}\pi^3$.
\end{example}
\begin{example}\label{ex:3-3}
	N=$6$. The following matrix data  % An immersion $x: T^3=\mathbb{R}^3/\Lambda_3\rightarrow \mathbb{S}^{11}$ is given by
	$$Q=\left(
	\begin{array}{ccc}
		1 & -\frac{1}{2} & 0 \\
		-\frac{1}{2} & 1 & -\frac{1}{2} \\
		0 & -\frac{1}{2} & 1 \\
	\end{array}
	\right),~
	Y^t=\left(
	\begin{array}{cccccc}
		1 & 0 & 1&0&0&1 \\
		0 & 1 & 1&0&1&1\\
		0 & 0 & 0 &1&1&1\\
	\end{array}
	\right)
	,~(c_1^2,\, c_2^2,\, c_3^2,\, c_4^2,\, c_5^2,\, c_6^2)=\left(\frac{1}{6},\, \frac{1}{6},\,\frac{1}{6},\, \frac{1}{6},\,\frac{1}{6},\,\frac{1}{6}\right).
	$$
	gives a $\lambda_1$-minimal flat $3$-torus in $\mathbb{S}^{11}$, which  is irreducible, linearly full and has volume $\frac{8\sqrt 6}{9}\pi^3$.
\end{example}
We will prove in the next section that Example~\ref{ex:3-prod} $\sim$ Example~\ref{ex:3-3} enumerate all examples of $\lambda_1$-minimal flat $3$-tori in spheres.
\begin{example}\label{ex:4-prod}
Taking the direct products (as in \eqref{eq-prod}) of two  $\lambda_1$-minimal flat $2$-tori, % from  the Clifford $2$-torus and the equilateral $2$-torus,
one can obtain three reducible   $\lambda_1$-minimal flat $4$-tori in $\mathbb{S}^7$, $\mathbb{S}^9$ and $\mathbb{S}^{11}$, they are all linearly full with volumes $\pi^4$, $\frac{3\sqrt 3}{4}\pi^4$ and $\frac{4}{3}\pi^4$, respectively.

Taking the direct products (as in \eqref{eq-prod}) of one $\lambda_1$-minimal flat $3$-tori given in Example~\ref{ex:3-1} $\sim$ Example~\ref{ex:3-3} with the unite circle in $\mathbb{R}^2$, we can obtain another three  reducible   $\lambda_1$-minimal flat $4$-tori in $\mathbb{S}^9$, $\mathbb{S}^{11}$ and $\mathbb{S}^{13}$, they are all linearly full with volumes  $\frac{2\sqrt 3}{3}\pi^4$, $\frac{4}{3}\pi^4$ and $\sqrt 2\pi^4$, respectively.

For brevity, we omit the data set of these $6$ reducible $\lambda_1$-minimal flat $4$-tori.
\end{example}
\begin{example}\label{ex:5-1}
	$N=5$. The following matrix data %An immersion $x: T^4=\mathbb{R}^4/\Lambda_4\rightarrow \mathbb{S}^9$ is given by
	$$Q=\left(
	\begin{array}{cccc}
		1 & -\frac{1}{4} & -\frac{1}{4} & -\frac{1}{4} \\
		-\frac{1}{4} & 1 & -\frac{1}{4} & -\frac{1}{4} \\
		-\frac{1}{4} & -\frac{1}{4} & 1 & -\frac{1}{4} \\
		-\frac{1}{4} & -\frac{1}{4} & -\frac{1}{4} & 1 \\
	\end{array}
	\right),~~~
	Y^t=\left(
	\begin{array}{ccccc}
		1 & 0 & 0 & 0 & 1 \\
		0 & 1 & 0 & 0 & 1 \\
		0 & 0 & 1 & 0 & 1 \\
		0 & 0 & 0 & 1 & 1 \\
	\end{array}
	\right)
	,$$$$(c_1^2,\, c_2^2,\, c_3^2,\, c_4^2,\,c_5^2)=\left(\frac{1}{5},\, \frac{1}{5},\,\frac{1}{5},\, \frac{1}{5},\,\frac{1}{5}\right)
	$$
    gives a $\lambda_1$-minimal flat $4$-torus in $\mathbb{S}^{9}$, which is irreducible, linearly full and has volume $\frac{16\sqrt 5}{25}\pi^4$.
\end{example}
\begin{example}\label{ex:6-1}
	$N=6$. The following matrix data %An immersion $x: T^4=\mathbb{R}^4/\Lambda_4\rightarrow \mathbb{S}^{11}$ is given by
	$$Q=\left(
	\begin{array}{cccc}
		1 & -\frac{1}{2} & -\frac{1}{6} & -\frac{1}{6} \\
		-\frac{1}{2} & 1 & -\frac{1}{6} & -\frac{1}{6} \\
		-\frac{1}{6} & -\frac{1}{6} & 1 & -\frac{1}{3} \\
		-\frac{1}{6} & -\frac{1}{6} & -\frac{1}{3} & 1 \\
	\end{array}
	\right)
	,~~~
	Y^t=\left(
	\begin{array}{cccccc}
		1 & 0 & 1 & 0 & 0 & 1 \\
		0 & 1 & 1 & 0 & 0 & 1 \\
		0 & 0 & 0 & 1 & 0 & 1 \\
		0 & 0 & 0 & 0 & 1 & 1 \\
	\end{array}
	\right)
	,$$$$(c_1^2,\, c_2^2,\, c_3^2,\, c_4^2,\, c_5^2,\, c_6^2)=\left(\frac{1}{6},\, \frac{1}{6},\,\frac{3}{16},\, \frac{5}{48},\,\frac{3}{16},\,\frac{3}{16}\right)
	$$
	gives a $\lambda_1$-minimal flat $4$-torus in $\mathbb{S}^{11}$, which is irreducible, linearly full and has volume $\frac{3}{2}\pi^4$.
\end{example}
\begin{example}\label{ex:7-1}
	$N=7$. The following matrix data %An immersion $x: T^4=\mathbb{R}^4/\Lambda_4\rightarrow \mathbb{S}^{13}$ is given by
	$$Q=\left(
	\begin{array}{cccc}
		1 & -\frac{1}{2} & -\frac{1}{4} & -\frac{1}{4} \\
		-\frac{1}{2} & 1 & -\frac{1}{4} & -\frac{1}{4} \\
		-\frac{1}{4} & -\frac{1}{4} & 1 & \frac{1}{4} \\
		-\frac{1}{4} & -\frac{1}{4} & \frac{1}{4} & 1 \\
	\end{array}
	\right)
	,~~~
	Y^t=\left(
	\begin{array}{ccccccc}
		1 & 0 & 1 & 0 & 1 & 0 & 1 \\
		0 & 1 & 1 & 0 & 1 & 0 & 1 \\
		0 & 0 & 0 & 1 & 1 & 0 & 0 \\
		0 & 0 & 0 & 0 & 0 & 1 & 1 \\
	\end{array}
	\right)
	,$$$$(c_1^2,\, c_2^2,\, c_3^2,\, c_4^2,\, c_5^2,\, c_6^2,\, c_7^2)=\left(\frac{1}{6},\, \frac{1}{6},\, 0,\,\frac{1}{6},\,\frac{1}{6},\,\frac{1}{6},\,\frac{1}{6}\right)
	$$
	gives a $\lambda_1$-minimal flat $4$-torus in $\mathbb{S}^{13}$, which is irreducible, and has volume $\frac{8\sqrt 3}{9}\pi^4$. Note that it is not linearly full in $\mathbb{S}^{13}$ but in $\mathbb{S}^{11}$. As far as we know, this is the first example of $\lambda_1$-minimal immersion in the literature that does not span the whole first eigenspace.%, seen as a totally geodesic submanifold of $\mathbb{S}^{13}$.
\end{example}
%This immersion is not linearly full in $\mathbb{S}^{13}$, i.e., it does not occupy all first eigenfunctions.
\begin{example}\label{ex:7-2}
	$N=7$. The following matrix data %An immersion $x: T^4=\mathbb{R}^4/\Lambda_4\rightarrow \mathbb{S}^{13}$ is given by
	$$Q=\left(
	\begin{array}{cccc}
		1 & -\frac{1}{2} & -\frac{1}{4} & -\frac{1}{8} \\
		-\frac{1}{2} & 1 & -\frac{1}{4} & -\frac{1}{8} \\
		-\frac{1}{4} & -\frac{1}{4} & 1 & -\frac{1}{4} \\
		-\frac{1}{8} & -\frac{1}{8} & -\frac{1}{4} & 1 \\
	\end{array}
	\right)
	,~~~
	Y^t=\left(
	\begin{array}{ccccccc}
		1 & 0 & 1 & 0 & 1 & 0 & 1 \\
		0 & 1 & 1 & 0 & 1 & 0 & 1 \\
		0 & 0 & 0 & 1 & 1 & 0 & 1 \\
		0 & 0 & 0 & 0 & 0 & 1 & 1 \\
	\end{array}
	\right)
	,$$$$(c_1^2,\, c_2^2,\, c_3^2,\, c_4^2,\, c_5^2,\, c_6^2,\, c_7^2)=\left(\frac{1}{6},\, \frac{1}{6},\,\frac{1}{6},\, \frac{1}{12},\,\frac{1}{12},\,\frac{1}{6},\,\frac{1}{6}\right)
	$$
	gives a $\lambda_1$-minimal flat $4$-torus in $\mathbb{S}^{13}$, which is irreducible, linearly full and has volume $\frac{8\sqrt 3}{9}\pi^4$.
\end{example}

\begin{example}\label{ex:7-3}
	$N=7$. The following matrix data % An immersion $x: T^4=\mathbb{R}^4/\Lambda_4\rightarrow \mathbb{S}^{13}$ given by
 $$Q=\left (
  \begin {array} {cccc} 1 &\frac {\sqrt {13} - 
                    7} {12} &\frac {\sqrt {13} - 
                    7} {12} &\frac {\sqrt {13} - 
                    4} {6} \\\frac {\sqrt {13} - 
                    7} {12} & 1 &\frac {1 - \sqrt {13}} {6} &\frac{\sqrt {13} - 7} {12} \\
                    \frac {\sqrt {13} - 
                    7} {12} &\frac {1 - \sqrt {13}} {6} & 1 &\frac{\sqrt {13} - 7} {12} \\\frac {\sqrt {13} - 
             4} {6} &\frac {\sqrt {13} - 7} {12} &\frac {\sqrt {13} - 
       7} {12} & 1 \\\end {array} \right)
	,~~~
	Y^t=
 \left(
\begin{array}{ccccccc}
 1 & 0 & 0 & 1 & 0 & 0 & 1 \\
 0 & 1 & 0 & 1 & 0 & 1 & 1 \\
 0 & 0 & 1 & 1 & 0 & 1 & 1 \\
 0 & 0 & 0 & 0 & 1 & 1 & 1 \\
\end{array}
\right)
	,$$$$(c_1^2,\, c_2^2,\, c_3^2,\, c_4^2,\, c_5^2,\, c_6^2,\, c_7^2)=\left(\frac{\sqrt{13}-2}{12},\,\frac{5-\sqrt{13}}{8},\,\frac{5-\sqrt{13}}{8},\,\frac{\sqrt{13}-2}{12},\, \frac{\sqrt{13}-2}{12},\,\frac{\sqrt{13}-2}{12},\,\frac{5-\sqrt{13}}{12}\right)
	$$
 \iffalse
	$$Q=\left(
	\begin{array}{cccc}
		1 & -\frac{1}{2} & \frac{\sqrt{13}-4}{6} & \frac{1-\sqrt{13}}{12}\\
		-\frac{1}{2} & 1 & \frac{1-\sqrt{13}}{6} & \frac{\sqrt{13}-7}{12}\\
		\frac{\sqrt{13}-4}{6} & \frac{1-\sqrt{13}}{6} & 1 & \frac{\sqrt{13}-7}{12}\\
		\frac{7-\sqrt{13}}{12} & \frac{\sqrt{13}-7}{12} & \frac{\sqrt{13}-7}{12}& 1 \\
	\end{array}
	\right)
	,~~~
	Y^t=
 \left(
	\begin{array}{ccccccc}
		1 & 0 & 1 & 0 & 1 & 0 & 0 \\
		0 & 1 & 1 & 0 & 1 & 0 & 1 \\
		0 & 0 & 0 & 1 & 1 & 0 & 1 \\
		0 & 0 & 0 & 0 & 0 & 1 & 1 \\
	\end{array}
	\right)
	,$$$$(c_1^2,\, c_2^2,\, c_3^2,\, c_4^2,\, c_5^2,\, c_6^2,\, c_7^2)=\left(\frac{\sqrt{13}-2}{12},\,\frac{\sqrt{13}-2}{12},\,\frac{5-\sqrt{13}}{12},\, \frac{\sqrt{13}-2}{12},\,\frac{\sqrt{13}-2}{12},\,\frac{5-\sqrt{13}}{8},\,\frac{5-\sqrt{13}}{8}\right)
	$$
 \fi
	gives a $\lambda_1$-minimal flat $4$-torus in $\mathbb{S}^{13}$, which is irreducible, linearly full and has volume $\frac{\sqrt {26\sqrt{13}-70}}{3}\pi^4$.  %$\frac{216}{35+13\sqrt{13}}$.
 
	As far as we know, this is the first example of $\lambda_1$-minimal immersion in the literature whose %volume and
	Gram matrix are not rational, comparing to that minimal tori always have rational Gram matrices up to a rescaling(see \cite{Bryant}).
\end{example}
%This immersion is the first example whose Gram matrix is not rational. 
\iffalse
\begin{remark}
Suppose $\mathbb{R}^2/\Lambda_2$ is a flat  $2$-torus with $Q$ be a Gram matrix of $\Lambda_2^*$. It follows from the Bryant's characterization \cite{Bryant} that this  $2$-torus  can be immersed minimally 
 into spheres if and only if up to a rescaling, the entries of $Q$ are all rational numbers.   
\end{remark}
\fi
\begin{example}\label{ex:8-1}
	$N=8$. The following matrix data %An immersion $x: T^4=\mathbb{R}^4/\Lambda_4\rightarrow \mathbb{S}^{15}$ is given by
	%\begin{equation*}
	%\begin{split}
	$$Q=\left(
	\begin{array}{cccc}
		1 & -\frac{1}{2} & 0 & -\frac{1}{4} \\
		-\frac{1}{2} & 1 & -\frac{1}{2} & 0 \\
		0 & -\frac{1}{2} & 1 & -\frac{1}{4} \\
		-\frac{1}{4} & 0 & -\frac{1}{4} & 1 \\
	\end{array}
	\right)
	,~~~
	Y^t=\left(
	\begin{array}{cccccccc}
		1 & 0 & 1 & 0 & 0 & 1 & 0 & 1 \\
		0 & 1 & 1 & 0 & 1 & 1 & 0 & 1 \\
		0 & 0 & 0 & 1 & 1 & 1 & 0 & 1 \\
		0 & 0 & 0 & 0 & 0 & 0 & 1 & 1 \\
	\end{array}
	\right)
	,$$
	$$
	(c_1^2,\, c_2^2,\, c_3^2,\, c_4^2,\, c_5^2,\, c_6^2,\, c_7^2,\, c_8^2)=\left(\frac{1}{8},\, \frac{1}{8},\,\frac{1}{8},\, \frac{1}{8},\,\frac{1}{24},\,\frac{1}{8},\,\frac{1}{6},\,\frac{1}{6}\right)
	$$
	gives a $\lambda_1$-minimal flat $4$-torus in $\mathbb{S}^{15}$, which is irreducible, linearly full and has volume $\frac{2\sqrt  6}{3}\pi^4$.
\end{example}

\begin{example}\label{ex:8-2}
	$N=8$. The following matrix data %An immersion $x: T^4=\mathbb{R}^4/\Lambda_4\rightarrow \mathbb{S}^{15}$ is given by
	$$Q=\left(
	\begin{array}{cccc}
		1 & -\frac{1}{2} & \frac{\sqrt{3}}{2}-1 & 1-\frac{\sqrt{3}}{2} \\
		-\frac{1}{2} & 1 & \frac{1}{2}-\frac{\sqrt{3}}{2} & \sqrt{3}-2 \\
		\frac{\sqrt{3}}{2}-1 & \frac{1}{2}-\frac{\sqrt{3}}{2} & 1 & \frac{1}{2}-\frac{\sqrt{3}}{2} \\
		1-\frac{\sqrt{3}}{2} & \sqrt{3}-2 & \frac{1}{2}-\frac{\sqrt{3}}{2} & 1 \\
	\end{array}
	\right)
	,~~~
	Y^t=\left(
	\begin{array}{cccccccc}
		1 & 0 & 1 & 0 & 1 & 0 & 0 & 1 \\
		0 & 1 & 1 & 0 & 1 & 0 & 1 & 1 \\
		0 & 0 & 0 & 1 & 1 & 0 & 1 & 1 \\
		0 & 0 & 0 & 0 & 0 & 1 & 1 & 1 \\
	\end{array}
	\right)
	,$$
	$$
	(c_1^2,\; c_2^2,\; c_3^2,\; c_4^2,\; c_5^2,\; c_6^2,\; c_7^2,\; c_8^2)=\left(\frac{3-\sqrt{3}}{12},\; \frac{\sqrt{3}}{12},\;\frac{\sqrt{3}}{12},\; \frac{3-\sqrt{3}}{12},\;\frac{3-\sqrt{3}}{12},\;\frac{3-\sqrt{3}}{12},\;\frac{\sqrt{3}}{12},\;\frac{\sqrt{3}}{12}\right)
	$$
	gives a $\lambda_1$-minimal flat $4$-torus in $\mathbb{S}^{15}$, which is irreducible, linearly full and has volume $\frac{\sqrt {12+8 \sqrt{3}}}{3}\pi^4$.
	This is another example of $\lambda_1$-minimal immersion whose %volume and
	Gram matrix are not rational.
\end{example}

\begin{example}\label{ex:9-1}
	$N=9$. The following matrix data % An immersion $x: T^4=\mathbb{R}^4/\Lambda_4\rightarrow \mathbb{S}^{17}$ is given by
	$$Q=\left(
	\begin{array}{cccc}
		1 & -\frac{1}{2} & 0 & -\frac{1}{6} \\
		-\frac{1}{2} & 1 & -\frac{1}{2} & \frac{1}{6} \\
		0 & -\frac{1}{2} & 1 & -\frac{1}{2} \\
		-\frac{1}{6} & \frac{1}{6} & -\frac{1}{2} & 1 \\
	\end{array}
	\right)
	,~~~
	Y^t=\left(
	\begin{array}{ccccccccc}
		1 & 0 & 1 & 0 & 0 & 1 & 0 & 0 & 1 \\
		0 & 1 & 1 & 0 & 1 & 1 & 0 & 0 & 1 \\
		0 & 0 & 0 & 1 & 1 & 1 & 0 & 1 & 1 \\
		0 & 0 & 0 & 0 & 0 & 0 & 1 & 1 & 1 \\
	\end{array}
	\right)
	,$$
	$$(c_1^2,\; c_2^2,\; c_3^2,\; c_4^2,\; c_5^2,\; c_6^2,\; c_7^2,\; c_8^2,\; c_9^2)=\left(\frac{1}{8},\; \frac{1}{8},\;\frac{1}{12},\; \frac{1}{12},\;\frac{1}{8},\;\frac{1}{12},\;\frac{1}{8},\;\frac{1}{8},\;\frac{1}{8}\right)
	$$
	gives a $\lambda_1$-minimal flat $4$-torus in $\mathbb{S}^{17}$, which is irreducible, linearly full and has volume $\frac{16}{9}\pi^4$.
\end{example}

\begin{example}\label{ex:9-2}
	$N=9$.  The following matrix data %An immersion $x: T^4=\mathbb{R}^4/\Lambda_4\rightarrow \mathbb{S}^{17}$ is given by
	$$Q=\left(
	\begin{array}{cccc}
		1 & -\frac{1}{2} & -\frac{1}{4} & \frac{1}{4} \\
		-\frac{1}{2} & 1 & -\frac{1}{4} & -\frac{1}{2} \\
		-\frac{1}{4} & -\frac{1}{4} & 1 & -\frac{1}{4} \\
		\frac{1}{4} & -\frac{1}{2} & -\frac{1}{4} & 1 \\
	\end{array}
	\right)
	,~~~
	Y^t=\left(
	\begin{array}{ccccccccc}
		1 & 0 & 1 & 0 & 1 & 0 & 0 & 1 &0\\
		0 & 1 & 1 & 0 & 1 & 0 & 1 & 1 &1\\
		0 & 0 & 0 & 1 & 1 & 0 & 1 & 1 &0\\
		0 & 0 & 0 & 0 & 0 & 1 & 1 & 1 &1\\
	\end{array}
	\right)
	,$$
	$$(c_1^2,\; c_2^2,\; c_3^2,\; c_4^2,\; c_5^2,\; c_6^2,\; c_7^2,\; c_8^2,\; c_9^2)=\left(\frac{1}{9},\; \frac{1}{9},\;\frac{1}{9},\; \frac{1}{9},\;\frac{1}{9},\;\frac{1}{9},\;\frac{1}{9},\;\frac{1}{9},\;\frac{1}{9}\right)
	$$	
	gives a $\lambda_1$-minimal flat $4$-torus in $\mathbb{S}^{17}$, which is irreducible, linearly full and has volume $\sqrt  3\pi^4$.
\end{example}

\begin{example}\label{ex:10}
	$N=10$. The following matrix data %An immersion $x: T^4=\mathbb{R}^4/\Lambda_4\rightarrow \mathbb{S}^{19}$ is given by
	$$Q=\left(
	\begin{array}{cccc}
		1 & -\frac{1}{2} & 0 & 0 \\
		-\frac{1}{2} & 1 & -\frac{1}{2} & 0 \\
		0 & -\frac{1}{2} & 1 & -\frac{1}{2} \\
		0 & 0 & -\frac{1}{2} & 1 \\
	\end{array}
	\right)
	,~~~
	Y^t=\left(
	\begin{array}{cccccccccc}
		1 & 0 & 1 & 0 & 0 & 1 & 0 & 0 & 0 & 1 \\
		0 & 1 & 1 & 0 & 1 & 1 & 0 & 0 & 1 & 1 \\
		0 & 0 & 0 & 1 & 1 & 1 & 0 & 1 & 1 & 1 \\
		0 & 0 & 0 & 0 & 0 & 0 & 1 & 1 & 1 & 1 \\
	\end{array}
	\right)
	,$$
	$$(c_1^2,\; c_2^2,\; c_3^2,\; c_4^2,\; c_5^2,\; c_6^2,\; c_7^2,\; c_8^2,\; c_9^2,\;c_{10}^2)=\left(\frac{1}{10},\; \frac{1}{10},\;\frac{1}{10},\; \frac{1}{10},\;\frac{1}{10},\;\frac{1}{10},\;\frac{1}{10},\;\frac{1}{10},\;\frac{1}{10},\;\frac{1}{10}\right)
	$$
	gives a $\lambda_1$-minimal flat $4$-torus in $\mathbb{S}^{19}$, which is irreducible, linearly full and has volume $\frac{4}{\sqrt 5}\pi^4$.
\end{example}
We will prove in the next section that Example~\ref{ex:4-prod} $\sim$ Example~\ref{ex:10}, and Example~\ref{ex-4tori}  enumerate all examples of $\lambda_1$-minimal flat $4$-tori in spheres.
\section{Shortest vectors in lattices}\label{sec-lattice}

As mentioned in the introduction, %pointed out in Section~\ref{sec3}, by the theorem of Soufi and Ilias \cite{Sou-Ili}, we just need to consider the isometric immersions of flat tori %$T^3=\mathbb{R}^3/\Lambda_3$ in spheres %$\mathbb{S}^p$ by the first eigenfunctions.
the classification of $\lambda_1$-minimal tori of dimension $3$ and $4$ in spheres relies on deep investigation of shortest lattice vectors for lattices of rank $3$ and $4$. This section is devoted to the discussion of some related properties of lattices, which are of independent  interest.

%{\bf Notations:} In this section we set $\Lambda^*_n$ to be a lattice of rank $n$, $N$ to be the number of distinct shortest lattice vectors of $\Lambda^*_n$ up to $\pm1$, $\Xi$ to be the set of these $N$ shortest vectors, which is also assumed to be of rank $n$ and denoted by
%$$\{\xi_1, \xi_2, \cdots,\xi_N\}.$$
%Without loss of generality, we also assume that these vectors are of unit length.

Let $\Lambda^*_n$ be a lattice of rank $n$,  %In the sequel, we always denote by 
$N$ be the number of distinct shortest lattice vectors of $\Lambda^*_n$ up to $\pm1$,  %determined by $Q$
%up to $\pm1$,
%i.e., the half dimension of the first eigenspace of flat torus $T^n=\mathbb{R}^n/\Lambda_n$.
%Without loss of generality, we assume this shortest length is %they are of length
%$1$. %It is well known that $N\geq n$ (see Corollary 3.4 in \cite{doCarmo}).
%Let
and $\Xi$ %$\triangleq\{\xi_j\}_{j=1}^{N}$ %$\{\xi_j\}_{j=1}^{N}$
be the set of these $N$ shortest vectors, which are denoted by
$$\xi_1, \xi_2, \cdots,\xi_N.$$
%$$\Xi\triangleq\{\xi_j\}_{j=1}^{N}.$$
For simplicity, in the sequel, these shortest vectors are assumed to be of unit length. %we always assume that the shortest length
%{\bf Convention:} In the sequel, we always assume that the shortest length of $\Lambda_n^*$ is $1$, and $\text{rank}(\Xi)=n$. 
\iffalse
First of all, given two $\lambda_1$-minimal tori $f_i:T^{n_i}\to \mathbb{S}^{m_i} (i=1,2)$, one can construct a new $\lambda_1$-minimal torus by direct product
\[
f=\left(\sqrt{\frac{n_1}{n_1+n_2}}f_1,\sqrt{\frac{n_2}{n_1+n_2}}f_2\right):T^{n_1}\times T^{n_2}\to \mathbb{S}^{m_1+m_2+1}.
\]
For $\lambda_1$-minimal $2$-torus, the Clifford torus is a product of $T^1\times T^1$ while the other is not. As for $3$-torus, there are products of $T^1\times T^1\times T^1$ and $T^2\times T^1$. So we focus on the $\lambda_1$-minimal $3$-tori of non-product structure, which will be called \emph{irreducible} and appear as neither $Q$ nor $Y$  diagonal in blocks.
\fi

%Before discussing the classification of $\lambda_1$-minimal flat tori, we present some properties of lattices, which are interesting in themselves.
\begin{definition}
A set of $k+1$ vectors in $\mathbb{R}^n$ is called a %\emph
{\bf \em generic $(k+1)$-tuple} if it is of rank $k$ and any $k$ vectors in it are linearly independent.
\end{definition}
\begin{definition}
    A lattice $\Lambda_n^*$ with $N\geq n$ %of rank $n$
    is called {\bf \em s-reducible} %\emph {\bf \em prime}
    if there is a non-trivial decomposition $\mathbb{R}^n=V_1\oplus V_2$ such that $$%\Xi\cap V_1\neq \emptyset,~~~\Xi\cap V_2\neq \emptyset,~~~
    \Xi=(\Xi\cap V_1)\cup (\Xi\cap V_2).$$ It will be called {\bf\em s-irreducible} otherwise.   %$\text{rank}(\Xi)=n$, and any $n$ linearly independent lattice vectors in $\Xi$ form generators of this lattice.
\end{definition}
%\begin{remark}
%To make this definition make sense, we always assume that $\Lambda_n^*$ is a lattice with $N\geq n$ in the following. 
It is easy to see that %for a lattice $\Lambda_n^*$ of rank $n$, if $N\leq n$, then $\Lambda_n^*$ is always s-reducible. Moreover, 
if $N=n$, or $N=n+1$ and $\Xi$ is not a generic $(n+1)$-tuple, then $\Lambda_n^*$ is s-reducible.
%\end{remark}
%It is equivalent to say
\iffalse
Note that $k+1$ vectors $\xi_1,\cdots,\xi_{k+1}$ form  a generic $(k+1)$-tuple if and only if it is of rank $k$ and
\[
\xi_{k+1}=a_1\xi_1+\cdots+a_k\xi_k,\quad a_i\not=0.
\]
\fi
\begin{lemma}\label{lem:irreducible}
For an s-irreducible lattice $\Lambda_n^*$ of rank $n\geq 3$, 
%if $N\geq n$, then 
%with $n\geq 3$,
%For an irreducible minimal torus $T^n$ $(n\geq 3)$,
there always exists a generic $k$-tuple in $\Xi$ for some $k\geq 4$.
\end{lemma}
\begin{proof}
 Let  $\{\xi_1,\cdots,\xi_n\}\subset \Xi$ be a generator of $\Lambda^*_n$. Suppose the opposite that %If 
 there are only generic $3$-tuples %allowed 
 in $\Xi$. Then for any $n+1\leq i\leq N$, when we write $\xi_i$ as a linear combination of these generator vectors, there are exactly two non-zero coefficients. From the s-irreducible assumption we know $N\geq n+2$. %， and any $\xi_i \,(i>n)$ is a linear combination of exactly two vectors of generators. 
 We can assume $\xi_{n+1}=a_1\xi_1+a_2\xi_2$  with $a_1a_2\not=0$. Consider two subspaces $V_1\triangleq\mathrm{Span}\{\xi_3,\cdots,\xi_n\}$ and $V_2\triangleq \mathrm{Span}\{\xi_1,\xi_2\}$, it follows that there exists at least one vector, say $\xi_{n+2}$, belongs to neither $V_1$ nor $V_2$. %Considering s-irreducible assumption again, we know there exists at least one vector, say $\xi_{n+2}$, can be contained in neither $V_1=\mathrm{Span}\{\xi_3,\cdots,\xi_n\}$ nor $V_2=F\mathrm{Span}\{\xi_1,\xi_2\}$. 
 %So 
 We assume  $\xi_{n+2}=b_1\xi_1+b_2\xi_3$ with $b_1b_2\not=0$. Then it is straightforward to verify that  $\{\xi_2,\xi_3,\xi_{n+1},\xi_{n+2}\}$ forms a generic $4$-tuple, which gives us a contradiction.
\end{proof}

\begin{definition}\label{def:prime}
	%Suppose a lattice of rank $n$ has shortest length $1$ and $\Xi$ defined as above, it
	
	A lattice $\Lambda_n^*$ %of rank $n$ 
 is called {\bf \em prime} if %$\emph{ rank}(\Xi)=n$, and 
 %any $n$ linearly independent vectors in $\Xi$ form a generator of this lattice.
 this lattice can be generated by any $n$ linearly independent vectors in $\Xi$. 
\end{definition}
\begin{remark}\label{rk-pruni}
%It follows from Lemma~\ref{lem:prime} that
By definition, for a prime lattice, $\Xi$ always satisfies the unimodular condition defined in Definition~\ref{def-unimo}.
\end{remark}
The following conclusion is easy to obtain.
\begin{lemma}\label{lem:prime}
	Let $\Lambda_n^*$ be a prime lattice generated by %, and
	$\{\alpha_1,\cdots,\alpha_n\}$. %be a generator of $\Lambda_n^*$. %then any $n$ columns in $Y^t$ has determinant $0$ or $\pm 1$. Moreover,
	Then for any $\xi\in \Xi$, we have
	\[
	\xi=a_1\alpha_1+\cdots+a_n\alpha_n,\quad a_i\in \{0,\pm 1\}.
	\]
\end{lemma}
For lattices of lower rank ($\leq 4$), we obtain a sufficient condition to discriminate whether it is prime. The covering radius in lattice theory will be used, which is defined for a lattice $\Lambda_n^*$ as follows:
$$\mu(\Lambda_n^*)\triangleq \inf\Big{\{}r\,\Big{|}\,\mathbb{R}^n= \cup_{p\in\Lambda_n^*} B^n(p,r)\Big{\}},$$
where $B^n(p,r)$ is the ball centered at $p$ with radius $r$. Any ball with radius larger than $\mu(\Lambda^*_
n)$ must contain a point of $\Lambda_n^*$.
The following well-known estimate (see \cite{Gruber}) will be used to prove our main %the next 
theorem in this section.
\begin{lemma}\label{lem-covr}
Let  $\Lambda_n^*$ be a lattice of rank $n$, if ${\rm rank}(\Xi)=n$ and the shortest vectors of   $\Lambda_n^*$ are of unit length, 
 % Since ${\rm rank}(\Xi)=n$ and the shortest vectors in $\Lambda_n^*$ have length $1$, %there exist at least $n$ linearly independent vectors attaining this length.
  then
	%we have
 $\mu(\Lambda_n^*)\leq\frac{{\sqrt{n}}}{{2}}.$
\end{lemma}
%Then we have a necessary condition of %$\lambda_1$-minimality for low %dimensions, which will be used latter.
\begin{theorem}\label{thm-gen}
	Suppose $\Lambda_n^*$ is a lattice of rank $n$ with ${\rm rank}(\Xi)$=n. %which has at least $n$ linearly independent shortest lattice points.  %vectors attaining the shortest length in $n$-dimensional lattice $\Lambda_n^*$.
	If $n\leq 4$, then either $\Lambda_n^*$ is prime, or it can be generated by %except the case which is $\Lambda_4^*$ spanned by
	the row vectors of
\begin{equation}\label{eq:exception}
		\bgm
	1&0&0&0\\
	0&1&0&0\\
	0&0&1&0\\
	\frac{1}{2}&\frac{1}{2}&\frac{1}{2}&\frac{1}{2}
	\edm.
\end{equation}
\end{theorem}
\begin{proof} It is obviously true for $n=1$. Suppose it is true for $n=k-1$ ($\leq 3$), we will show that in the case of $n=k$ the conclusion is also true. %for when $n=k$,
Let $\{\xi_1, \xi_2,\cdots,\xi_k\}$ be any given linearly independent shortest vectors in $\Lambda_k^*$. Consider the sublattice $\Lambda_{k-1}^*$ of rank $k-1$ containing $\{\xi_2,\cdots,\xi_k\}$.
%Since $\xi_2,\cdots,\xi_k$ are linearly independent in some $k-1$ dimensional lattice $\Lambda_{k-1}^*$, they must be the basis of $\Lambda_{k-1}^*$.
Then it follows from the inductive hypothesis that there exists a generator $\{\et_1,\xi_2,\cdots,\xi_k\}$ of $\Lambda_k^*$ such that
	\[
	\begin{pmatrix}
		\xi_1\\
		\xi_2\\\vdots\\
		\xi_k
	\end{pmatrix}=\begin{pmatrix}
	a_{1}&a_{2}&\cdots&a_{k}\\
	&1&&\\
	&&\ddots&\\
	&&&1
\end{pmatrix}\begin{pmatrix}
\et_1\\
\xi_2\\\vdots\\
\xi_k
\end{pmatrix},
	\]
where $a_1, a_2, \cdots, a_k\in \mathbb{Z}$,  we can assume they are all non-negative and $a_1>0$ by changing directions of these vectors. Since ${\et_1+c\xi_2,\xi_2,\cdots,\xi_k}$ $(c\in \mathbb{Z})$ is still a a generator of $\Lambda_k^*$, we get
\[
\xi_1=a_1(\et_1+c\xi_2)+(a_2-ca_1)\xi_2+\cdots+a_k\xi_k.
\]
So suitable value $c$ will make $a_1>a_2-ca_1\geq 0$. Therefore,
 we may assume $a_1>a_i\geq 0$ for $2\leq i\leq k$.

If $a_1=1$ then $a_i=0$ for $2\leq i\leq k$ and we derive that ${\xi_1, \xi_2,\cdots,\xi_k}$ is a generator of $\Lambda_k^*$. If $a_1\geq 2$,
let $\et_1=\eta_1^{\perp}+\eta_1^{\top}$, where $\top$ (w.r.t. $\perp$) % $\eta_1^{\top}$ is
denotes the orthogonal projection %of $\eta$
onto $\Pi_0\triangleq\mathrm{Span}_\mathbb{R}\{\xi_2,\cdots,\xi_k\}$ (w.r.t. the normal space of $\Pi_0$ %$\mathrm{Span}_\mathbb{R}\{\xi_2,\cdots,\xi_k\}$
). So $\xi_1=a_1\eta_1^{\perp}+\xi_1^{\top}$ and
\[0<
|\eta_1^{\perp}|^2=\frac{|\xi_1|^2-|\xi_1^{\top}|^2}{a_1^2}=\frac{1-|\xi_1^{\top}|^2}{a_1^2}\leq \frac{1}{4}.
\]
It follows that the intersection of $k$-ball $B^k(0,1)$ with the hyperplane $\Pi\triangleq\eta_1^\perp+\Pi_0%\mathrm{Span}_\mathbb{R}\{\Lambda_{k-1}^*\}
$ is a $(k-1)$-ball $B^{k-1}(\eta_1^\perp, r)$ with
$$r^2=\frac{a_1^2-1}{a_1^2}+\frac{|\xi_1^\top|^2}{a_1^2}\geq\frac{3}{4}.$$

When $k<4$,  we can see $r$ is larger than the covering radius of $\Lambda_{k-1}^*$. Therefore, in $B^{k-1}(\eta_1^\perp, r)\subset B^k(0,1)$, there exists at least one  point of %sub-lattice
$\Lambda_k^*\cap\Pi$. %, whose distance to $0$ will be less than $1$.
This contradicts with our assumption that $1$ is the shortest length of $\Lambda_k^*$. %we have $\frac{\sqrt{3}}{2}>\frac{\sqrt{k-1}}{2}$. By Lemma~\ref{lem-covr} in $B(\eta_1^\perp, r)$ there exists at least one point of sub-lattice $\Lambda_k^*\cap\Pi$, whose distance to $0$ will be less than $1$.
Similarly, in the case of $k=4$, if $r^2>\frac{3}{4}$, we can also obtain a contradiction.

Therefore, when $k=4$, we have $r^2=\frac{3}{4}$. This yields $a_1=2$ and $\xi_1^\top=0$. Moreover, $a_2=a_3=a_4=1$. Otherwise, say $a_4=0$, then
\[
\begin{pmatrix}
	\xi_1\\
	\xi_2\\
	\xi_3
\end{pmatrix}=\begin{pmatrix}
2&1&1\\
&1&\\
&&1
\end{pmatrix}\begin{pmatrix}
\et_1\\
\xi_2\\
\xi_3
\end{pmatrix},
\]
 which implies in the sublattice $\mathrm{Span}_\mathbb{Z}\{\eta_1, \xi_2, \xi_3\}$, %$3$-dimensional lattice, where linearly independent
 the shortest vectors $\xi_1,\xi_2,\xi_3$ can not form a generator,  %can't form ge but have shortest length.
  which contradicts with the induction hypothesis. %we will get a contradiction

Note that $\xi_1^\top=0$ means $\xi_1$ is orthogonal to other $\xi_i$. It follows from $2\et_1=\xi_1-\xi_2-\xi_3-\xi_4$ and $|\eta_1|\geq 1$ that  $|\xi_2+\xi_3+\xi_4|\geq\sqrt{3}$.  Moreover, using
\[
\left|\frac{\xi_1+\xi_2-\xi_3-\xi_4}{2}\right|=|\eta_1+\xi_2|\geq 1,
\]
we have $|\xi_2-\xi_3-\xi_4|\geq \sqrt{3}$. Similarly, $|\xi_2-\xi_3+\xi_4|\geq \sqrt{3}$ and $|\xi_2+\xi_3-\xi_4|\geq \sqrt{3}$. %In the mean time,
Combining these with the following identities,
\[
|\xi_2+\xi_3+\xi_4|^2+|\xi_2-\xi_3-\xi_4|^2+|\xi_2-\xi_3+\xi_4|^2+|\xi_2+\xi_3-\xi_4|^2=4(|\xi_2|^2+|\xi_3|^2+|\xi_4|^2)=12,
\]
we can derive that
\[
|\xi_2+\xi_3+\xi_4|^2=|\xi_2-\xi_3-\xi_4|^2=|\xi_2-\xi_3+\xi_4|^2=|\xi_2+\xi_3-\xi_4|^2=3
\]
which implies $\{\xi_1, \xi_2, \xi_3, \xi_4\}$ forms an orthonormal basis of $\mathbb{R}^4$ and we complete the proof of this theorem.
\end{proof}

%In the rest part, we assume $\Lambda_n^*$ is prime, i.e., it can be generated by any $n$ linearly independent vectors in $\Xi$, which always exist since $x$ is an immersion.

Given a lattice $\Lambda_n^*$ of rank $n$, to investigate the set $\Xi$ constituted by all shortest lattice vectors up to $\pm1$, we can consider the intersections of it with all sublattices of rank $n-1$. % there is a maximal number for the points of  intersection between the sub-lattice and $\Xi$. If  $\Xi'\subset\Xi$ is such intersection set, we define $m(\Xi):=\sharp(\Xi')$.
Let %$\Xi'\subset\Xi$ such that
\begin{equation}\label{eq-mxi}
m(\Xi)%:=\sharp(\Xi')
\triangleq\max \Big\{\sharp(\Xi\cap \tilde\Lambda)\,\Big|\,\tilde\Lambda\subset\Lambda_n^* \text{~is a sublattice of rank}~ n-1\Big\},
\end{equation}
and $\Xi'$ be one of the intersection attaining $m(\Xi)$.
%Assume $\Lambda_{n-1}^*$ is a sub-lattice attaining this maximal value, and set $\Xi'=\Xi\cap \Lambda_{n-1}^*$.
Note that there may be more than one such intersections.  %$m(\Xi)$ is uniquely determined while
%the choice of $\Lambda_{n-1}^*$ and $\Xi'$ may not be uniquely determined.

In the rest part of this section, we always assume $\Lambda_n^*$ is a prime lattice of rank $n$. For such lattice, we have $m(\Xi)\geq n-1$. Let $\Xi'$ be a chosen intersection attaining $m(\Xi)$,
we assume that
$$\Xi'=\{\xi_1,\xi_2,\cdots,\xi_{m(\Xi)}\},$$
and $\Lambda_n^{*}$ is generated by $\{\xi_1, \xi_2,\cdots,\xi_{n-1}, \xi_{m(\Xi)+1}\}$.
Then there exist $N$ integer vectors $$\{A_i=(a_{i1}, a_{i2}, \cdots,a_{in})\}_{i=1}^N\subset \mathbb{Z}^n,$$ such that
\beq\label{eq:canonical}
\bgm
\xi_1\\
\xi_2\\
\vdots\\
\xi_N
\edm=\bgm
A_1\\
A_2\\
\vdots\\
A_N
\edm \bgm
\xi_1\\
\xi_2\\
\vdots\\
\xi_{n-1}\\
\xi_{m(\Xi)+1}
\edm.
\eeq
Obviously, $\{A_1,A_2,\cdots,A_{n-1},A_{m(\Xi)+1}\}$ is the standard basis of $\mathbb{R}^n$. It follows from the prime assumption and Lemma~\ref{lem:prime} that $$a_{ji}\in \{0,\pm1\},~~~1\leq j\leq N, 1\leq i\leq n.$$
Moreover, by changing the direction of some vectors in $\Xi$ if necessary, we can assume the last coordinate of $A_i$ equals $0$ for $1\leq i\leq m(\Xi)$, and $1$ for $m(\Xi)+1\leq i\leq N$. % if $\Lambda_n^*$ is prime.

\iffalse
Among all the $(n-1)$-dimensional sub-lattice of $\Lambda_n^*$, there is a maximal number for the points of  intersection between the sub-lattice and $\Xi$. If  $\Xi'\subset\Xi$ is such intersection set, we define $m(\Xi):=\sharp(\Xi')$.
Note that $m(\Xi)$ is uniquely determined while the choice of $\Xi'$ may not.
For any given $\Xi'$ we set
$$\Xi'=\{\xi_1,\xi_2,\cdots,\xi_{m(\Xi)}\},$$
and $\Lambda_n^{*}$ is generated by $\{\xi_1, \xi_2,\cdots,\xi_{n-1}, \xi_{m(\Xi)+1}\}$.
So
\beq\label{eq:canonical}
\bgm
\xi_1\\
\xi_2\\
\vdots\\
\xi_N
\edm=\bgm
A_1\\
A_2\\
\vdots\\
A_N
\edm \bgm
\xi_1\\
\xi_2\\\vdots\\
\xi_{m(\Xi)+1}
\edm.
\eeq
Moreover, by changing the directions of some vectors in $\Xi$ if necessary, we can assume the last entry of $A_i$ is $0$ ($ i\leq m(\Xi)$) or $1$ ($ i >m(\Xi)$).
\fi

We will still abuse the notation $Y$ to denote either the set constituted by $A_i$, or the matrix constructed by $A_i$ as in \eqref{eq:canonical}.  %Let $Q$ be the Gram matrix of $\Lambda_n^*$ under the generators $\{\xi_1, \xi_2,\cdots,\xi_{n-1}, \xi_{m(\Xi)+1}\}$.
%\subsection{Three Lemmas on prime lattices}
Note that the prime assumption on $\Xi\subset \Lambda_n^*$ is equivalent to say that any $n$ linearly independent vectors in $Y$ %\subset \mathbb{Z}^n$
form a generator of $\mathbb{Z}^n$.
For applications in the next section, we conclude some further properties of prime lattices in the following three lemmas.
\begin{lemma}\label{lem-0}
For a prime lattice, in terms of matrix, %this implies
all minors of $Y$ can only take values in $\{0, \pm1\}$.
\end{lemma}
\begin{proof}
Suppose the opposite that there is a nonzero minor involving the $i_1\text{-th}, i_2\text{-th}, \cdots, i_k$-th rows and $j_1\text{-th}, j_2\text{-th}, \cdots, j_k$-th columns that is not equal to $\pm 1$, then $\mathbb{Z}^n$ can not be generated by $\{A_{j_1}, A_{j_2}, \cdots,A_{j_k}\}$ and $\{A_{i_{k+1}},A_{i_{k+2}}, \cdots, A_{i_{n}}\}$,  since 
$$|A_{j_1}\wedge\cdots\wedge A_{j_k}\wedge A_{i_{k+1}}\wedge\cdots\wedge  A_{i_n}|>1,$$ where $\{i_{k+1}, \cdots,i_n\}$ is the complementary set of $\{i_1, \cdots, i_k\}$ in $\{1,2,\cdots,n-1,m(\Xi)+1\}$. This gives a contradiction with our prime assumption. \end{proof}
\begin{lemma}\label{lem-1}
For any given $1\leq i\leq n$ and $m(\Xi)+1\leq j,k \leq N$, we have $a_{ji}a_{ki}\geq 0$, where $a_{ji}$ is the $i$-th coordinate of $A_j$.
\end{lemma}
\begin{proof} Assume $a_{j i}a_{k i}<0$. Then from $a_{jn}=a_{kn}=1$ we get the minor
$
\begin{vmatrix}
	a_{j i}&a_{k i}\\
	1&1
\end{vmatrix}=\pm 2,
$
%Therefore, $\mathbb{Z}^n$ can not be generated by $\{A_j, A_k\}$ and $\{A_l \, | \, 1\leq l\neq i\leq n-1\}$.  %(is a subset of $Y$ automatically)
%can not form a generator of $\mathbb{Z}^n$,
which is a contradiction with our prime assumption.
\end{proof}
For a subset $I\subset \{1,2,\cdots,n\}$, we say a subset $X\subset \mathbb{Z}^n$ has a {\em partial order according to  $I$}, if for any given two vectors $\xi, \eta\in X$, their $i$-th coordinates $\xi_i,\eta_i$ and $j$-th coordinates $\xi_j,\eta_j$ satisfy 
\[
(\xi_i-\eta_i)(\xi_j-\eta_j)\geq 0,\quad i,j\in I.
\]

\iffalse
\begin{lemma}\label{lem-2}
Suppose there is a vector $C \in Y$ such that three coordinates $\{x_{i_1}(C), x_{i_2}(C), x_{i_3}(C)\}$ of $C$ satisfy
$$x_{i_1}(C) x_{i_2}(C)=1,~~~ x_{i_3}(C)=0.$$
Then the subset $Y_\eta \triangleq \{ A_i\in Y\, | \, x_{i_3}(A_i)=1\}$ has a partial order according to $\{i_1, i_2\}$, so does the subset  $Y_\eta^{-1} \triangleq \{  A_i\in Y\, | \, x_{i_3}(A_i)=-1\}$.
\end{lemma}
\begin{proof}
Take two arbitrary vectors  $C_1, C_2\in Y_\eta$, without loss of generality, we assume $x_{i_3}(C_1)=x_{i_3}(C_2)=1$. Then the minor
\[
\begin{vmatrix}
	x_{i_1}(C)&x_{i_1}(C_1)&x_{i_1}(C_2)\\
	x_{i_2}(C)&x_{i_2}(C_1)&x_{i_2}(C_2)\\
	&1&1
\end{vmatrix}=\pm(x_{i_1}(C_1)-x_{i_1}(C_2))(x_{i_2}(C_1)-x_{i_2}(C_1)).
\]
\end{proof}
\fi

\begin{lemma}\label{lem-2}
If there is a vector $A_j \in Y$ such that three coordinates $\{a_{j i_{1}}, a_{j i_{2}}, a_{j i_{3}}\}$ of $A_j$ satisfy
$$a_{j i_{1}} a_{j i_{2}}=1,~~~ a_{j i_{3}}=0,$$
then  for any $1\leq r \leq N$, we have
$$a_{r i_1} a_{r i_2}\geq 0,$$
and the subset $Y_{i_3} \triangleq \{ A_k\in Y\, | \, a_{k i_3}=1\}$ has a partial order according to $\{i_1, i_2\}$, so does the subset  $\widehat{Y}_{i_3} \triangleq \{ A_k\in Y\, | \, a_{k i_3}=-1\}$.
\end{lemma}
\begin{proof} Similarly as in the proof of Lemma~\ref{lem-1}, by considering the minor $\begin{vmatrix}
	a_{j i_1}&	a_{r i_1}\\
         a_{j i_2}&	a_{r i_2}
\end{vmatrix}$, we can derive the first conclusion.

Given two arbitrary vectors  $A_k, A_l\in Y_{i_{3}}$, %without loss of generality,
we have $a_{k i_{3}}=a_{l i_{3}}=1$, which implies $a_{k i_{1}}a_{l i_{1}}\geq 0$ and $a_{k i_{2}}a_{l i_{2}}\geq 0$ by considering the minors $\begin{vmatrix}
	a_{k i_1}&	a_{l i_1}\\
         1&	1
\end{vmatrix}$ and $\begin{vmatrix}
	a_{k i_2}&	a_{l i_2}\\
         1&	1
\end{vmatrix}$, respectively. Therefore, $a_{k i_{1}}-a_{l i_{1}}$ and $a_{k i_{2}}-a_{l i_{2}}$ all take values in $\{0,\pm1\}$. Consider the minor
\[
\begin{vmatrix}
	a_{j i_{1}}&a_{k i_{1}}&a_{l i_{1}}\\
	a_{j i_{2}}&a_{k i_{2}}&a_{l i_{2}}\\
	0&1&1
\end{vmatrix}=\pm[(a_{k i_{1}}-a_{l i_{1}})-(a_{k i_{2}}-a_{l i_{2}})].
\]
It follows from the prime assumption that $(a_{k i_{1}}-a_{l i_{1}})(a_{k i_{2}}-a_{l i_{2}})\geq 0$. % to guarantee that $\mathbb{Z}^n$ can be generated by $\{A_j, A_k, A_l\}$ and $\{A_{i_4}, \cdots, A_{i_n}\}$, %(is a subset of $Y$ automatically)
%form a generator of $\mathbb{Z}^n$,
%where $\{i_4, \cdots,i_n\}$ is the complementary set of $\{i_1, i_2, i_3\}$ in $\{1,2,\cdots,n\}$.
Similar discussion applies to $\widehat{Y}_{i_{3}}$.
\end{proof}
\begin{lemma}\label{lem-3}
If $A_1+A_2+\cdots+A_{n-1}+A_{m(\Xi)+1}$ belongs to $Y$, then all the coordinates of any $A_j\in Y$ must be either $\geq 0$ or $\leq 0$. Especially, after changing the directions of some vectors, all entries of $Y$ take values in $\{0, 1\}$.
%all $m(\Xi)+1\leq j\leq N$,
%$$a_{ji}\in\{0,1\},~~~1\leq i\leq n.$$
%the coordinates of $A_j$ take values in $\{0,1\}$.
\end{lemma}
\begin{proof}
Let $A_0=A_1+A_2+\cdots+A_{n-1}+A_{m(\Xi)+1}$. If there is a vector $A_j\in Y$  %with $j\geq m(\Xi)+1$
 which doesn't satisfy the conclusion, then one can find a $2$-minor in $A_0$ and $A_j$ taking values other than $\{0, \pm1\}$. %Without loss of generality, we can assume $a_{j1}=-1$.  Then it is easy to verify that $\mathbb{Z}^n$ can not be generated by $\{A_0, A_j\}$ and $\{A_2, A_3, \cdots, A_{n-1}\}$,  %(is a subset of $Y$ automatically)
%can not form a generator of $\mathbb{Z}^n$,
This gives a contradiction with our prime assumption.
\end{proof}
The following definition will also be used in the next section.  %sequel discussion.
\begin{definition}
A vector in $\mathbb{Z}^n$ is called $k$-null, if it has exactly $n-k$ nonzero coordinates. Under fixed generator of $\Lambda_n^*$, a lattice vector in $\Lambda_n^*$ is called $k$-null if its coordinate vector is $k$-null.
\end{definition}

\section{Classification of conformally flat and $\lambda_1$-minimal 3-tori and 4-tori}\label{sec-4}
It follows from the works of Montiel-Ros \cite{Mon-Ros} and El Soufi-Ilias \cite{Sou-Ili} that for each conformal structure on compact manifold $(M^n,[g_0])$, there exists at most one metric  $g\in[g_0]$ so that $(M^n,g)$ can be minimally immersed into a sphere %$\mathbb{S}^m$
by the first eigenfunctions. Moreover, if $(M^n,g_0)$ is homogeneous, then such $\lambda_1$-minimal metric must be $g_0$ itself (up to a constant dilation). Note that the flat torus $T^n$ is homogeneous, so we only need to classify all non-congruent $\lambda_1$-minimal immersions of falt $3$-tori and $4$-tori in spheres. %(up to a constant dilation) of $T^n$ in a sphere % $\mathbb{S}^m$
%by the first eigenfunctions. After dilating the metric by a constant, the theorem of Takahashi indicates that such immersions are automatically minimal.
%We just need to classify all the isometric $\lambda_1$-immersions of flat $3$-tori and $4$-tori in spheres.

Let $x: T^n=\mathbb{R}^n/\Lambda_n\rightarrow \mathbb{S}^{2N-1}$ be a $\lambda_1$-minimal flat torus. %an immersion as \eqref{eq-x},
Here we do not assume it is linearly full, and denote by $N$ the number of distinct shortest lattice vectors of $\Lambda^*_n$ up to $\pm1$, %determined by $Q$
%up to $\pm1$,
i.e., the half dimension of the first eigenspace of flat torus $T^n=\mathbb{R}^n/\Lambda_n$. Without loss of generality, we assume this shortest length is %they are of length
$1$. It is well known that $N\geq n$ (see Corollary 3.4 in \cite{doCarmo}). Let
$\Xi$ %$\triangleq\{\xi_j\}_{j=1}^{N}$ %$\{\xi_j\}_{j=1}^{N}$
be the set of these $N$ shortest vectors, which are denoted by
$\xi_1, \xi_2, \cdots,\xi_N$. Then according to Remark \ref{rem:reducible}, $x$ is reducible if and only if $\Lambda^*_n$ is s-reducible defined as in Section \ref{sec-lattice}.
%$$\xi_1, \xi_2, \cdots,\xi_N.$$
%$$\Xi\triangleq\{\xi_j\}_{j=1}^{N}.$$
\subsection{Classification of conformally flat 3-tori}

\begin{theorem}\label{thm:3}
Up to congruence, there are five distinct   $\lambda_1$-minimal immersions of conformally flat $3$-tori in spheres. Two of them are reducible ones given in Example~\ref{ex:3-prod},  %\eqref{eq:reducible3},
others are irreducible ones given in  Example~\ref{ex:3-1} $\sim$ \ref{ex:3-3}. They are all listed in the Table~\ref{tab:my_tabel3}.
\end{theorem}
\begin{proof}%We just need to classify all the isometric $\lambda_1$-immersions of flat $T^3=\mathbb{R}^3/\Lambda_3$ in $\mathbb{S}^m$. %by the first eigenfunctions.
%Let $x: T^3=\mathbb{R}^3/\Lambda_3\rightarrow \mathbb{S}^{2N-1}$ be a  $\lambda_1$-minimal flat torus.
For $\lambda_1$-minimal flat $3$-torus, it follows from Theorem~\ref{thm-gen} that the dual lattice $\Lambda_3^*$ is always prime. Hence by Remark~\ref{rk-pruni} and Proposition~\ref{prop-homo},  the $\lambda_1$-minimal immersion $x$ is homogeneous. Then it follows from Corollary~\ref{cor-unique} that  %which means 
we only need to prove the corresponding integer set $Y$ of $x$ is exactly that given in Example~\ref{ex:3-1} $\sim$ \ref{ex:3-3}, for which %${\rm rank}$ 
$\{A^tA\,|\, A\in Y\}$ is of rank $N$ can be easily checked. 
%can be easily check that 
%determine what kind of integer set $Y$ it can be allocated.
\iffalse
If the immersion is reducible, then it has to be $T^1\times T^1\times T^1$ or $T^2\times T^1$, whose $Q$ and $Y$ have expressions respectively as follows,
\begin{equation}\label{eq:reducible3}
Q=\begin{pmatrix}
	1&&\\&1&\\&&1
\end{pmatrix},\quad
Y^t=\begin{pmatrix}
	1&&\\&1&\\&&1
\end{pmatrix};\quad
Q=\begin{pmatrix}
	1&-\frac{1}{2}&\\-\frac{1}{2}&1&\\&&1
\end{pmatrix},\quad
Y^t=\begin{pmatrix}
	1&&1&\\&1&1&\\&&&1
\end{pmatrix}.
\end{equation}
\fi

%Now
Suppose $x$ is irreducible. It follows from Lemma~\ref{lem:irreducible} that there exists a generic $4$-tuple in $\Xi$, which is assumed to be $\{\xi_1, \xi_2, \xi_3, \xi_4\}$. Moreover, we choose $\{\xi_1, \xi_2, \xi_3\}$ to be a generator of $\Lambda^*_3$ so that $\xi_4=\xi_1+\xi_2+\xi_3$. Then Lemma~\ref{lem-3} implies that all coordinates of lattice vectors in $\Xi$ can be assumed to take values in $\{0,1\}$. Therefore $\Xi\setminus\{\xi_1, \xi_2, \xi_3, \xi_4\}$ is constituted by $1$-null lattice vectors, the number of which is no more than $3$. On the other hand, from
$$(\xi_1+\xi_2)\wedge(\xi_2+\xi_3)\wedge(\xi_1+\xi_3)=2\xi_1\wedge\xi_2\wedge\xi_3,$$
we know this number can not be $3$, which implies $N\leq 6$. If $N=4$, we obtain Example~\ref{ex:3-1}. If $N=5$, we obtain Example \ref{ex:3-2} by making a permutation to $\{\xi_1, \xi_2, \xi_3\}$ such that the only $1$-null lattice vector in $\Xi$ is  $\xi_1+\xi_2$. Similarly, we arrive at Example \ref{ex:3-3} for $N=6$.
\end{proof}

\subsection{Classification of conformally flat 4-tori}
%Let $x: T^4=\mathbb{R}^4/\Lambda_4\rightarrow \mathbb{S}^{2N-1}$ be a  $\lambda_1$-minimal flat torus.
For $\lambda_1$-minimal flat $4$-torus, it follows from Theorem~\ref{thm-gen} that the dual lattice $\Lambda_4^*$ % all lattices of  rank $4$
is prime with only one exception, which is generated by %except the case which is $\Lambda_4^*$ spanned by
	the row vectors of \eqref{eq:exception}.
	We will firstly discuss the $\lambda_1$-minimal immersion of such exceptional torus.
\begin{lemma}
The $\lambda_1$-minimal isometric immersion of such exceptional torus is homogeneous.
\end{lemma}
\begin{proof}
It follows from \eqref{eq:exception} that the shortest vectors of $\Lambda_4^*$ are composed of the following column vectors $\xi_i$ and $-\xi_i$,
\begin{equation}\label{eq:nonprime}
\left (
\begin{array}{rrrrrrrrrrrr}
     1&0&0&0&{1}/{2}&{1}/{2}&{1}/{2}&{-1}/{2}&1/2&1/2&1/2&1/2  \\
     0&1&0&0&{1}/{2}&{1}/{2}&{-1}/{2}&{1}/{2}&1/2&1/2&-1/2&-1/2\\
     0&0&1&0&{1}/{2}&{-1}/{2}&{1}/{2}&{1}/{2}&1/2&-1/2&1/2&-1/2\\
     0&0&0&1&{1}/{2}&{-1}/{2}&{-1}/{2}&{-1}/{2}&-1/2&1/2&1/2&-1/2
\end{array}
\right ),
\end{equation}
which can be divided into three blocks: $I_4$, $S$, $S^t$. It is obvious that $S$ is an  orthogonal matrix and $S^3=I_4$.

It suffices to show that the symmetric products from any $\eta$-set are linearly independent, where the conclusion arises according to Lemma  \ref{prop-homo}.

Suppose $\eta=\xi_{r_1}\pm \xi_{s_1}$.  %Since $S$ is clearly orthogonal and $S^3=I_4$,
Either $\{\xi_{r_1},\xi_{s_1}\}$ comes from the same block which can be assumed  $I_4$  and thus $|\eta|^2=2$, or from different blocks which can be assumed  $I_4$ and $S$ so that $|\eta|^2=1$ or $3$. Here we have used the symmetry induced by $S$.

When $|\eta|^2=3$, we may assume the first coordinate of $\eta$ being $3/2$. Then it is easy to verify that there is no other possibilities for $\{\pm\xi_{r_i},\pm \xi_{s_j}\}$ (note that these vectors have to be distinct). Clearly, $\xi_{r_1}\odot\xi_{s_1}$ is linearly independent.

When $|\eta|^2=2$, the other pairs $\{\xi_{r_i},\xi_{s_i}\}$ must also come from a same block, for $\xi_{r_i}$ and $\xi_{s_i}$ have to be orthogonal. We may assume  $\eta=(1,1,0,0)$ such that all pairs $\{\xi_{r_i},\xi_{s_i}\}$ are given as follows,
\[
\eta=\xi_1+\xi_2=\xi_5+\xi_6=\xi_9+\xi_{10}.
\]
Then by direct computation, we obtain that $\xi_1\odot\xi_2$, $\xi_5\odot\xi_6$ and $\xi_9\odot\xi_{10}$ are linearly independent.

When $|\eta|^2=1$, $\eta$ is one of the shortest vectors. We may assume  $\eta=\xi_5$ such that all pairs $\{\xi_{r_i},\xi_{s_i}\}$ are given as follows,
\[
\eta=\xi_1-\xi_{12}=\xi_2+\xi_{11}=\xi_3+\xi_{10}=\xi_4+\xi_9.
\]
Then it is straightforward to verify that $\xi_1\odot\xi_{12}$, $\xi_2\odot\xi_{11}$, $\xi_3\odot\xi_{10}$ and $\xi_4\odot\xi_{9}$ are linearly independent.
\end{proof}
\begin{proposition}\label{prop:except}
%For the unique exceptional case of %$T^4$,
For such exceptional torus, %is homogeneous.
there is altogether a $2$-parameter family of $\lambda_1$-minimal isometric immersions in $\mathbb{S}^{23}$ up to congruence, given %in Example~\ref{ex-4tori}. 
as follows (see also Example~\ref{ex-4tori}):
\beq\label{eq-excep}
\begin{aligned}
\!\!\Big(a_1e^{i\pi(u_1+u_2+u_3+u_4)},a_1e^{i\pi(u_1+u_2-u_3-u_4)},a_1e^{i\pi(u_1-u_2+u_3-u_4)},a_1&e^{i\pi(-u_1+u_2+u_3-u_4)},\\
a_2e^{i\pi(u_1+u_2+u_3-u_4)},a_2e^{i\pi(u_1+u_2-u_3+u_4)},a_2e&^{i\pi(u_1-u_2+u_3+u_4)}, a_2e^{i\pi(u_1-u_2-u_3-u_4)},\\ &~a_3e^{2i\pi u_1},a_3e^{2i\pi u_2},a_3e^{2i\pi u_3},a_3e^{2i\pi u_4}\Big),
\end{aligned}
\eeq
where $0\leq a_1\leq a_2\leq a_3$ and $a_1^2+a_2^2+a_3^2=\frac{1}{4}$.
%Moreover, among these examples, each one is not congruent to any other.
\end{proposition}
\begin{proof}
From \eqref{eq:nonprime} one can see that $\Lambda_4^*$ can be generated by $-\xi_1,-\xi_2,-\xi_3,\xi_5$ whose
 Gram matrix
is
\beq\label{eq-excep1}
\left(
\begin{array}{cccc}
 1 & 0 & 0 & \frac{1}{2} \\
 0 & 1 & 0 & \frac{1}{2} \\
 0 & 0 & 1 & \frac{1}{2} \\
 \frac{1}{2} & \frac{1}{2} & \frac{1}{2} & 1 \\
\end{array}
\right).
\eeq
By a straightforward computation, we can derive that the matrix $Y$ characterizing all shortest lattice vectors up to $\pm1$ %(the order used here does not coincide with that  in \eqref{eq:nonprime}) 
is given by
$$
Y=\left(
\begin{array}{cccccccccccc}
1&1&0&0& 1&0&0& 1&1&0&0&1\\
1&0&1&0& 0&1&0& 1&0&1&0&1\\
0&1&1&0& 0&0&1& 1&0&0&1&1\\
1&1&1&1& 1&1&1& 1&0&0&0&2\\
\end{array}
\right)^t.
$$
We point out that the order $A_i$ in $Y$ does not coincide with that of $\xi$ in \eqref{eq:nonprime}.
Moreover, for any $c_1^2+c_2^2\leq1/4$,
\beq\label{eq-ci}
\Big(c_2^2,\;c_2^2,\;c_2^2,\;c_2^2,\;\frac{1}{4}-c_1^2-c_2^2,\;\frac{1}{4}-c_1^2-c_2^2,\;\frac{1}{4}-c_1^2-c_2^2,\;\frac{1}{4}-c_1^2-c_2^2,\; c_1^2\;c_1^2,\;c_1^2,\;c_1^2\Big)
\eeq
defines a $\lambda_1$-minimal isometric immersion. In fact, these enumerate all possibilities of the $\lambda_1$-minimal isometric immersion,  since the rank of $\{A^tA\,|\, A\in Y\}$ is $10$, which can be verified directly. Define $ a_1=\sqrt{c_2^2},\;  a_2=\sqrt{\frac{1}{4}-c_1^2-c_2^2},\,a_3=\sqrt{c_1^2}$, then  these immersions can be written down explicitly as \eqref{eq-excep}.

Next, we discuss the congruence of these immersions. Note that an ambient congruence induces an isometry on the flat torus $T^n=\mathbb{R}^n/\Lambda_n$. It is well known that there are two kinds of isometries on $T^n$. One is produced by the translations on $\mathbb{R}^n$. It is obvious that the ambient congruence corresponding to such isometry can not transform one immersion of the form \eqref{eq-excep} to another. The other is induced from the orthogonal transformations on $\mathbb{R}^n$ which  preserve the lattice $\Lambda_n$, and then $\Lambda_n^*$. Since such orthogonal transformations preserve the lengths and angles of lattice vectors, they induce perturbations between the sets $\{\pm I_4\}$, $\{\pm S\}$ and $\{\pm S^2\}$, which further induce perturbations on $\{a_1, a_2, a_3\}$. %induce  It is obvious that the influence of these per perturbations on
\end{proof}
\begin{remark}\label{rk-1para}
The immersion given in \eqref{eq:nonprime} can be seen as a twist product:
\[x(u)\triangleq\Big(a_1f(uS),\;a_2f(uS^2),\;a_3f(u)\Big)\in \mathbb{S}^7(2a_1)\times\mathbb{S}^7(2a_2)\times \mathbb{S}^7(2a_3)\subset\mathbb{S}^{23},\]
where $u=(u_1,u_2,u_3,u_4)$, $f(u)=%\frac{1}{2}
(e^{2i\pi u_1},e^{2i\pi u_2},e^{2i\pi u_3},e^{2i\pi u_4})$, and $S$ is the following orthogonal matrix of order $3$:  
$$\left (
\begin{array}{rrrr}
     {1}/{2}&{1}/{2}&{1}/{2}&{-1}/{2} \\
     {1}/{2}&{1}/{2}&{-1}/{2}&{1}/{2}\\
     {1}/{2}&{-1}/{2}&{1}/{2}&{1}/{2}\\
     {1}/{2}&{-1}/{2}&{-1}/{2}&{-1}/{2}
\end{array}
\right ).$$  %is the Clifford $4$-torus in $\mathbb{S}^7$, i.e.,  the product of four copies of the unit circle in $\mathbb{R}^2$.  
%However, when $a_3<1$, 
The flat torus involved in $x(u)$ is %the immersion of
$$\mathbb{R}^4/\mathrm{Span}_{\mathbb{Z}}\{e_1-e_4, e_2-e_4, e_3-e_4, 2e_4\},$$
%which is not isometric to \eqref{eq-4clliford}. 
with the volume %of $x(u)$ is 
$2\pi^4$. Note that when $a_3=\frac{1}{2}$, $x(u)$ reduces to the double covering of the Clifford 4-torus with the underlying flat torus %given by
$$%\label{eq-4clliford}
\mathbb{R}^4/\mathrm{Span}_{\mathbb{Z}}\{e_1, e_2, e_3, e_4\}.
$$
%Such $\lambda_1$-minimal $4$-torus has  %$\lambda_1$-minimal immersion of type $T^1\times T^1\times T^1\times T^1$, $u=(u_1,u_2,u_3,u_4)$, $S$ is defined as in the above lemma
\end{remark}

\begin{theorem}\label{thm-classify4}
\iffalse
Up to congruence, there are only one exceptional minimal immersion of conformally flat $4$-tori in $\mathbb{S}^{\mathfrak{q}}$ by the first eigenfunctions as shown in Proposition~\ref{prop:except}, six reducible prime cases as in \ref{ex:redu} and ten irreducible prime cases as shown in Examples~\ref{ex:5-1} $\sim$ Examples~\ref{ex:10}.
\fi
Up to congruence, Example~\ref{ex:4-prod} $\sim$ Example~\ref{ex:10},  and Proposition~\ref{prop:except} exhaust all $\lambda_1$-minimal immersions of conformally flat $4$-tori in spheres. They are all listed in the Table~\ref{tab:my_tabel}.
\end{theorem}
\begin{proof}
The exceptional case has been discussed in Proposition~\ref{prop:except}. Next, we assume $\Lambda_4^*$ is prime. It follows from Remark~\ref{rk-pruni} and Proposition~\ref{prop-homo} that the $\lambda_1$-minimal immersion $x$ is homogeneous.  
%Then it follows from
Combining this with Corollary~\ref{cor-unique},   %which means 
we only need to prove the corresponding integer set $Y$ of $x$ is exactly that given in Example~\ref{ex:4-prod} $\sim$ Example~\ref{ex:10}, for which %${\rm rank}$ 
$\{A^tA\,|\, A\in Y\}$ is of rank $N$ can be easily checked. 

%which means we only need to determine what kind of integer set $Y$ it can be allocated. 
Note that all reducible ones have be given in Example~\ref{ex:4-prod}.
%Now we consider the irreducible case.
We will complete the classification of  irreducible ones %the classification,
through the following lemmas and propositions involving prime lattices of rank $4$.
\end{proof}
%%%%%%%%%%%%
\iffalse
Similar to $3$-tori, we assume $\Xi'$ is generated by $\xi_1,\xi_2,\xi_3$. Due to Theorem \ref{thm:3}, the linear relation of $\Xi'$ has to be one of the following $5$ types,
\begin{equation}\label{eq:5types}
	\begin{aligned}
&(I)\begin{pmatrix}
	1 &  & \\
	& 1 & \\
	& & 1
\end{pmatrix},\quad (II)\begin{pmatrix}
	1 &  & 1 &    \\
	& 1 & 1 &    \\
	&  &  & 1
\end{pmatrix},\quad (III)
\begin{pmatrix}
	1 &  &  & 1   \\
	& 1 &  & 1   \\
	&  & 1 & 1
\end{pmatrix},\\
 &\quad (IV)\begin{pmatrix}
	1 & &1 &  & 1   \\
	& 1 &1&  & 1   \\
	&  & &1 & 1
\end{pmatrix},\quad (V)\begin{pmatrix}
	1 & &1 & & & 1   \\
	& 1 &1& &1 & 1   \\
	&  & &1&1 & 1
\end{pmatrix}.
\end{aligned}
\end{equation}
By Lemma \ref{lem:irreducible} we know there could exist lattice with generic $4$-tuple but no generic $5$-tuple. For simplicity we will call a vector $k$-null if there are $k$ zeroes in its entries. Then we have
\fi
%%%%%%%%%%%%%%%
In the rest discussion, when some generator of $\Lambda_4^*$ is chosen, we will identify the lattice vectors with their coordinates with respect to the given generator such that all vectors belong to $\mathbb{Z}^4$.
\begin{lemma}\label{lem-nge}
Suppose $\Lambda_4^*$ is an s-irreducible prime lattice of rank $4$, if there is no generic $5$-tuple in $\Xi$, then $N\leq7$ and $Y$ takes the form as given in Example~\ref{ex:7-1}.
\end{lemma}
\begin{proof}
It follows from Lemma~\ref{lem:irreducible} that there always exists a generic $4$-tuple in $\Xi$. Set $\{\eta_1, \eta_2, \eta_3, \eta_4\}$ to be a generator of $\Lambda_4^*$ such that the generic $4$-tuple is given by
$$\{\eta_1, \eta_2, \eta_3, \eta_5\triangleq \eta_1+\eta_2+\eta_3\}.$$
Using Lemma~\ref{lem-3}, we can assume that all the coordinates of lattice vectors in $\Xi\cap \mathrm{Span}_{\mathbb{Z}} \{\eta_1, \eta_2,\eta_3\}$ take values in $\{0,1\}$. %at most $\eta_2+\eta_3$ can belong to $\Xi$.

Since $\Lambda_4^*$ is s-irreducible, there is at least one vector in $\Xi\setminus \mathrm{Span}_{\mathbb{Z}}\{\eta_1, \eta_2, \eta_3\}$ other than $\eta_4$. Let $\eta$ be such vector which can not be $2$-null. Otherwise, 
 $\eta=\eta_j\pm\eta_4$ for some $1\leq j \leq3$ and then $\{\eta_1, \eta_2, \eta_3,\eta_4, \eta_5,\eta\}\setminus \{\eta_j\}$ constitute a generic $5$-tuple. So we can assume $\eta$ is $(a_1, a_2, 0, 1)$ by the symmetry of $\eta_1,\eta_2,\eta_3$.
It follows from Lemma~\ref{lem-2} that $ a_1a_2>0$. Therefore, after changing the direction of $\eta$ and $\eta_4$ if necessary we can assume  $ a_1=a_2=1$. Moreover, using Lemma~\ref{lem-1} and the partial order according to $\{1,2,3\}$ (see Lemma~\ref{lem-2}), we know there exists no other $1$-null vectors in $\Xi\setminus \mathrm{Span}_{\mathbb{Z}}\{\eta_1, \eta_2, \eta_3\}$,  which means
$\Xi\setminus \mathrm{Span}_{\mathbb{Z}}\{\eta_1, \eta_2, \eta_3\}=\{\eta_4, \eta\}.$

If  $\eta_1+\eta_3$ (resp. $\eta_2+\eta_3$) belongs to $\Xi$, then $\{\eta_4, \eta, \eta_5\}$ together with $\{\eta_1+\eta_3, \eta_1\}$ (resp. $\{\eta_2+\eta_3, \eta_2\}$) constitute a generic $5$-tuple. Therefore, besides $\{\eta_1, \eta_2, \eta_3,\eta_5\}$, $\eta_1+\eta_2$ is the only vector may allowed  in $\Xi$ from $\mathrm{Span}_{\mathbb{Z}} \{\eta_1, \eta_2,\eta_3\}$, which completes the proof of this lemma.
\end{proof}
\begin{lemma}\label{lem-4}
Suppose $\Xi$ contains a generic $5$-tuple $X$. Then we can choose $\Xi'$ such that $\sharp(\Xi'\cap X)= 3$.
\end{lemma}
\begin{proof}

We assume $X=\{\eta_1, \eta_2, \eta_3, \eta_4, \eta_5\}$ and $\eta_5=\eta_1+\eta_2+\eta_3+\eta_4$.  Choose $\{\eta_1, \eta_2, \eta_3, \eta_4\}$ to be a generator of $\Lambda_4^*$,
then by Lemma~\ref{lem-3} we have all coordinates of vectors in $\Xi$ taking values in $\{0,1\}$.

Since $\mathrm{rank}(\Xi')=3$, it suffices to prove that  we can choose $\Xi'$ such that $\sharp(\Xi'\cap X)\geq 3$.

If $\Xi'$ contains only $0$-null or $3$-null vectors (i.e. $\eta_i$) then $\Xi=\{\eta_1, \eta_2, \eta_3, \eta_4, \eta_5\}$ and $m(\Xi)=3$. For $m(\Xi)>3$, $\Xi'$ contains at least a $1$-null  or $2$-null vector. Note that by rechoosing  a generator of $\Lambda_4^*$ in $X$, a given $2$-null vector (such as $\eta_1+\eta_2$) can be transformed to a $1$-null vector (such as $\eta_1+\eta_2+\eta_3$ by choosing $\{\eta_5, -\eta_3, -\eta_4, -\eta_1\}$ as  a new generator). Therefore, without loss of generality, we assume there is a $1$-null vector
$\eta\triangleq \eta_1+\eta_2+\eta_3\in\Xi'$.  

It is easy to see that if $\sharp(\Xi'\cap \{\eta_1, \eta_2,\eta_3\})=2$ then $\Xi'=\Xi\cap \mathrm{Span}_{\mathbb{Z}}\{\eta_1,\eta_2,\eta_3\}$.

If $\sharp(\Xi'\cap \{\eta_1, \eta_2,\eta_3\})=1$, we can assume it is $\eta_1$. Since $\sharp(\Xi\cap\mathrm{Span}_{\mathbb{Z}} \{\eta_1, \eta_2,\eta_3\})\geq 4$ and $\eta_1,\eta\in\Xi'$, we know $\sharp(\Xi'\setminus\mathrm{Span}_{\mathbb{Z}} \{\eta_1, \eta_2,\eta_3\})\geq 2$ in which all the vectors have coordinates $1$ with respect to $\eta_4$. It follows from Lemma~\ref{lem-2} that the existence of $\eta=\eta_1+\eta_2+\eta_3$ implies the vectors in $\Xi'\setminus\mathrm{Span}_{\mathbb{Z}} \{\eta_1, \eta_2,\eta_3\}$ 
obey the partial order according to $\{1,2,3\}$, which means there are at most one $k$-null vector in $\Xi'\setminus\mathrm{Span}_{\mathbb{Z}} \{\eta_1, \eta_2,\eta_3\}$ for every $0\leq k\leq 3$.
When $\Xi'\setminus\mathrm{Span}_{\mathbb{Z}} \{\eta_1, \eta_2,\eta_3\}$ is composed of exactly two vectors, we get $\Xi\cap\mathrm{Span}_{\mathbb{Z}} \{\eta_1, \eta_2,\eta_3\}$ also attaining $m(\Xi)$ so that it can be chosen as the $\Xi'$ we desired. Or else, $\Xi'$ must contain $\eta_4$ or $\eta_5$. Since  $\eta_5-\eta_4=\eta\in \Xi'$, we have $\Xi'$ must contain both $\eta_4$ and $\eta_5$ simultaneously, % contained in $\Xi'$ 
which implies %and therefore 
$\eta_1,\eta_4,\eta_5\in\Xi'$. The third one in $\Xi'\setminus\mathrm{Span}_{\mathbb{Z}} \{\eta_1, \eta_2,\eta_3\}$ must be $\eta_1+\eta_4$ if it's $2$-null or  $\eta_2+\eta_3+\eta_4$ if it's $1$-null. Otherwise, $\eta_2$ or $\eta_3$ will be contained in $\Xi'$. Moreover, only one of these two vectors could be allowed in $\Xi'$ for they violate the partial order according to $\{1,2,3\}$. This is obviously a case of $N=7$.

If $\Xi'\cap \{\eta_1, \eta_2,\eta_3\}=\emptyset$ then $\Xi'\cap \mathrm{Span}_{\mathbb{Z}}\{\eta_1, \eta_2,\eta_3\}=\{\eta\}$. Similarly, we have $\sharp(\Xi'\setminus\mathrm{Span}_{\mathbb{Z}} \{\eta_1, \eta_2,\eta_3\})\geq 3$ and all the vectors in $\Xi'\setminus\mathrm{Span}_{\mathbb{Z}} \{\eta_1, \eta_2,\eta_3\}$ admit at most one $k$-null vector for every $0\leq k\leq 3$. Therefore, $\Xi'\cap\{\eta_4,\eta_5\}\not=\emptyset$ and thus $\eta_4,\eta_5\in\Xi'$. 
Note that there are still others left in $\Xi'\setminus\mathrm{Span}_{\mathbb{Z}}\{\eta_1,\eta_2,\eta_3\}$. However, any $1$-null (resp. $2$-null)  minus $\eta_5$ (resp.  $\eta_4$) will yield $\Xi'\cap\{\eta_1, \eta_2,\eta_3\}\not=\emptyset$. Hence we finish the proof by this contradiction. 
\end{proof}

\begin{remark}\label{rem:N>7}
From the proof of above two lemmas one can see that for $N\geq 8$, we can always choose a generic $5$-tuple and a generator of $\Lambda_4^*$ as in Lemma \ref{lem-4} such that $\Xi'=\Xi\cap\mathrm{Span}_\mathbb{Z} \{\eta_1, \eta_2,\eta_3\}$ and $\eta\in \Xi'$.
\end{remark}

\begin{proposition}\label{lem-N<8}
If $N\leq 7$ and there exists a generic $5$-tuple in $\Xi$, then we can find  a generator of $\Lambda_4^*$ such that $Y$ takes the form as given in Example~\ref{ex:5-1}, Example~\ref{ex:6-1}, Example~\ref{ex:7-2} and Example~\ref{ex:7-3}.
\end{proposition}
\begin{proof}%Similarly as in the proof of Lemma \ref{lem-4}, we assume 
Suppose $\{\eta_1, \eta_2, \eta_3, \eta_4, \eta_5\}$ is a generic $5$-tuple in $\Xi$, with $\eta_5=\eta_1+\eta_2+\eta_3+\eta_4$. 
When %there is a generic $5$-tuple and 
$N=5$, it is easy to see that after choosing $\{\eta_1, \eta_2, \eta_3, \eta_4\}$ to be a generator of $\Lambda^*_4$, we obtain $Y$ as given in Example~\ref{ex:5-1}. %it is obvious that $Y$ is given in Example~\ref{ex:5-1} and $m(\Xi)=3$.
%When $N=6$, the proof of Lemma \ref{lem-4} implies that the only remainder vector not belonging to this generic $5$-tuple can be assumed to be $\eta$ by choosing suitable generators of $\Lambda_4^*$,   %as in the proof of Lemma \ref{lem-4}
%which leads to Example~\ref{ex:6-1}. 

When $N\geq 6$, similarly as discussed in the proof of Lemma~\ref{lem-4}, we can assume
$$\eta_1, \eta_2, \eta_3, \eta_1+\eta_2+\eta_3, \eta_4, \eta_5\in \Xi.$$
For the case of $N=6$, by choosing $\{\eta_1, \eta_2, \eta_3, \eta_4\}$ to be a generator of $\Lambda^*_4$, we arrive at Example~\ref{ex:6-1}. 
For the case of $N=7$, %besides $\{\eta_1, \eta_2, \eta_3, \eta_4, \eta_5,\eta\}$, 
if the %only 
remainder lattice vector lies on $\mathrm{Span}_{\mathbb{Z}}\{\eta_1, \eta_2, \eta_3\}$, then we obtain $Y$ as given in Example~\ref{ex:7-2} after a permutation in $\{\eta_1, \eta_2, \eta_3\}$. Now we assume the remainder lies on $\Xi\setminus\mathrm{Span}_{\mathbb{Z}}\{\eta_1, \eta_2, \eta_3\}$. Either it is $1$-null, which can be assumed $\eta_2+\eta_3+\eta_4$ without loss of generality, so that we arrive at Example~\ref{ex:7-3}. %we choose $\{\eta_1, \eta_2+\eta_3+\eta_4, -\eta_4, \eta_2\}$ as generators of $\Lambda_4^*$ so that 
%$Y$ takes the form as given in Example~\ref{ex:7-3}. 
Or it is $2$-null, which can be assumed $\eta_3+\eta_4$ without loss of generality. Then choosing $\{-(\eta_1+\eta_2+\eta_3), \eta_5, -(\eta_3+\eta_4), -\eta_2\}$ as a generator of  $\Lambda_4^*$, we also obtain $Y$ as in Example~\ref{ex:7-3}.
\end{proof}
 
\begin{proposition}
If $N\geq 8$, then there exist a generator of $\Lambda_4^*$ such that $Y$ takes the form as given in Example~\ref{ex:8-1}$\sim$Example~\ref{ex:10}.
\end{proposition}
\begin{proof}
It follows from Remark~\ref{rem:N>7} that there exist a generic $5$-tuple $\{\eta_1, \eta_2, \eta_3, \eta_4,\eta_5\}$ and $\Xi'$ such that 
$$\{\eta_1, \eta_2, \eta_3, \eta\triangleq\eta_1+\eta_2+\eta_3\}\in \Xi',\quad\eta_5=\eta_1+\eta_2+\eta_3+\eta_4.$$ 
By choosing $\{\eta_1, \eta_2, \eta_3, \eta_4\}$ as a generator of $\Lambda_4^*$, it follows from Lemma~\ref{lem-3} that all coordinates of vectors in $Y$ can only take values in $\{0,1\}$. 

If $m(\Xi)=4$, the vectors in $\Xi$ other than $\eta_1, \eta_2, \eta_3, \eta, \eta_4, \eta_5$ don't lie on $\mathrm{Span}_{\mathbb{Z}}\{\eta_1,\eta_2,\eta_3\}$. As in the proof of Lemma~\ref{lem-4}, they can be assumed $\eta_3+\eta_4$ and $\eta_2+\eta_3+\eta_4$ due to the partial order according to $\{1,2,3\}$ (see Lemma \ref{lem-2}). It means $\sharp(\Xi\cap\mathrm{Span}_{\mathbb{Z}}\{\eta_2,\eta_3,\eta_4\})=5$  against $m(\Xi)=4$. So $m(\Xi)\geq 5$.

When $m(\Xi)=5$, we can assume $\eta_1+\eta_2\in \Xi'$ after making a permutation to $\{\eta_1, \eta_2, \eta_3\}$. Note that $\eta,\eta_5\in \mathrm{Span}_{\mathbb{Z}}\{\eta_1+\eta_2,\eta_3,\eta_4\}$, neither $\eta_3+\eta_4$ nor $\eta_1+\eta_2+\eta_4$ is allowed in $\Xi$ which will make $\sharp(\Xi\cap\mathrm{Span}_{\mathbb{Z}}\{\eta_1+\eta_2,\eta_3,\eta_4\})\geq 6>m(\Xi)$. Due to the partial order according to $\{1,2,3\}$ and the symmetry of $\eta_1$ and $\eta_2$, there is at most one $1$-null vector in $\Xi\setminus\mathrm{Span}_{\mathbb{Z}}\{\eta_1,\eta_2,\eta_3\}$ which can be assumed $\eta_2+\eta_3+\eta_4$, and at most one $2$-null vector in $\Xi\setminus\mathrm{Span}_{\mathbb{Z}}\{\eta_1,\eta_2,\eta_3\}$ which can be assumed $\eta_2+\eta_4$.
So $N\leq 9$. If $N=9$, both $\eta_2+\eta_3+\eta_4$ and $\eta_2+\eta_4$ exist and $Y$ takes the form as given in Example~\ref{ex:9-2}.  If $N=8$ and only $\eta_2+\eta_3+\eta_4$ exists, then we arrive at Example~\ref{ex:8-2}. If $N=8$ and only $\eta_2+\eta_4$ exists, then after choosing $\{-\eta_2, -\eta_1, -\eta_3, \eta_5\}$ as a new generator, we can also obtain $Y$ as given in Example~\ref{ex:8-2}.

When $m(\Xi)=6$, up to a permutation of $\{\eta_1, \eta_2, \eta_3\}$, we can assume $\eta_1+\eta_2, \eta_2+\eta_3 \in \Xi'$. From 
\[\begin{aligned}&(\eta_1+\eta_3)\wedge(\eta_1+\eta_2)\wedge(\eta_2+\eta_3)\wedge\eta_4=(\eta_1+\eta_3+\eta_4)\wedge(\eta_1+\eta_2)\wedge(\eta_2+\eta_3)\wedge\eta_4\\
=&(\eta_2+\eta_4)\wedge(\eta_2+\eta_3)\wedge(\eta_1+\eta_2)\wedge\eta_5
=2\eta_1\wedge\eta_2\wedge\eta_3\wedge\eta_4,
\end{aligned}\]
we know none of $\eta_1+\eta_3$, $\eta_1+\eta_3+\eta_4$ or $\eta_2+\eta_4$ can be contained in $\Xi$. %Analogously, by 
Combining this with the partial order according to $\{1,2,3\}$ and the symmetry of $\eta_1$ and $\eta_3$, there is at most one $1$-null vector in $\Xi\setminus\mathrm{Span}_{\mathbb{Z}}\{\eta_1,\eta_2,\eta_3\}$ which can be assumed $\eta_2+\eta_3+\eta_4$, and at most one $2$-null vector in $\Xi\setminus\mathrm{Span}_{\mathbb{Z}}\{\eta_1,\eta_2,\eta_3\}$ which can be assumed $\eta_3+\eta_4$. So $N\leq 10$. If $N=8$ then $Y$ takes the form as given in Example~\ref{ex:8-1}. If $N=10$, both $\eta_2+\eta_3+\eta_4$ and $\eta_3+\eta_4$ exist, which leads to  Example~\ref{ex:10}. If $N=9$ and only $\eta_3+\eta_4$ exists, then we obtain $Y$ as given in Example~\ref{ex:9-1}. If $N=9$ and only $\eta_2+\eta_3+\eta_4$ exists, then after choosing $\{-\eta_3, -\eta_2, -\eta_1, \eta_5\}$ as a new generator, we have also $Y$ taking the form as given in Example~\ref{ex:9-1}.

The non-existence of $\eta_1+\eta_3$ implies $m(\Xi)<7$ and we finish the proof. 
\end{proof}
\begin{remark}\label{rk-maxxi}
Through the discussion in this section, we can obtain that in a prime lattice of rank $n\leq 4$, there exists at most $\frac{n(n+1)}{2}$ distinct lattice vectors of shortest length up to $\pm 1$. 
\end{remark}

\section{On $\lambda_1$-minimal flat tori of higher dimension}\label{sec-high}
Similarly as in section~\ref{sec-examples}, we can construct many examples of higher dimensional $\lambda_1$-minimal flat tori in spheres. For simplicity, we choose a certain class to introduce in this section.
%\subsection{Standard simplex}

Consider the set $X_n\subset \mathbb{Z}^n$ given by the column vectors of
%When every vector in $\Xi$ has a angle of $\frac{\pi}{3}$ with each neighbor vectors, the lattice can be seen as constructed from a $n$-dimensional standard simplex, which has
\[
X^t_n=\begin{pmatrix}
	1 & & 1 & & & 1 &      & & & 1\\
	  &1& 1 & &1& 1 &      & & & 1\\
	  & &   &1&1& 1 &      & & & 1\\
	  & &   & & &   &\ddots& & & \vdots\\
	  & &   & & &   &      & &1\cdots&1\\
	  & &   & & &   &      &1&1\cdots& 1\\
\end{pmatrix}_{n\times\frac{n(n+1)}{2}}.
\]
\iffalse
\[
X^t_n=\begin{pmatrix}
     & & & &1\\
	 & & & &1\\
	 & & & &1\\
	 & & & &\vdots\\
     & &1&1\cdots1&1\\
	 &1&1&1\cdots1& 1\\
\end{pmatrix}_{n\times\frac{n(n+1)}{2}}.
\]
\fi
We call $X_n$ the { \em ladder set}.
It is straightforward to verify that $W_{X_n}$ is a singleton set, only contains
\[
Q_n=\begin{pmatrix}
	1           &-\frac{1}{2}& & & & \\
	-\frac{1}{2}& 1          &-\frac{1}{2}&&&\\
	&-\frac{1}{2}&1&&&\\
	&&\ddots&\ddots&\ddots&\\
	&&&&1&-\frac{1}{2}\\
	&&&&-\frac{1}{2}&1\\
\end{pmatrix}.
\]
By direct computation, we can obtain that
\beq\label{eq-Q}
Q^{-1}_n=\frac{2}{n+1}\begin{pmatrix}
	n&n-1&n-2&\cdots&3&2&1\\
	n-1&(n-1)2&(n-2)2&\cdots&3\cdot 2&2\cdot 2&2\\
	n-2&(n-2)2&(n-3)3&\cdots&3\cdot 3&3\cdot 2&3\\
	\vdots&\vdots&\vdots&&\vdots&\vdots&\vdots\\
	3&3\cdot 2&3\cdot 3&\cdots&(n-3)3&(n-3)2&n-2\\
	2&2\cdot 2&3\cdot 2&\cdots&(n-2)2&(n-1)2&n-1\\
	1&2&3&\cdots&n-2&n-1&n\\
\end{pmatrix}, %=(a_{ij}),
\eeq
and
%where $a_{ij}=i(n+1-j)$ for any $i\leq j$. It is obvious that
\beq\label{eq-Qin}
Q_n^{-1}=\frac{2}{n+1}\left(A_1^t A_1+A_2^t A_2+\cdots+A_{\frac{n(n+1)}{2}}^tA_{\frac{n(n+1)}{2}}\right).
\eeq
Moreover, for any $(a_1, a_2, \cdots, a_n)\in \mathbb{Z}^n\backslash 0$, we have
\beq\label{eq-simplicial}
(a_1, a_2, \cdots, a_n)Q_n(a_1, a_2, \cdots, a_n)^t=\sum_{i=1}^n a_i^2-\sum_{i=1}^{n-1} a_{i}a_{i+1}=\sum_{i=1}^{n-1} \frac{(a_{i}-a_{i+1})^2}{2}+\frac{a_1^2+a_n^2}{2}\geq 1,
\eeq
with the equality holds if and only if $(a_1, a_2, \cdots, a_n)$ comes from $X_n$.
%\textcolor{red}{It takes equality if and only if the vector comes from $X_n$.}
%Therefore, we have:
\begin{proposition}\label{prop-symlicial1}
The matrix data set $\{X_n,\, Q_n,\, \frac{2}{n(n+1)}(1,1,\cdots,1)\}$
provides a $\lambda_1$-minimal flat torus in
$\mathbb{S}^{n(n+1)-1}$, which is linearly full and has volume $\frac{(2\sqrt 2)^n}{\sqrt{n+1}(\sqrt n)^n}\pi^n$.
\end{proposition}
%immersion from $T^n$ to $\mathbb{S}^{n(n+1)-1}$.
One can see that when we take $n=2, 3, 4$,
we obtain the equilateral torus, Example~\ref{ex:3-3}, and Example~\ref{ex:10} respectively.
%the WHATS the NAME? for $T^2$, Example \ref{ex:3-3} for $T^3$ and Example \ref{ex:10} for $T^4$ are exactly the cases of standard simplex.

%\subsection{A lattice based on standard simplex}

Next, for a given $1\leq k\leq n$, we consider the integer set $X_{n,k}$ obtained by   removing the last $n-k$ rows from the ladder set $X_n$, i.e.,
\beq\label{eq-Xnk}
X^t_{n,k}=\begin{pmatrix}
	1 & & 1 & & & 1 &      & & & \\
	&1& 1 & &1& 1 &      & & & \\
	& &   &1&1& 1 &      & & & 1\\
	& &   & & &   &\ddots& & & \vdots\\
	& &   & & &   &      & &1\cdots&1\\
	& &   & & &   &      &1&1\cdots& 1\\
\end{pmatrix}_{n\times(\frac{n(n-1)}{2}+k)}.
\eeq
\begin{remark}
Note that when $k=n$, $X_{n,k}=X_n$, which can automatically determine a $\lambda_1$-minimal flat $n$-torus as discussed above. When $k=1$, %this is also true, since
$X_{n,k}$ is a block diagonal matrix, from which a reducible $\lambda_1$-minimal flat $n$-torus can be obtained.
\end{remark}
%identifies with $X_n$ for $k=1$ and $k=n$, $X^t_{n,k}$ can always determined we always have the  in which case one is reducible and the other is the standard simplex.
Next let us consider the case of $2\leq k\leq n-1$. Such kind of $X_{n,k}$ is called the {\em faulted ladder set}.
%With respect to this lattice we will have
\begin{lemma} %The maximal value of the determinant function  restricted on $W_{X_{n,k}}$ is attained by
	Define
	\[
	Q_{n,k}\triangleq Q_n+\frac{1}{2k}(E_{n,n-k}+E_{n-k,n}), \quad
	\]
	where $E_{i,j}=e_i^t e_j$ and $e_i$ is the $i$-th row of $I_n$, % is a matrix has entries all being $0$ except $a_{(n-k)n}=a_{n(n-k)}=\frac{1}{2k}$,
	then $Q_{n,k}\in W_{X_{n,k}}$, and $$Q_{n,k}^{-1}=R_n+S_k-T_k,$$
	with
	\[
	R_n=\begin{pmatrix}
		Q^{-1}_{n-1}&\\&0
	\end{pmatrix},\quad S_k=\begin{pmatrix}
	O&\\&Q^{-1}_k
\end{pmatrix},\quad T_k=\begin{pmatrix}
O&&\\&Q^{-1}_{k-1}&\\&&0
\end{pmatrix}.
	\]
\end{lemma}
\begin{proof}
The first conclusion is easy to be verified, we only prove the second one.

Write
$$Q_n^{-1}=\frac{2}{n+1}\left(\sum_{1\leq i<j\leq n}i(n+1-j)\left(E_{i,j}+E_{j,i}\right)+\sum_{1\leq i \leq n}i(n+1-i)E_{i,i}\right),$$
%where $a_{ij}={i(n-j)},~1\leq i\leq j\leq n$.
It follows %from \eqref{eq-Q}
that
%\begin{shrink}
\begin{align}
    &R_n=\frac{2}{n}\left(\sum_{1\leq i<j\leq n-1}i(n-j)\left(E_{i,j}+E_{j,i}\right)+\sum_{1\leq i \leq n-1}i(n-i)E_{i,i}\right),\label{eq-Rn}\\
    &S_k=\frac{2}{k+1}\!\!\left(\!\!\sum_{1\leq i<j\leq k}i(k+1-i)\left(E_{n-k+i,n-k+j}+E_{n-k+j,n-k+i}\right)+\!\!\!\sum_{1\leq i \leq k}i(k+1-i)E_{n-k+i,n-k+i}\right),\label{eq-Sk}\\
    &T_k=\frac{2}{k}\left(\sum_{1\leq i<j\leq k-1}i(k-i)\left(E_{n-k+i,n-k+j}+E_{n-k+j,n-k+i}\right)+\sum_{1\leq i \leq k-1}i(k-i)E_{n-k+i,n-k+i}\right). \label{eq-Tk}
\end{align}
%\end{shrink}
Note that we can also express $Q_{n,k}$ as follows:
\begin{equation}
\begin{aligned}
	&Q_{n,k}= \begin{pmatrix}
		Q_{n-1}&\\&1
	\end{pmatrix}-\frac{1}{2}(E_{n,n-1}+E_{n-1,n})+\frac{1}{2k}(E_{n,n-k}+E_{n-k,n}), \\
	=&\begin{pmatrix}
		Q_{n-k}&\\&Q_k
	\end{pmatrix}-\frac{1}{2}(E_{n-k+1,n-k}+E_{n-k,n-k+1})+\frac{1}{2k}(E_{n,n-k}+E_{n-k,n})\\
	=&\!\!\begin{pmatrix}
		Q_{n-k}&&\\&Q_{k-1}&\\&&1
	\end{pmatrix}\!\!\!-\!\frac{1}{2}(E_{n,n-1}+E_{n-1,n})\!-\!\frac{1}{2}(E_{n-k+1,n-k}+E_{n-k,n-k+1})\!+\!\frac{1}{2k}(E_{n,n-k}+E_{n-k,n})\label{eq-Q3}.
\end{aligned}\end{equation}
Combining the fact that
$$E_{i,j}E_{k,l}=\delta_{j,k}E_{i,l},~~~1\leq i, j, k, l\leq n$$
with \eqref{eq-Rn} $\sim$ \eqref{eq-Q3}, we can obtain that
\begin{align*}
	&Q_{n,k}R_n= \begin{pmatrix}
		I_{n-1}&\\&0
	\end{pmatrix}+\sum_{n-k< i\leq n-1}\frac{n-k-i}{k}E_{n,i}, \\
	&Q_{n,k}S_k=\begin{pmatrix}
		0&\\&I_k
	\end{pmatrix}-\sum_{1\leq i\leq k}\frac{k-i}{k}E_{n-k,n-k+i},\\
    &Q_{n,k}T_k=\begin{pmatrix}
		0&&\\&I_{k-1}&\\&&0
	\end{pmatrix}-\sum_{1\leq i\leq k-1}\frac{i}{k}E_{n,n-k+i}-\sum_{1\leq i\leq k-1}\frac{k-i}{k}E_{n-k,n-k+i},
\end{align*}
from which the conclusion follows. %complete the proof.
\end{proof}
It follows from \eqref{eq-Qin} that
%Now we get
\[
R_n=\frac{2}{n}\left(A_1^tA_1+\cdots+A_{\frac{n(n-1)}{2}}^tA_{\frac{n(n-1)}{2}}\right)=\frac{2}{n}\sum_{j=1}^{n-1}\sum_{i=1}^j
A^t_{\frac{j(j+1)}{2}+i}A_{\frac{j(j+1)}{2}+i},\quad
%T_k=\frac{2}{k}\left(\tilde{B}_1+\cdots+\tilde{B}_{\frac{k(k-1)}{2}}\right),
\]\[
S_k=\frac{2}{k+1}\sum_{j=n-k}^{n-1}\sum_{i=1}^{j+1+k-n}
A^t_{\frac{j(j+1)}{2}+i}A_{\frac{j(j+1)}{2}+i},~~~
T_k=\frac{2}{k}\sum_{j=n-k}^{n-2}\sum_{i=1}^{j+1+k-n}
A^t_{\frac{j(j+1)}{2}+i}A_{\frac{j(j+1)}{2}+i},
\]
where $A_i^t$ is the $i$-th column vector of  $X^t_{n,k}$. %with the first $n-k$ and the last entries being $0$, $\hat B_i$ comes from the last $k$ column vectors and $B_i$ from the first $\frac{n(n-1)}{2}$ columns. Since $\tilde{B}_i$ belongs to $\{B_j\}$,
Note that in $Q_{n,k}^{-1}$, the coefficient of  $A^t_{\frac{j(j+1)}{2}+i}A_{\frac{j(j+1)}{2}+i}$ is
\[
\frac{2}{n}+\frac{2}{k+1}-\frac{2}{k}=\frac{2(k^2+k-n)}{nk(k+1)}
\]
for $n-k\leq j\leq n-2$ and $1\leq i\leq j+1+k-n$; the other coefficients are $\frac{2}{n}>0$ or $\frac{2}{k+1}>0$.
\begin{proposition} Suppose $2\leq k\leq n-1$. %\textcolor{red}{in my first version, it is $1\leq k\leq n$. Because we need prove they are also $\lambda_1$-minimal, especially the case $k=n$}
%Suppose $K$ is an integer such that $K(K-1)<n\leq K(K+1)$.

(1) When $k^2+k>n$, the faulted ladder set $X^t_{n,k}$ given in \eqref{eq-Xnk}  can determine a linearly full and $\lambda_1$-minimal flat $n$-torus in $\mathbb{S}^{n(n-1)+2k-1}$. %. %a sphere of dimension $\frac{n(n-1)}{2}+k$.

(2) When $k^2+k=n$, %and $n=K(K+1)$
the faulted ladder set $X^t_{n,k}$ given in \eqref{eq-Xnk}  can determine a linearly full and $\lambda_1$-minimal flat torus in a sphere of dimension $n(n-1)-k^2+3k-1$, with ${k(k-1)}$ dimensional eigenfunctions redundant.  %a sphere of dimension $k^4+2 k^3-k^2+2 k-1$ with $\frac{k(k-1)}{2}$ eigenfunctions redundant.
\end{proposition}
\iffalse
\begin{remark}%When $n<k^2+k$, %or $\{k=K$ and $n<K(K+1)\}$ we get an immersion in de-Sitter space of dimension ${n(n-1)}+2k-1$ with negative indices ${k(k-1)}$.
Note that for $k=1$ and $k=n$, we always have the  in which case one is reducible and the other is the standard simplex.
\end{remark}
\fi
\begin{proof}
Assume $\Lambda_n^*$ is the lattice determined by $Q_{n,k}~(1\leq k\leq n)$, and the corresponding generator is $\{\xi_1,\cdots,\xi_n\}$. Denote by $\Lambda_n$ the dual lattice. Define $\left(c_1^2, c_2^2, \cdots\cdots, c^2_{\frac{n(n+1)}{2}+k}\right)$ as follows:
\[c^2_{\frac{j(j+1)}{2}+i}=\begin{cases}
\dfrac{2(k^2+k-n)}{nk(k+1)},&n-k\leq j\leq n-2 \text{ and } 1\leq i\leq j+1+k-n;\\
\dfrac{2}{n},&j<n-k \text{ and } 1\leq i\leq j;\\
\dfrac{2}{k+1},&\text{others}.
\end{cases}
\]
Since the matrix data set $\Big\{X_{n,k}, Q_{n,k}, \big(c_1^2, c_2^2, \cdots\cdots, c^2_{\frac{n(n+1)}{2}+k}\big)\Big\}$ satisfy \eqref{eq-AQA} and \eqref{eq:flat}, they can determine an isometric minimal immersion of $T^n=\mathbb{R}^n/\Lambda_n$ in $\mathbb{S}^{n(n-1)+2k-1}$.

We are left to show $1$ is the shortest length in $\Lambda_n^*$, and  all the lattice vectors of this length having coordinates vectors as given in $X_{n,k}$ up to the direction.
%Now we will show that these immersions are all $\lambda_1$-minimal.
\iffalse
\begin{proposition}
If a lattice $\Lambda^*_n$ has a basis whose Gram matrix is $Q_{n,k}~(1\leq k\leq n)$, then the shortest length of lattice is $1$ and all the unit-length vectors can be determined from $Y_{n,k}$.
\end{proposition}
\begin{proof}
Suppose $\xi_1,\cdots,\xi_n$ form such a basis. From the expression of $Q_{n,k}$ one can see that in the $(n-1)$-dimensional subspace $\tilde\Pi$ spanned by $\xi_1\cdots,\xi_{n-1}$ it is exactly a $(n-1)$-dimensional sub-lattice of standard simplex. So we may assume that the conclusion holds true for this $(n-1)$-dimensional sub-lattice, due to the existing classification results for $n\leq 5$.
\fi

Consider the sublattice generated by  $\{\xi_1,\cdots,\xi_{n-1}\}$. It is obvious that it takes $Q_{n-1}$ as the Gram matrix. Using \eqref{eq-simplicial}, we know $1$ is the shortest length in this sublattice. Next, we will use Theorem~\ref{thm:general} to show this is also true for $\Lambda_n^*$.

%it is exactly a $(n-1)$-dimensional sub-lattice of standard simplex. So we may assume that the conclusion holds true for this $(n-1)$-dimensional sub-lattice, due to the existing classification results for $n\leq 5$.

Firstly, we calculate the distance from $\xi_n$ to the hyperplane $\mathrm{Span}_{\mathbb{R}}\{\xi_1,\cdots,\xi_{n-1}\}$. Define $v\triangleq(v_1,\cdots,v_{n-1})$, where $v_{n-k}=1/(2k)$, $v_{n-1}=-1/2$ and the others are $0$. By \eqref{eq:dn} we get
\[
d_n^2=1-vQ_{n-1}^{-1}v^t=1-\frac{k-1}{2k}=\frac{k+1}{2k}>\frac{1}{4},
\]
which implies for any vector  $\xi=a_1\xi_1+\cdots+a_n\xi_n\in\Lambda_n^*$, if $|a_n|> 1$ then $\|\xi\|> 1$. When $a_n=0$, as a lattice vector in the sublattice, it is obvious that $\|\xi\|\geq1$, with equality holds if and only if $\pm\xi$ belongs to  $X_{n-1}$, which can be embedded into $X_{n,k}$. So we only need to discuss the case of $a_n=\pm1$. By considering $-\xi$ if necessary, we can assume $a_n=1$. Then we have
\begin{align*}
|\xi|^2-1&=a_1^2+\cdots+a_{n-1}^2-a_1a_2-\cdots-a_{n-2}a_{n-1}-a_{n-1}+\frac{1}{k}a_{n-k}\\
&=\sum_{i=1}^{n-1}\frac{1}{2}(a_i-a_{i+1})^2+\frac{1}{2}\left(a_1^2-1\right)+\frac{1}{k}a_{n-k},
\end{align*}
which is larger than 0 when $a_{n-k}>0$.
For the case of $a_{n-k}=0$, it is obvious that $|\xi|^2-1\geq 0$, and the equality holds if and only if %It takes ``$=$" if and only if
$a_1=\cdots=a_{p-1}=0$ and
$a_p=\cdots=a_n=1$ for some $p>n-k$,
%with the others equal zero. Such kind of $\xi$
which exactly corresponds to a certain row of $X_{n,k}$. When $a_{n-k}<0$, using the Cauchy inequality, we have
$$\sum_{i=n-k}^{n-1}(a_i-a_{i+1})^2\geq\frac{1}{k}\left(\sum_{i=n-k}^{n-1}(a_i-a_{i+1})\right)^2=\frac{1}{k}\left(1-a_{n-k}\right)^2>-\frac{2}{k}a_{n-k},$$
from which it follows that $|\xi|^2-1> 0$, and we finish the proof.
\iffalse
Let $p\in \mathbb{Z}$ such that $a_1=\cdots=a_{p-1}=0$ and $a_p\not=0$.

If $p>n-k$, we get \[\|\xi\|^2-1=\sum_{i=p}^{n-1}\frac{1}{2}(a_i-a_{i+1})^2+\frac{1}{2}\left(a_p^2-1\right)\geq 0.\]
The equality holds if and only if %It takes ``$=$" if and only if
$a_p=\cdots=a_n=1$, which exactly corresponds to a certain row of $Y_{n,k}$.

If $p\leq n-k$, then we assume that the numbers change $l$ times from $a_{n-k}$ to $a_n$ by $p_j$ $(1\leq j\leq l)$ for each time. So $l\leq k$ and $1-a_{n-k}=p_1+\cdots +p_l$. As a result,
\begin{align*}
    \|\xi\|^2-1&=\sum_{i=p}^{n-k-1}\frac{1}{2}(a_i-a_{i+1})^2+\sum_{i=n-k}^{n-1}\frac{1}{2}(a_i-a_{i+1})^2+\frac{1}{2}\left(a_p^2-1\right)+\frac{1}{k}a_{n-k}\\
    &\geq \sum_{i=n-k}^{n-1}\frac{1}{2}(a_i-a_{i+1})^2+\frac{1}{k}a_{n-k}=\sum_{i=1}^{l}\frac{1}{2}p_i^2+\frac{1}{k}a_{n-k}\\
    &\geq \frac{(1-a_{n-k})^2}{2l}+\frac{1}{k}a_{n-k}=\frac{1+a_{n-k}^2}{2l}+\frac{l-k}{kl}a_{n-k}.
\end{align*}
The last two expressions show that no matter $a_{n-k}>0$, $=0$ or $<0$ there is always $\|\xi\|^2-1>0$ and we finish the proof.
\end{proof}
\fi
\end{proof}
\begin{remark}
The $\lambda_1$-minimal flat torus in this section has volume $\frac{\sqrt{k}(2\sqrt 2)^n}{\sqrt{k+1}(\sqrt n)^{n+1}}\pi^n$ $(1\leq k\leq n)$. It comes from the fact that $\det Q_{n,k}=\frac{n(k+1)}{2^nk}$ which again indicates $Q_{n,k}$ is positive-definite. One can easily check that Example~\ref{ex:3-2}, \ref{ex:8-1} and \ref{ex:9-1} exactly correspond to the case $(n,k)=(3,2)$, $(4,2)$ and $(4,3)$ respectively.

%when $n=3$, we exactly arrive at Example~\ref{ex:3-2} for $k=2$; when $n=4$, we exactly arrive at Example~\ref{ex:8-1} (w.r.t. Example~\ref{ex:9-1}) for $k=2$ (w.r.t. $k=3$).  
\end{remark}

\section{Berger's problem on conformally flat $3$-tori and $4$-tori}\label{sec-Berge}
As recalled in the introduction, on $n$-tori ($n\geq 3$), there is no solution to the Berger's problem, i.e., one can not expect a uniform upper bound for $\mathcal{L}(g)$ among all smooth Riemannian metrics. However, if %the metric is 
restricted to the flat metric, or a certain class of some conformally flat metrics, we can solve the Berger's problem for $n\leq 4$. That is, we will prove Theorem~\ref{thm1}  and Theorem~\ref{thm2} given in the introduction.

\iffalse
show among all flat $3$-tori,  $$\lambda_1(g)V(g)^{\frac{2}{n}}\leq 4 \sqrt[3]{2} \pi ^2,$$
and the equality is attained by the $\lambda_1$-minimal flat $3$-torus given in Example~\ref{ex:3-3}; while among all flat $4$-tori,  
$$\lambda_1(g)V(g)^{\frac{2}{n}}\leq 4 \sqrt{2} \pi ^2,$$
and the equality is attained by those $\lambda_1$-minimal flat $4$-tori given in Example~\ref{ex-4tori}. 
\fi

Note that finding the upper bound for $\lambda_1(g)V(g)^{\frac{2}{n}}$ among all flat $n$-tori, is equivalent to find the upper bound for volumes of all flat $n$-tori with $4\pi^2$ as the first eigenvalue, which can be done by calculating the minimum %positive lower bound for 
of the determinant of those lattices (of rank $n$) with $1$ as the shortest length. Since a shortest lattice vector %of shortest length 
can always be extended to be the first vector of a generator, in terms of the Gram matrix, we only need to calculate the minimum  of $\det$ on the following set: 
$$\Omega_1\triangleq \big\{Q\in \Sigma_{+}\, \big{|}\, Q_{11}=1, vQv^t\geq 1~{\rm for~all}~v\in\mathbb{Z}^n\backslash \{0\}\big\},$$
where $Q_{ij}$ is the entry of $Q$. 
%Set 
%$$\Omega_1\triangleq \big\{Q\in \Sigma_{+}\, \big{|}\, 1~{\rm is~the~ shortest ~length~of~the ~lattice~determined~ by} ~Q\big\}.$$
% We only need to calculate the minimal value of $\det$ on $\Omega_1$. 
\iffalse
\begin{lemma}
$\Omega_1$ is convex and compact. 
\end{lemma}
\begin{proof}
The convexity of $\Omega_1$ is obvious. To prove the compactness, we only need to prove that $\Omega_1$ is both closed and bounded. 

Firstly, we prove that $\Omega_1$ is closed. Suppose $\{Q_n\}$ is a convergent sequence in $\Omega_1$, whose limit is denoted by $\widetilde{Q}$. We denote by $D_{nk}$ (resp. $D_{0k}$) the leading principal minor of order $k$ for $Q_n$ (resp. $\widetilde{Q}$). Note that for all $1\leq k\leq n$, the submatrix corresponding to $D_{nk}$ is also positive definite. It can determine a sublattice of rank $k$, which also takes $1$ as the shortest length. %by applying 
It follows from the Minkowski's 1-st theorem that 
$$\sqrt{D_{nk}}\geq\frac{V(\mathbb{S}^k)}{2^k}.$$ 
By continuity, we have  $\sqrt{D_{0k}}\geq\frac{V(\mathbb{S}^k)}{2^k}>0$ for all $1\leq k\leq n$, which implies $\widetilde{Q}\in \Sigma_{+}$. On the other hand, it is obvious that $\widetilde{Q}_{11}=1$, and $v \widetilde{Q} v^t\geq 1$ for all $v\in \mathbb{Z}^n\backslash \{0\}$. Therefore $\widetilde{Q}\in \Omega_1$. 

Next, we prove that $\Omega_1$ is bounded. ...............
\end{proof}
\fi
We still use the notation $\Xi$ to denote the set of shortest lattice vectors % of the shortest length
up to $\pm 1$ in a given lattice $\Lambda_n^*$.  
\begin{lemma}\label{lem-mini}
%Suppose $n\leq 4$, and $Q\in \Omega_1$ is a positive definite matrix. 
For the lattice $\Lambda_n^*$ determined by $Q\in \Omega_1$, if ${\rm rank}(\Xi)<n$, then $Q$ is not a minimal point of $\det$ on $\Omega_1$.
\end{lemma}
\begin{proof}
We assume $\Xi$ is contained in the $(n-1)$-dimensional sublattice $\Lambda_{n-1}$ with a generator $\{\xi_1,\cdots,\xi_{n-1}\}$ and $\alpha$ is a vector in $\Lambda_n^*\backslash \Lambda_{n-1}$ such that $\{\xi_1,\cdots,\xi_{n-1},\alpha\}$ is a generator of $\Lambda_{n}^*$. 

Set $\alpha=\alpha^\top+\alpha^\perp$, in which $\alpha^\top$ is the orthogonal projection of $\alpha$ onto $\mathrm{Span}\{\xi_1,\cdots,\xi_{n-1}\}$. Let ${\Lambda}^*_{n,t}$ be a new lattice generated by $\{\xi_1,\cdots,\xi_{n-1},\alpha-t\alpha^\perp\}$ for $0\leq t<1$, $\Pi_{k,t}$ be the affine hyperplane defined by $\sum_{i=1}^{n-1}c_i\xi_i+k(\alpha-t\alpha^\perp)$ $(c_i\in \mathbb{R})$ for any $k\in \mathbb{Z}$.  

Note that there exists an integer $K>0$ such that $\Pi_{k,0}\cap \overline{B(0,1)}=\emptyset$ if and only if $|k|>K$. %Obviously, 
By continuity, there is $0<t_0<1$ such that for any $0\leq t\leq t_0$, %we have 
$\Pi_{k,t}\cap \overline{B(0,1)}=\emptyset$ if and only if $|k|>K$. 

As for $0<|k|\leq K$, it is easy to see that for any vector of $\Lambda_{n,t}^*\cap\Pi_{k,t}$, its orthogonal projection onto $\mathrm{Span}\{\xi_1,\cdots,\xi_{n-1}\}$ does not depend on %is invariant to 
$t$, while $\Pi_{k,t}\cap \overline{B(0,1)}$ is a family of co-centered $(n-1)$-dimensional closed balls expanding as $t$ goes from $0$ to $t_0$. Given the fact that all the vectors of $\Lambda_{n,0}^*\cap\Pi_{k,0}$ are located outside of $\Pi_{k,0}\cap \overline{B(0,1)}$, there exists $t_k>0$ such that all the vectors of $\Lambda_{n,t}^*\cap\Pi_{k,t}$ are still located outside of $\Pi_{k,t}\cap \overline{B(0,1)}$ for any $0\leq t\leq t_k$. 

Set $\delta=\min\{t_0,\cdots,t_K\}$. % Hence, if $\delta=\min\{t_0,\cdots,t_K\}$ 
Then
the lattice $\Lambda^*_{n,\delta}$ admits the same shortest vectors as $\Lambda_n^*$ while the Gram matrix $Q_\delta$ satisfying $\det Q_\delta=(\det(\xi_1,\cdots,\xi_{n-1},\alpha-\delta\alpha^\perp))^2=(1-\delta)^2\det Q$.
\end{proof}
Due to Theorem \ref{thm-gen} and \eqref{eq-excep1}, for a lattice $\Lambda_n^*$ with ${\rm rank}(\Xi)=n\leq 4$,  we can always choose a generator of $\Lambda_n^*$ such that the Gram matrix  takes $1$ as its diagonal entries. %This follows from Theorem~\ref{thm-gen} and \eqref{eq-excep1}. % if $\Lambda_n^*$ is prime. If  
Set 
\begin{equation}\label{eq-vQvt}
\Omega_2\triangleq \{Q\in \Sigma_{+}\, \big{|}\, Q_{ii}=1\text{ for all $1\leq i\leq n$, and $vQv^t\geq 1$ for all $v\in\mathbb{Z}^n\backslash \{0\}$}\}.
\end{equation}
Then according to Lemma~\ref{lem-mini}, we have the minimum  of $\det$ on $\Omega_1$ can only be attained on $\Omega_2$.  
\begin{lemma}\label{lem-convex}
$\Omega_2$ is convex and compact. 
\end{lemma}
\begin{proof}
The convexity of $\Omega_2$ is obvious. To prove the compactness, we only need to prove that $\Omega_2$ is both bounded and closed. 

Suppose $Q\in \Omega_2$. It follows from $\langle Q,Q\rangle=\mathrm{tr}Q^2<(\mathrm{tr}Q)^2=n^2$ that $\Omega_2$ is bounded. 

Next, we prove that $\Omega_2$ is closed. Suppose $\{Q_n\}$ is a convergent sequence in $\Omega_2$, whose limit is denoted by $Q_0$. We denote by $D_{nk}$ (resp. $D_{0k}$) the leading principal minor of order $k$ for $Q_n$ (resp. $Q_0$). Note that for all $1\leq k\leq n$, the submatrix corresponding to $D_{nk}$ is also positive definite. It can determine a sublattice of rank $k$, which also takes $1$ as the shortest length. %by applying 
It follows from the Theorem~$13$ in \cite{Siegel} (a corollary of the Minkowski's first theorem) that 
$$\sqrt{D_{nk}}\geq\frac{V(\mathbb{S}^k)}{2^k},$$ 
where $V(\mathbb{S}^k)$ is the volume of the standard round $k$-sphere.  
By continuity, we have  $\sqrt{D_{0k}}\geq\frac{V(\mathbb{S}^k)}{2^k}>0$ for all $1\leq k\leq n$, which implies $Q_0\in \Sigma_{+}$. In the mean time, it is obvious that the diagonal entries of $Q_0$ are all $1$ and $v Q_0 v^t\geq 1$ for all $v\in \mathbb{Z}^n\backslash \{0\}$. Therefore $Q_0\in \Omega_2$, and $\Omega_2$ is closed.  
\end{proof}

\begin{lemma}
%$\Omega_2=\Omega_3$.
$\Omega_2$ is a %bounded 
convex polytope.
\end{lemma}
\begin{proof}
We denote by $\{0,\pm 1\}^n$ the set of integer vectors whose coordinates take values only in $\{0,\pm 1\}$. 
%Moreover, 
Define 
\beq \label{eq-qvert}
\begin{split}
\Omega_3\triangleq%\label{eq-Omega22}
\bigg\{Q\in S(n)\, \bigg{|}\, Q_{ii}=1 \text{ for all $1\leq i\leq n$,}\hspace{3.3cm}\\
\text{and $vQv^t\geq 1$ for all nonzero $v\in\{ 0,\pm1\}^n$}\bigg\}.
\end{split}
\eeq
%\begin{align*}
%$$\Omega_3\triangleq\{Q\in S(n)\, \big{|}\, Q_{ii}=1\text{ for all $1\leq i\leq n$, and $vQv^t\geq 1$  for all nonzero $v$ in } \{0,\pm 1\}^n}$$
%\\&\text{ for all nonzero $v$ with elements taking values in $\{0,\pm 1\}$}\},
%\end{align*}
It follows from Theorem~\ref{thm:special} that %we will have 
$\Omega_2=\Omega_3\cap \Sigma_+$. 

We claim that $\Omega_2=\Omega_3$. %according to Theorem~\ref{thm:special}. 
For this, let us consider the topology on $\Omega_3$ induced from the ambient space $S(n)$. It is easy to show that $\Omega_3$ is convex, which implies $\Omega_3$ is connected. 
Seeing $\Omega_2$ as a subset in $\Omega_3$, we can firstly derive that it is closed by Lemma~\ref{lem-convex}. %$\Omega_2$ is a closed set in $S(n)$, it is also closed as a subset in $\Omega_3$. 
On the other hand, note that $\Sigma_{+}$ is open in $S(n)$, it follows from $\Omega_2=\Omega_3\cap \Sigma_+$ that $\Omega_2$ is also open in $\Omega_3$. Then the claim follows from the fact that $\Omega_2\neq \emptyset$.  

It is not hard to see that the constraints $vQv^t\geq 1$ for all nonzero $v\in\{ 0,\pm1\}^n$ define a %domain 
polyhedron in $S(n)$ by the intersection of finite half-spaces, and $\Omega_3$ is one of its facets. %Therefore, 
Combining this and Lemma~\ref{lem-convex}, we derive that  $\Omega_2=\Omega_3$ is a %bounded 
convex polytope. 
\iffalse
    Suppose $Q\in \Omega_3$. Let $y\in \mathbb{Q}^n$ be a rational vector, i.e., all the coordinates of $y$ are rational numbers.  Since $y$ is proportional to an integer vector in  $\mathbb{Z}^n$, we have $yQy^t> 0$. Then it follows from the density of $\mathbb{Q}^n$ in $\mathbb{R}^n$ that $Q$ is semi-positive definite which means $\Omega_3\subset \Sigma$. $I_n$ obviously lies in $\mathring\Omega_3$, the interior of $\Omega_3$. Therefore, $\Omega_3$ is not a polytope in $\partial\Sigma$. It means the affine subspace containing $\Omega_3$ is transversal to $\partial\Sigma$ and then $\mathring\Omega_3\subset\mathring\Sigma=\Sigma_+$. In the mean time, $\Omega_2=(\mathring\Omega_3\cap\Sigma_+)\cup(\partial\Omega_3\cap\Sigma_+)$. So $\mathring\Omega_3=\mathring\Omega_3\cap\Sigma_+\subset\Omega_2$. Since $\Omega_2$ is closed, we get $\Omega_2\subset\Omega_3=\overline{\mathring\Omega_3}\subset\overline{\Omega_2}=\Omega_2$.
    \fi
\end{proof}
According to the concavity of $\ln\circ\det$ on $\Sigma_{+}$ introduced in Lemma~\ref{lem-det}, one can easily obtain the next conclusion. 
\begin{lemma}\label{lem-minimal}
The minimum  of $\det$ on $\Omega_2$ %must be 
is attained at some vertex of $\Omega_2$. 
\end{lemma}

\begin{lemma}\label{rk-sharpxi}
Suppose $Q$ %\in \mathcal{V}(\Omega_2)$ 
is a vertex of $\Omega_2$, then 
%and $\sharp(\Xi)\geq \frac{n(n+1)}{2}$ for $Q$.  
%Note that in the space $S(n)$ %terms of 
%of symmetric matrix, % of $n\times n$, 
%the vertex of $\Omega_2$ is described as a solution of $\frac{n(n+1)}{2}$ linearly independent linear equations, as well as some other linear inequalities. Therefore for any vertex of $\Omega_2$, we have 
 $\sharp(\Xi)\geq \frac{n(n+1)}{2}$. 
\end{lemma}
\begin{proof}
Since $Q$ is a vertex and $\dim{S(n)}=\frac{n(n+1)}{2}$, it satisfies the constraints given in \eqref{eq-qvert}, %for all $v\in\{0,\pm1\}^n\backslash \{0\}$ 
%$$vQv^t\geq 1,$$
among which there should be a system of linear equations with rank $\frac{n(n+1)}{2}$. As a result, there should be at least $\frac{n(n+1)}{2}$ integer vectors such that $vQv^t= 1$, from which the conclusion follows. %we get  $\sharp(\Xi)\geq \frac{n(n+1)}{2}$. 
\end{proof}
%\end{proof}
\begin{proposition}\label{prop-vertex}
Suppose $n\leq 4$, then every vertex of $\Omega_2$ can determine a $\lambda_1$-minimal flat $n$-torus in some sphere. Up to congruence, % and reparameterization, 
these $\lambda_1$-minimal flat $n$-tori are \\
\noindent(1) the equilaterial $2$-torus in $\mathbb{S}^5$ given in Example~\ref{ex:3-prod} for  $n=2$; \\
\noindent(2) the $\lambda_1$-minimal flat $3$-torus in $\mathbb{S}^{11}$ given in Example~\ref{ex:3-3} for  $n=3$; \\
\noindent(3) the $\lambda_1$-minimal flat $4$-torus in $\mathbb{S}^{19}$ 
given in Example~\ref{ex:10}, or those $\lambda_1$-minimal flat $4$-tori given in Example~\ref{ex-4tori} for  $n=4$.
\end{proposition}
\begin{proof}
Firstly, for these examples, using their matrix data, one can check directly that the Gram matrix $Q$ belongs to $\Omega_2$, and the rank of 
$\{A^tA | A\in Y\}$ is exactly $\frac{n(n+1)}{2}$. Therefore they all correspond to the vertices of $\Omega_2$. 

Suppose $Q$ is a vertex of $\Omega_2$. If the lattice determined by $Q$ is not prime, then we arrive at the only exceptional torus, and the conclusion follows from Proposition~\ref{prop:except}. 

Next, we assume the lattice determined by $Q$ is prime. Let $X$ be the set of integer vectors corresponding to $\Xi$. Combining Remark~\ref{rk-maxxi} and Lemma~\ref{rk-sharpxi}, %we have $\sharp(\Xi)\geq \frac{n(n+1)}{2}$. %for which  
%we only need to prove the first conclusion. We still denote by $Y$ the set of integer vectors corresponding to $\Xi$. 
%From the discussion in Section~\ref{sec-4}, 
we can derive that $\sharp(\Xi)=\frac{n(n+1)}{2}$. Furthermore, it follows from the discussion in Section~\ref{sec-4} that $X$ is exactly the ladder set $X_n$ up to a unimodular transformation in $SL(n,\mathbb{Z})$. Note that $W_{X_n}$ is a singleton set. So $Q$ is congruent to $Q_n$ by a unimodular transformation in $SL(n,\mathbb{Z})$, which completes the proof of this proposition.  
\end{proof}
Combining Lemma~\ref{lem-minimal} with Proposition~\ref{prop-vertex}, we can directly calculate the minimum of $\det$ on $\Omega_2$ (hence on $\Omega_1$), from which Theorem \ref{thm1} follows. 

For a given flat metric $g_0$ on $n$-torus $T^n$, it was proved by El Soufi and Ilias in \cite{Sou-Ili3} that $g_0$ maximizes $\mathcal{L}(g)$ in $[g_0]$, if the first eigenspace of $g_0$ is of dimension no less than $2n$. Combining this with Theorem~\ref{thm1}, we can obtain Theorem \ref{thm2}. 

\textbf{Acknowledgement:}
The first author and the third author is supported by NSFC No. 12171473. The second author is supported by NSFC No. 11971107. The authors are grateful to Prof. C.P. Wang for his continuous encouragement on this work. 
The authors are thankful to Prof. Q.-S. Chi and Prof. R. Kusner for valuable discussions.
%PW was partly supported by the Project 11971107 of NSFC. ZXX is supported by Yueqi Scholars of CUMTB.

%\bibliographystyle{plain}
%\bibliography{ref}

\begin{thebibliography}{99}
\bibitem{Berger}  M. Berger, {\em Sur les $premi\grave{e}res$ valeurs propres des vari\'et\'es riemanniennes}, Compositio Math. 26 (1973), 129-149.
\bibitem{Bryant} R.L. Bryant, {\em Minimal surfaces of constant curvature in $S^n$}, Trans. Amer. Math. Soc. 290 (1985), 259-271. 
%\bibitem{Bryant1} R.L. Bryant, {\em Surfaces in conformal geoemtry}, Proc. Sympos. Pure Math. 48 (1988), 227-240. 
%\bibitem{Bryant2} R.L. Bryant, {\em On the conformal volume of $2$-tori}, arXiv: 1507.01485. 
%\bibitem{Clabi}

\bibitem {CH} Choe, J.,  Hoppe, J. {\em 	  Some minimal submanifolds generalizing the Clifford torus},
    Math. Nachr. 291 (2018), %no. 17-18, 
    2536-2542.
\bibitem{Colbois-Dodziuk} B. Colbois, J. Dodziuk, {\em Riemannian metrics with large $\lambda_1$}, Proc. Amer. Math. Soc. 122 (1994),  905-906. 

\bibitem{doCarmo} M.P. Do Carmo, F.J. Flaherty,  {\em Riemannian geometry}, Boston: Birkh\"auser, 1992.
\bibitem{doCarmo-Wallach} M.P. Do Carmo, N.R. Wallach, {\em Minimal immersions of spheres into spheres}, Ann.
Math. 93 (1971), 43-62.
\bibitem{Soufi2} A. El Soufi, H. Giacomini, M.  Jazar, {\em A unique extremal metric for the least eigenvalue of the Laplacian on the Klein bottle}, Duke Math. J. 135 (2006), 181-202.

\bibitem{Soufi-Ilias} A. El Soufi, S. Ilias, {\em Immersions minimales, $premi\grave{e}re$ valeur propre du laplacien et volume conforme}, Math. Ann. 275 (1986), 257-267. 

\bibitem{Sou-Ili} A. El Soufi, S. Ilias, {\em Riemannian manifolds admitting isometric immersions by their first eigenfunctions}, Pacific J. Math. 195 (2000), 91-99.

\bibitem{Sou-Ili3} A. El Soufi, S. Ilias, {\em Extremal metrics for the first eigenvalue of the Laplacian in a conformal class}, Proc. Amer. Math. Soc. 131  (2003), 1611-1618.

\bibitem{Gruber} P.M. Gruber, {\em Geometry of numbers, Handbook of convex geometry}. North-Holland, 1993: 739-763.
\bibitem{Hersch} J. Hersch, {\em Quatre propri\'et\'es isop\'erim\'etriques de membranes sph\'eriques $homog\grave{e}nes$}, C. R. Acad. Sci. Paris S\'er. 270 (1970),
1645-1648. 
\bibitem{Hirsch-Mader} J. Hirsch, E. M\"ader-Baumdicker, {\em A note on Willmore minimizing Klein bottles in Euclidean space}, Adv. Math. 319 (2017), 67-75.  

\bibitem{Nadi2} D. Jakobson, N. Nadirashvili, I. Polterovich, {\em Extremal metric for the first eigenvalue on a Klein bottle}, Canad. J. Math. 58 (2006), 381-400. 

\bibitem{Karpukhin} M. Karpukhin, {\em On the Yang-Yau inequality for the first Laplace eigenvalue}, Geom. Funct. Anal. 29 (2019), 1864-1885. 

\bibitem{Karpukhin-Stern} M. Karpukhin, D. Stern, {\em Existence of harmonic maps and eigenvalue optimization in higher dimensions}, arXiv preprint arXiv:\,2207.13635 (2022). 

\bibitem{Kenmotsu} K. Kenmotsu, {\em On minimal immersions of $\mathbb{R}^2$ into $\mathbb{S}^n$}, J. Math. Soc. Japan 28 (1976),
182-191.
\bibitem{Li-Yau} P. Li, S.T. Yau, {\em A new conformal invariant and its applications to the Willmore conjecture and the first eigenvalue of compact
surfaces}, Invent. Math. 69 (1982), 269-291.

\bibitem{Matthiesen-Siffert} H. Matthiesen, A. Siffert, {\em Handle attachment and the normalized first eigenvalue}, arXiv preprint arXiv:\,1909.03105 (2019).

\bibitem{Muto-Onita-Urakawa} H. Muto, Y. Ohnita, H. Urakawa, {\em Homogeneous minimal hypersurfaces in the unit spheres and the first eigenvalues of their Laplacian}, Tohoku Math. J. 36 (1984), 253-267.

%\bibitem{John48}F.~John, {\em Extremum problems with inequalities as subsidiary conditions}. Studies and Essays Presented to R. Courant on his 60th Birthday, January 8, 1948, 187-204. Interscience Publishers, Inc., New York, N. Y., 1948.

\bibitem{Mon-Ros} S. Montiel,  A. Ros, {\em Minimal immersions of surfaces by the first eigenfunctions and conformal area,} Invent. Math. 83 (1986), 153-166.
\bibitem{Nadi} N. Nadirashvili, {\em Berger's isoperimetric problem and minimal immersions of surfaces}, Geom. Funct. Anal. 6 (1996), 877-897.
\bibitem{Park-Urakawa} J.S. Park, H. Urakawa, {\em Classification of harmonic mappings of constant energy density into spheres}, Geom. Dedicata 37 (1991), 211-226.
\bibitem{Petrides} R. Petrides, {\em Existence and regularity of maximal metrics for the first Laplace eigenvalue on surfaces}, Geom. Funct. Anal. 24 (2014), 1336-1376.
%\bibitem{Sidi-Bro}
\bibitem{Petrides2} R. Petrides, {\em Maximizing one Laplace eigenvalue on $n$-dimensional manifolds}, arXiv preprint arXiv:\,2211.15636 (2022).
\bibitem{Ryshkov73} S.S. Ryshkov, {\em On the theory of Hermite-Minkowski reduction of positive quadratic forms}, Zap. Nauchn. Sem. Leningr. Otd. Mat. Inst. 33 (1973), 37-64. 

\bibitem{Siegel} C.L. Siegel,
{\em Lectures on the geometry of numbers},
Springer Science \& Business Media, 2013.
\bibitem{Taka} T. Takahashi, {\em Minimal immersions of Riemannian manifolds}, J. Math. Soc. Japan 18 (1966), 380-385.
\bibitem{Tang-Yan} Z.Z. Tang, W.J. Yan,  {\em Isoparametric foliation and Yau conjecture on the first eigenvalue}, J. Differ. Geom. 94  (2013), 521-540.
\bibitem{TZ}  Z.Z. Tang, Y.S. Zhang, {\em  Minimizing cones associated with isoparametric foliations.} J. Differential Geom. 115 (2020), %no. 2, 367-393.

\bibitem{2Wang} C.P. Wang, P. Wang,
{\em The Morse index of minimal products of minimal submanifolds in spheres}, Sci. China Math. (2022), 1-20. %https://doi.org/10.1007/s11425-021-1963-3
%\bibitem{Wang}
%C.P. Wang, {\em M\"{o}bius geometry of submanifolds in $\mathbb{S}^{n}$,} Manuscripta Math. 96, 517-534 (1998).
\bibitem{Wolpert} S. Wolpert, {\em The eigenvalue spectrum as moduli for flat tori,} Trans. Amer. Math. Soc, 244 (1978), 313-321.
\end{thebibliography}

\vspace{5mm} \noindent Ying L\"u\\
{\small\it  School of Mathematical Sciences, Xiamen University, Xiamen, 361005, P. R. China.\\
Email: {lueying@xmu.edu.cn}}

\vspace{5mm} \noindent Peng Wang\\
{\small\it  School of Mathematics and Statistics, FJKLMAA, Fujian Normal University, Fuzhou 350117, P. R. China.\\
Email: {pengwang@fjnu.edu.cn}}

\vspace{5mm} \noindent Zhenxiao Xie\\
{\small\it Department of Mathematics, China University of Mining and Technology (Beijing),
Beijing 100083, P. R. China.
Email: {xiezhenxiao@cumtb.edu.cn}}

\end{document}